\documentclass[11pt,reqno]{amsart}

\usepackage[shortlabels]{enumitem}

\usepackage{color}
\usepackage{amsmath, amsthm, amssymb}
\usepackage{amsfonts}

\usepackage[ansinew]{inputenc}
\usepackage[dvips]{epsfig}
\usepackage{graphicx}
\usepackage[english]{babel}
\usepackage{hyperref}
\theoremstyle{plain}
\newtheorem{theorem}{Theorem}[section]
\newtheorem{corollary}[theorem]{Corollary}
\newtheorem{lemma}[theorem]{Lemma}
\newtheorem{proposition}[theorem]{Proposition}

\theoremstyle{definition}
\newtheorem{definition}[theorem]{Definition}

\theoremstyle{remark}
\newtheorem{remark}[theorem]{Remark}

\numberwithin{equation}{section}

\newcommand{\average}{{\mathchoice {\kern1ex\vcenter{\hrule height.4pt
				width 6pt depth0pt} \kern-9.7pt} {\kern1ex\vcenter{\hrule
				height.4pt width 4.3pt depth0pt} \kern-7pt} {} {} }}

\def\R{\mathbb{R}}
\def\S{\mathbb{S}}

\def\Q{\mathbb{Q}}
\def\N{\mathbb{N}}

\newcommand{\loc}{{\mathrm{loc}}}

\begin{document}

	\title[The fractional obstacle problem with drift: higher regularity]{The fractional obstacle problem with drift:\\ higher regularity of free boundaries}
	
	\author{Teo Kukuljan}
	\address{Universit\"at Z\"urich, Institut f\"ur Mathematik, Winterthurerstrasse, 8057 Z\"urich, 
		Switzerland}
	\email{teo.kukuljan@math.uzh.ch}
	
	\thanks{The author has received funding from the European Research Council (ERC) under the Grant Agreement No 801867, and from the Swiss National Science Foundation project 200021\_178795. 
	The paper was written under supervision of Xavier Ros-Oton, we thank him for proposing many ideas and approaches.}
	
	\begin{abstract}
		We study the higher regularity of free boundaries in obstacle problems for integro-differential operators with drift, like $(-\Delta)^s +b\cdot\nabla$, in the subcritical regime $s>\frac{1}{2}$. Our main result states that once the free boundary is $C^1$ then it is $C^\infty$, whenever $s\not\in\Q$.
		
		In order to achieve this, we establish a fine boundary expansion for solutions to linear nonlocal equations with drift in terms of the powers of distance function. Quite interestingly, due to the drift term, the powers do not increase by natural numbers and the fact that $s$ is irrational plays al important role. Such expansion still allows us to prove a higher order boundary Harnack inequality, where the regularity holds in the tangential directions only.
	\end{abstract}

	\maketitle

	\section{Introduction}

	Obstacle problems for integro-differential operators naturally appear in probability and mathematical finance, for example in the optimal stopping problem for L\'evy processes with jumps, which has been used in pricing models for American options since 1970, see \cite{CT04}. More recently, obstacle problems of this kind  appeared also in other fields of science, for example biology and material science, see \cite{CDM16,R18,S15} and references therein.
	
	Henceforth,  more and more effort has been put into understanding obstacle problems for nonlocal operators. The obstacle problem is a non-linear equation that can be written in the form
	$$\min\{ Lu,u-\varphi\} = 0\quad \quad \text{ in }\quad\R^n,$$
	where $\varphi$ is a given smooth  function with compact support, called the obstacle, and $L$ is some kind of nonlocal operator. The main goal of study is to understand the set $\{u>\varphi\}$, concretely to find out the regularity of its boundary $\partial\{u>\varphi\},$ called the free boundary.  The most basic and canonical example of a nonlocal operator is the fractional Laplacian, $(-\Delta)^s$, given by
	$$  (-\Delta)^s u (x) = c_{n,s} p.v. \int_{\R^n} \frac{u(x)-u(y)}{|x-y|^{n+2s}}dy,\quad s\in(0,1),$$
	where the constant $c_{n,s}$ is chosen so that the Fourier symbol of the operator is $|\xi|^{2s}.$ 
	The main result regarding the regularity of the free boundary in the case of the fractional Laplacian states that the free boundary is $C^\infty$ outside of a set of singular or degenerate points. Results of this type are usually established in three steps:
	\begin{enumerate}[(a)] 
		\item The free boundary splits into regular and singular/degenerate points,
		
		\item Near regular points the free boundary is $C^{1,\alpha}$ for some $\alpha>0$,
		
		\item If the free boundary is $C^{1,\alpha}$, then it is $C^\infty$.
	\end{enumerate}
	Parts (a) and (b) were established in \cite{ACS08,CSS08}, and part (c) in \cite{DS16,JN17,KPS15,KR19}. 
	Analogous results have been established for a family of integro-differential operators of the form
	\begin{align*}
		Lu(x) &= p.v.\int_{\R^n} (u(x)-u(x+y))K(y)dy\\
		&=\frac{1}{2}\int_{\R^n} (2u(x)-u(x+y)-u(x-y))K(y)dy,
	\end{align*}
	where the kernel $K$ satisfies
	\begin{align}\label{kernelConditions}
	\begin{split}
	&K \text{ is even, homogeneous and}\\
	&\frac{\lambda}{|y|^{n+2s}}\leq K(y)\leq \frac{\Lambda}{|y|^{n+2s}}, \text{ for all }y\in\R^n \text{ with } 0<\lambda\leq \Lambda,\text{ }s\in(0,1).
	\end{split}
	\end{align}
	Steps (a) and (b) has been established in \cite{CRS17}, and step (c) in \cite{AR20}. 
	
	Two more cases of operators are of particular interest for study. The parabolic case, when the operator is given with $(-\Delta)^s + \partial_t$, and the drift case or the fractional Laplacian with drift, when the operator equals $(-\Delta)^s + b \cdot\nabla$. In order to be able to apply similar tools as in the case $b\equiv 0$, we must assume that the parameter $s>\frac{1}{2}$ -- the subcritical regime. Then the fractional Laplacian is of higher order than the derivative terms, which allows us to treat the drift term as reminders. Still, much less is known in these cases. In the parabolic case, step (a) and (b) were established in \cite{BFR18,CF13}, and there are no generalisations to a wider class of elliptic operators, and (c) is an open problem. Similarly, for the fractional Laplacian with drift, steps (a) and (b) have been established in \cite{GPPS17,PP15}, but again (c) remained an open problem. On the other hand, the fractional Laplacian with drift in the critical regime $s=\frac{1}{2}$ is studied in \cite{FR18}, where again steps (a) and (b) are established. 
	
	The aim of this paper is to continue the study of the free boundary around regular points for the obstacle problem with  drift in the subcritical regime $s>\frac{1}{2}$:
	\begin{align}\label{obstacleProblem}
	\min\{Lu+b\cdot\nabla u, u-\varphi\} = 0\quad \quad\text{in }\quad\R^n,
	\end{align}
	where $b$ is a constant vector in $\R^n$.
	We show that if $s\not\in\Q$, once the free boundary is $C^{1}$, it is in fact $C^\infty,$ as long as $\varphi$ is $C^\infty$. Our main result states the following.
	
	\begin{theorem}\label{1.1}
		Let $L$ be the fractional Laplacian. Let $s\in (\frac{1}{2},1)\backslash\mathbb{Q}$, and $\varphi$ be any $C^\infty_c(\R^n)$ obstacle, and $u$ be any solution to \eqref{obstacleProblem}. Then the free boundary $\partial\{u>\varphi\}$ is $C^\infty$ in a neighbourhood of any regular point.\footnote{We refer to \cite{GPPS17} or Section 7 below, for the definition of regular points.}
	\end{theorem}
	
	This is the analogue of the step (c) explained above, since the initial regularity is provided in \cite{GPPS17}. We  explain the exclusion of the rational parameters $s$ in the subsection below. Let us also point out, that our proof from the initial $C^{1}$ regularity to $C^\infty$ is done for a family of operators $L$ whose kernels satisfy \eqref{kernelConditions} and are $C^\infty(\mathbb{S}^{n-1}).$ Hence as soon as the initial regularity is provided for this class as well, we immediately obtain the analogue of the above theorem (see Corollary \ref{infiniteRegularityCorollary}).
	
	\subsection{Strategy of the proof}
	
	In order to obtain the higher regularity of the free boundary, we exploit the fact that the normal vector can be expressed with the quotients of the partial derivatives of the height function $w:=u-\varphi$ (see \cite{DS15,DS16} or  \cite[Section 5]{AR20}). Hence, we closely study the quotients $w_i/w_n$ in the domain $\Omega = \{w>0\}$. 
	
	Some ideas are drawn from \cite{AR20}, but we are faced with several difficulties arising from the drift term. 
	
	In \cite{AR20}, it is established that the quotients $w_i/w_n$ are as smooth as the boundary, say $C^\beta(\overline{\Omega})$. This implies that the normal vector $\nu$ to the boundary is $C^\beta$ as well, and hence the boundary itself is $C^{\beta+1}.$ A key step to show this is to establish that the quotients $w_i/d^s$ are $C^{\beta-1}(\overline{\Omega})$. Here, unfortunately, these results fail due to the presence of the drift term, and the best regularities we can get are $w_i/d^s,w_i/w_n\in C^{2s-1}(\overline{\Omega}).$
	Still, one may wonder if $w_i/w_n\in C^\beta(\partial \Omega)$, i.e., the regularity in \emph{tangential} directions only. This is indeed what we prove here, but it turns out to be quite delicate, as explained next. 
	
	We use the expansion result from \cite{AR20}, stating that
	\begin{equation}\label{expansionResultFromAR20}
		\left\lbrace\begin{array}{rcll}
		Lu &= &f&\text{in }\Omega\cap B_1,\\
		u&=&0&\text{in }\Omega^c\cap B_1,
		\end{array}\right. \Longrightarrow \quad u(x) = Q(x) d^s + O(|x|^{\beta-1+s}),
	\end{equation}
	for some polynomial $Q$ of degree $\lfloor\beta-1\rfloor$, provided that $f$ is smooth enough. For simplicity, let us turn to the one-dimensional case, $\Omega = \{x>0\}\subset\R$. When there is no drift term, \eqref{expansionResultFromAR20} applies directly on $w_i$ and gives the expansion of the form 
	$$w_i(x) = c_0 x_+^s + c_1 x_+^{1+s} + c_2 x_+^{2+s}+\ldots + O(|x|^{\beta-1+s}), $$
	which yields that $w_i/d^s$ agrees with a polynomial up to an error term of order $\beta-1$. This provides enough information to deduce the wanted regularity. When the drift term appears, the partial derivatives solve 
	$$\left\lbrace\begin{array}{rcll}
	Lw_i &=& f_i -b\cdot \nabla w_i &\text{in }\Omega\cap B_1,\\
	w_i&=&0&\text{in }\Omega^c\cap B_1.
	\end{array}\right. $$
	Since a priori we have very little regularity for $\nabla w_i$, the expansion result \eqref{expansionResultFromAR20} gives us only that $w_i = c_0 x_+^s + O(|x|^{s+\alpha})$, for some small $\alpha$. To continue the expansion, we deduce that than the gradient is of the form $\nabla w_i = c_0' x_+^{s-1} + O(|x|^{s-1+\alpha}),$ and then find a suitable constant $c_1$, such that $L(c_1 x_+^{s+\varepsilon_0}) = c_0' x_+^{s-1}$. Then we expand the function $w_i - c_1x_+^{s+2s-1}$, which has a bit better right-hand side, to get
	$$ w_i = c_0 x_+^s +c_1x_+^{s+2s-1} + O(|x|^{s+2s-1+\alpha}).$$ We proceed in the similar manner, to get the expansion 
	$$ w_i = c_0 x_+^s +c_1x_+^{s+2s-1} + c_2x_+^{s+2(2s-1)}+\ldots + O(|x|^{\beta-1+s}),$$
	where the powers in the expansion are of the form $s+k(2s-1) + l$, for $k,l\in\N$, but smaller than $\beta-1+s$. The procedure is based on the fact that $Lx_+^p = c_px^{p-2s}$ for some non-zero constant $c_p$. This equality fails when $p$ is of the form $m+s$ or $m+2s$ for any integer $m$. This leads to exclusion of rational parameters $s$.
	
	In fact, similar happens in the general case. If $\partial\Omega\subset \R^n$ is $C^{\beta}$, we are able to establish the expansion of the form
	$$w_i (x)= \sum_{k,l\geq 0}^{k(2s-1)+l\leq \beta-1}Q_{k,l,z}(x) d^{s+k(2s-1)+l}(x) + O(|x-z|^{\beta-1+s}),$$
	around every regular free boundary point $z$, where the zero order term of $Q_{0,0,z}$ is $C^{\beta-1}(\partial\Omega)$ as a function of $z$.
 	The expansion tells us, that the best regularity of $\frac{w_i}{d^s}$ and $\frac{w_i}{w_n}$ is $C^{2s-1}(\overline{\Omega})$, but moreover also that $  \frac{w_i}{d^s} \in C^{\beta-1}(\partial\Omega)$. Furthermore, with some extra steps we are able to deduce also $\frac{w_i}{w_n} \in C^{\beta-1+(2s-1)}(\partial\Omega)$, which provides that the normal $\nu$ to the boundary is of the same regularity, and so the boundary has to be $C^{\beta+(2s-1)}$. Since $2s-1>0$, this is enough to bootstrap and deduce that $\partial\Omega$ is actually $C^\infty.$
	
	Let us describe the expansion part a little bit more in details. 
	We improve the accuracy of the expansion of $w_i$ gradually. The improvement is obtained in two steps. First we show how the expansion of $w_i$ translates to the gradient. This presents a rather cumbersome step, which needs some additional interior regularity estimates and Corollary \ref{generalisedGrowthLemma}. 
	
	Then we need to correct the expansion of $w_i$, in such a way that its operator evaluation becomes small. Concretely, for every term $Qd^{\tilde{p}}$ in the drift term expansion, we need to find polynomial $\tilde{Q}$ and a suitable power $p$ so that $L(\tilde{Q} d^p)\approx Qd^{\tilde{p}}$. In order to establish it, one needs to show a result of the type
	\begin{equation}\label{theGratAssumption} 	
		\partial\Omega\in C^\beta\quad\Longrightarrow\quad L(\tilde{Q} d^p) = \phi d^{p-2s} + R,
	\end{equation}
	where $\phi\in C^\beta(\overline{\Omega})$ and $R\in C^{\beta-1+ p-2s}(\overline{\Omega})$, together with additional information about function $\phi$. To get a correspondence between $\tilde{Q}$ and $\phi$ we perform a blow-up argument with a limiting result, that reduces to the flat case. Together with the explicit computation in the flat case, we are able to argue the existence of $\tilde{Q}$, so that $L(\tilde{Q}d^{p+2s})\approx Q_i d^{p},$ up to some error terms. Note that in the computation in the flat case, we need that the power is not in $\N+2s$, otherwise logarithmic terms would appear. Since we apply it on the powers of the form $(2k+1)s - l$, our argument with expansions works only when $s$ is not of the form $\frac{m}{(2k+1)}$, for two integers $m,k$, and hence we set $s$ to be irrational.
	
	\subsection{Boundary regularity for linear nonlocal equations with drift}
	
	Our results hold for arbitrary solutions to linear equations of the form 
	\begin{equation}\label{linearEquation}
	\left\lbrace\begin{array}{rcl l}
	Lu+b\cdot\nabla u&=&f&\text{in  } \Omega\cap B_1\\
	u &=&0 & \text{in  } \Omega^c\cap B_1.
	\end{array}\right.
	\end{equation}
	With its aid we are able to establish the following boundary Harnack-type estimate.

	\begin{theorem}\label{Harnack}
		Let $s>\frac{1}{2},$ $s\not\in\Q$, and assume $\Omega\subset\R^n$ is a $C^\beta$ domain with $\beta> 2-s,$ $\beta\not\in\N$ and $\beta\pm s\not\in\N$. Let $L$ be an operator whose kernel $K$ is $C^{2\beta+1}(\S^{n-1})$ and satisfies conditions \eqref{kernelConditions}. For $i=1,2$, let $u_i$ be two solutions to
		$$\left\lbrace\begin{array}{rcl l}
		Lu_i+b\cdot\nabla u_i&=&f_i&\text{in  } \Omega\cap B_1\\
		u_i &=&0 & \text{in  } \Omega^c\cap B_1,
		\end{array}\right.$$
		with $f_i\in C^{\beta-2+s }(\overline{\Omega}\cap B_1)$ and $b\in\R^n.$ Assume that $u_2\geq c_2d^s$ in $B_1$, for some positive $c_2.$ Then
		\begin{equation}\label{optimalHarnack}
			\frac{u_1}{u_2}\in C^{2s-1}(\overline{\Omega}\cap B_{1/2}),
		\end{equation}
		and
		\begin{align*}
		\frac{u_1}{u_2}\in C^{\beta-1+(2s-1)}(\partial\Omega\cap B_{1/2}).
		\end{align*}
	\end{theorem}
	
	An important step towards towards the proof of Theorem \ref{Harnack} is the following boundary Schauder-type estimate for solutions to \eqref{linearEquation}.

	\begin{theorem}\label{Schauder}
		Let $s>\frac{1}{2},$ $s\not\in\Q$, and assume $\Omega\subset\R^n$ is a $C^\beta$ domain with $\beta>1+s$, $\beta\not\in\N,$ and $\beta\pm s\not\in\N$. Let $L$ be an operator whose kernel $K$ is $C^{2\beta+1}(\S^{n-1})$ and satisfies conditions \eqref{kernelConditions}. Let $u$ be a solution to
		\eqref{linearEquation}
		with $f\in C^{\beta-1-s}(\overline{\Omega})$ and $b\in\R^n.$ 
		
		Then
		\begin{equation}\label{optimalSchauder}
			\frac{u}{d^s} \in C^{2s-1}(\overline{\Omega}\cap B_{1/2}),
		\end{equation}
		and 
		$$\frac{u}{d^s}\in C^{\beta-1}(\partial\Omega\cap B_{1/2}).$$ 
	\end{theorem}

	We emphasise that \eqref{optimalHarnack} and \eqref{optimalSchauder} are optimal; and therefore the higher regularity of $u_1/u_2$ and $u/d^s$ holds only in the tangential directions.
	
	\subsection{Organization of the paper}
	
	In section 2 we present the notation and some tools we use throughout the paper. In Section 3, we prove \eqref{theGratAssumption} and some similar results regarding the evaluation of the operator $L$ on  powers of the distance function. Section 4 is devoted to the computation of the flat case and establishing the correspondence between polynomials described above. Furthermore, it provides a boundary regularity estimate needed in our framework. In Section 5, we prove expansion type results which find use in Section 6, where we prove 
	Theorem \ref{Schauder} and Theorem \ref{Harnack}. In Section 7 we prove Theorem \ref{1.1} and related results. At the end there is an appendix, where we prove technical auxiliary results, to lighten the body of the paper.

	\section{Notation and preliminaries}
	
	When $\beta\not\in\N$, with $C^\beta$ we mean the H\"older space $C^{\lfloor\beta\rfloor,\langle\beta\rangle}$, where $\lfloor\cdot\rfloor $ denotes the  integer part of a number and $\langle x\rangle = x - \lfloor x\rfloor$. Moreover, with  $C_0(\R^n)$ we denote the closure of continuous, compactly supported functions with respect to the $L^\infty$ norm.
	
	For $y\in\R^n$ we denote $\langle y\rangle = y/|y|.$ Also, we use the multi-index notation $\alpha\in\N^n$, $\alpha = (\alpha_1,\ldots,\alpha_n)$, $|\alpha|=\alpha_1+\ldots+\alpha_n$, and accordingly we denote
	$$\partial^\alpha = \left(\frac{\partial}{\partial x_1}\right)^{\alpha_1}\circ\ldots\circ \left(\frac{\partial}{\partial x_n}\right)^{\alpha_n}.$$ 
	
	With $\textbf{P}_k$ we denote the space of polynomials of order $k$. For a function $f$, we denote $T^k_a f$ its Taylor polynomial of order $k$ at a point $a$. 
	
	Throughout the paper $s$ is a parameter in $(0,1)$, often also $s>1/2$, and $\varepsilon_0 = 2s-1$. Furthermore for real numbers $a,b$ we denote $a\land b = \min\{a,b\}$ and $a\lor b=\max\{a,b\}.$ 
	
	The unit sphere is denoted with $\mathbb{S}^{n-1} = \{x\in\R^n : |x|=1\}.$ Sometimes the following comes in handy
	$$\textbf{1}_{condition} = \left\lbrace\begin{array}{cl}
	1&\text{if }condition\text{ holds}\\
	0&\text{otherwise.}
	\end{array}\right.$$

	Finally, $C$ indicates an unspecified constant not depending on any of the relevant quantities, and whose value is allowed to change from line to line. We make use of sub-indices whenever we will want to underline the  dependencies of the	constant.
	
	\subsection{Generalised distance function}
	Throughout the paper, for a domain $\Omega$ the distance function to its complement is of great use. 
	Since we need it to be more regular inside $\Omega$, we work with the generalised distance function defined as follows.
	
	\begin{definition}\label{definitionOfGeneralDistance}
		Let $\Omega\subset\R^n$ be an open set with $C^\beta$ boundary. We denote with $d\in C^\infty(\Omega)\cap C^{\beta}(\overline{\Omega})$ a function satisfying
		$$\frac{1}{C}\operatorname{dist} (\cdot, \Omega^c)\leq d\leq C\operatorname{dist} (\cdot, \Omega^c),$$
		$$|D^k d|\leq C_k d^{\beta-k},\quad\text{for all }k>\beta.$$
	\end{definition}

	The definition follows \cite[Definition 2.1]{AR20} and the precise construction is provided by \cite[Lemma A.2]{AR20}. 
	
	\subsection{Nonlocal equations for functions with polynomial growth at infinity}
	
	We are often in situation when the function on which we want to evaluate the operator $L$ does not satisfy the growth control. In our case it mostly occurs when doing the limiting arguments and blow-ups. Then we are faced with a function which has polynomial growth. The evaluation is done according the following definition.
	
	\begin{definition}\label{definitionGeneralised}
		Let $k\in\N$, $\Omega\subset\R^n$ be a bounded domain, $u\in L^1_{\operatorname{loc}}(\R^n)$, and $f\in L^\infty(\Omega).$ Assume that $u$ satisfies 
		$$\int_{\R^n} \frac{u(y)}{1+|y|^{n+k+2s}}dy <\infty.$$
		We say that
		\begin{equation}\label{definitionOfLInGeneralisedSense}
			Lu\stackrel{k}{=}f,\quad \text{in }\Omega,
		\end{equation}
		if there exists a family of polynomials $(p_R)_{R>0}\subset \textbf{P}_{k-1}$, and a family of functions $(f_R\colon \Omega\to\R)_{R>0}$, such that for all $R>\operatorname{diam}(\Omega)$ we have
		$$L(u\chi_{B_R}) = f_R+p_R\quad \text{in }\Omega,$$
		and 
		$$f_R\xrightarrow{R\to\infty}f\quad\text{uniformly in }\Omega.$$
		In the case when $\Omega$ is unbounded, we say that \eqref{definitionOfLInGeneralisedSense} holds, if it holds in any bounded subdomain.
	\end{definition}

	The definition follows the one in \cite[Section 3]{AR20}. There one can also find some properties of solutions to \eqref{definitionOfLInGeneralisedSense}.

	\section{Non-local operators and the distance function}
	
	The goal of this section is to prove the following result:
	
	\begin{theorem}\label{goalOfSection2}
		Let $K\in C^{2\beta+1}(\mathbb{S}^{n-1})$ be a kernel as in \eqref{kernelConditions}. Let $p\in(0,2s)$ and $\Omega$ be a domain in $\R^n$ with $C^\beta$ boundary, for some $\beta>1+2s-p.$  Let $\eta\in C^\infty(\R^n).$ Then we have
		$$L(\eta d^p) = \varphi d^{p-2s} + R,$$
		where $\varphi \in C^\beta(\overline{\Omega}),$ and $R\in C^{\beta-1-2s+p}(\overline{\Omega}),$ with
		$$||R||_{C^{\beta-1-2s+p}(\overline{\Omega})}\leq C,$$
		where $C$ depends only on $n,$ $s,$ $\eta,$ $||K||_{C^{2\beta-1}(\S^{n-1})},$ $ p,$ $\beta$ and the $C^\beta$ norm of $\partial\Omega$.
	\end{theorem}
	
	Let us start with the result which estimates $L(d^{s+\varepsilon})$ in a $C^{1,\alpha}$ domain, where $\varepsilon<s$ and $\alpha\leq1.$ It is done in a similar manner as \cite[Proposition 2.3]{RS16}. This result already gives all the regularity of $L(d^p)$ we need in this setting and we do not develop it further.
	
	\begin{lemma}\label{LonDistance}
		Let $\Omega$ be a $C^{1,\alpha}$ domain with $\alpha\leq1$ and $K$ a kernel satisfying \eqref{kernelConditions}. Then for $\varepsilon\in (-s,s)$ we have the following equality
		$$L(d^{s+\varepsilon})(x) = c_{s+\varepsilon} |\nabla d (x)|^{2s} d^{\varepsilon-s}(x) + R(x),\quad \quad x\in\Omega,$$
		where $c_{s+\varepsilon}$ is an explicit constant and $R$ satisfies $|R|\leq C d^{(\alpha+\varepsilon-s)\land 0}$. Moreover, $c_{s+\varepsilon}=0$ if and only if $\varepsilon=0.$
	\end{lemma}

	\begin{proof}
		Let $x_0\in\Omega$ and $\rho = d(x_0)$. Notice that when $\rho\geq \rho_0>0$, then $d^{\varepsilon+s}$ is smooth in the neighbourhood of $x_0$, and thus $L(d^{\varepsilon+s})(x_0)$ is bounded by a constant depending only on $\rho_0$. Thus we may assume that $\rho\in(0,\rho_0)$, for some small $\rho_0$ depending only on $\Omega$.
		
		Let us denote $l(x) = (d(x_0) + \nabla d(x_0)\cdot (x-x_0))_+$. With explicit computation and knowing the one-dimensional case, we get
		\begin{align*}
			L(l^{s+\varepsilon})&=L\left(|\nabla d(x_0)|^{\varepsilon+s}\left(\frac{d(x_0)}{|\nabla d(x_0)|}+ \frac{\nabla d(x_0)}{|\nabla d(x_0)|}(\cdot - x_0)\right)_+^{s+\varepsilon}\right)\\
			& = |\nabla d(x_0)|^{\varepsilon+s} c_{s+\varepsilon} \left(\frac{d(x_0)}{|\nabla d(x_0)|}+ \frac{\nabla d(x_0)}{|\nabla d(x_0)|}(\cdot - x_0)\right)^{\varepsilon-s},\quad \text{in }\{l>0\},
		\end{align*}
		where $c_{s+\varepsilon}=0$ if and only if $\varepsilon=0$ (see \cite[Theorem 3.10]{AR20}),
		and so we have 
		$$ L(l^{s+\varepsilon})(x_0) = |\nabla d(x_0)|^{2s}c_{s+\varepsilon} d^{\varepsilon-s}(x_0).$$
		
		Then, in the same way as in \cite[Proposition 2.3]{RS16} get the estimates
		$$|d(x_0+y)-l(x_0+y)|\leq C|y|^{1+\alpha},\quad y\in\R^n,$$
		$$ |\nabla d (x_0+y)-\nabla l(x_0+y)|\leq C|y|^\alpha,\quad y\in B_{\rho/2}.$$
		With bounding the derivatives carefully this gives
		$$|d^{s+\varepsilon}(x_0+y)-l^{s+\varepsilon}(x_0+y) |\leq \left\lbrace \begin{array}{cl}
		C\rho^{s+\varepsilon+\alpha-2}|y|^2&y\in B_{\rho/2},\\
		C|y|^{1+\alpha}(d^{s+\varepsilon-1}(x_0+y)+l^{s+\varepsilon-1}(x_0+y))&y\in B_1\backslash B_{\rho/2}\\
		C|y|^{s+\varepsilon}& y\in B_1^c.
		\end{array}\right.$$
		The first one bases on the estimate $||D^2(d^{s+\varepsilon}-l^{s+\varepsilon})||_{L^\infty(B_{\rho/2})}\leq C\rho^{s+\varepsilon+\alpha-2}$, which follows after explicit computation and various application of the estimate $|a^p-b^p|\leq |a-b|(a^{p-1}+b^{p-1})$, which is true for all positive $a,b$ and $p\in(-\infty,2)$. The second one is straight forward application of the latter estimate, and the last is based on the growth of $l$ at infinity.
		
		Now we are in position to estimate
		\begin{align*}
		\left| L(d^{s+\varepsilon})(x_0)\right. -&\left. c_{s+\varepsilon}|\nabla d(x_0)|^{2s}d^{\varepsilon-s}(x_0)\right| = \left| L(d^{s+\varepsilon}-l^{s+\varepsilon})(x_0) \right|\\
		\leq& \int_{\R^n} |d^{s+\varepsilon}-l^{s+\varepsilon}|(x_0+y)\frac{\Lambda}{|y|^{n+2s}}dy\\
		\leq& \int_{B_{\rho/2}}C\rho^{s+\varepsilon+\alpha-2}|y|^{2-n-2s}dy \\&+
			\int_{B_1\backslash B_{\rho/2}}C|y|^{1+\alpha}(d^{s+\varepsilon-1}+l^{s+\varepsilon-1})(x_0+y)\frac{1}{|y|^{n+2s}}dy\\
			&+\int_{B_1^c} C|y|^{-n-s+\varepsilon}dy\\
		\leq &C\rho^{\alpha + \varepsilon+s} + C(1+ \rho^{\varepsilon+\alpha-s}) \leq C\rho^{(\alpha+\varepsilon-s)\land 0}.
		\end{align*}
		The estimation of the first and the third integrand are computations, and on the middle one we apply \cite[Lemma 2.5]{RS16} twice.
	\end{proof}

	Now we turn to establishing Theorem \ref{goalOfSection2}. It extends \cite[Corollary 2.3]{AR20} in the sense that we allow wider choice of the power of the distance function. Basically we want to compute $L(d^p)$, for $p\in (0,2s).$ The approach is analogue to the one in \cite{AR20}. We transform the integral into a suitable form, then we expend the kernel into terms of increasing homogeneities and then treat each of them separately. Due to the change of power some cancellations are not happening and we have to track the additional terms. 
	
	We start with the following result, which computes $L(d^p)$ in the case where the domain is a half-space (often referred as the flat case). We strongly use the fact that the distance function is one-dimensional in that case and apply the computation from one dimension.

	\begin{lemma}\label{LonPower}
		Let $a\colon \mathbb{S}^{n-1}\to\R$ be an even function, such that $\int_{\mathbb{S}^{n-1}} |\theta_n|^{2s} a(\theta) d\theta <\infty$ and let $p\in(0, 2s)$. Then 
		$$p.v.\int_{\R^n} (x_n^p - (x_n+y_n)_+^p)\frac{a(\left\langle y\right\rangle )}{|y|^{n+2s}}dz = c_p x_n^{p-2s}\int_{\mathbb{S}^{n-1}} |\theta_n|^{2s} a(\theta) d\theta,\quad \text{when } x_n>0. $$
		Moreover, $c_p=0$ if and only if $p=s$.
	\end{lemma}
	\begin{proof}
		We use the knowledge that $\frac{1}{4}\int_\R (2\xi^p-(\xi+r)_+^p - (\xi-r)_+^p)\frac{1}{|r|^{1+2s}}dr = c_p \xi^{p-2s}$, for positive numbers $\xi,$ where $c_p=0$ if and only if $p=s$. This follows from the Liouville theorem in a half-space, see for example \cite[Theorem 3.10]{AR20}. Using the symmetry of the kernel, we rewrite the integral in the statement as
		$$ \frac{1}{2}\int_{\R^n} (2x_n^p - (x_n+y_n)_+^p-(x_n-y_n)_+^p)\frac{a(\left\langle y\right\rangle )}{|y|^{n+2s}}dz,$$
		and then applying the polar coordinates and using evenness of the integrand, we get
		$$\frac{1}{4}\int_{\mathbb{S}^{n-1}}\int_{\R} \big (2x_n^p - (x_n+r\theta_n)_+^p-(x_n-r\theta_n)_+^p\big )\frac{a(\theta )}{|r|^{n+2s}}r^{n-1}drd\theta=$$
		$$= \int_{\mathbb{S}^{n-1}} |\theta_n|^{p}\frac{1}{4}\int_{\R} \left(2\left(\frac{x_n}{|\theta_n|}\right)^p - \left(\frac{x_n}{|\theta_n|}+r\right)_+^p-\left(\frac{x_n}{|\theta_n|}-r\right)_+^p\right)\frac{1}{|r|^{1+2s}}dr\cdot a(\theta )d\theta=$$
		$$=\int_{\mathbb{S}^{n-1}} |\theta_n|^{p} c_p\left(\frac{x_n}{|\theta_n|}\right)^{p-2s}  a(\theta )d\theta = c_p x_n^{p-2s} \int_{\mathbb{S}^{n-1}} |\theta_n|^{2s} a(\theta) d\theta.$$
	\end{proof}
	
	We proceed with a result which connects the previous lemma with the terms which appear in computation further on. 
	
	\begin{lemma}\label{lostCalculation}
		Let $a\colon \mathbb{S}^{n-1}\to\R$ be an odd, bounded function, and let $p\in(0, 2s)$. Then 
		$$ p.v.\int_{\R^n} (z_n)_+^{p-1} \frac{a(\left\langle z-x\right\rangle)}{|z-x|^{n+2s-1}}dz = \tilde{c}_p x_n^{p-2s},\quad x\in\{x_n>0\}.$$
		Moreover, $\tilde{c}_p = 0$ if and only if $p=s$.
	\end{lemma}
	\begin{proof}
		Let us start with the left-hand side. By \cite[Lemma 2.4]{AR20}, we have
		$$ p.v.\int_{\R^n} (z_n)_+^{p-1} \frac{a(\left\langle z-x\right\rangle)}{|z-x|^{n+2s-1}}dz = p.v.\int_{\R^n} \nabla (z_n)_+^{p}\cdot (z-x) \frac{a(\left\langle z-x\right\rangle)/\left\langle z-x\right\rangle_n}{|z-x|^{n+2s}}dz $$ 
		$$= \tilde{L}((x_n)_+^p),$$
		where the kernel of $\tilde{L}$ is $\frac{a(\left\langle z-x\right\rangle)/\left\langle z-x\right\rangle_n}{|z-x|^{n+2s}}$. Since $a$ is bounded, $\frac{a(\theta)}{\theta_n}$ satisfies the assumption of Lemma \ref{LonPower}, and so the above expression equals $\tilde{c}_p x_n^{p-2s}$, where $\tilde{c}_p$ depends only on $p$ and $a$, and vanishes if and only if $p=s$.
	\end{proof}

	After splitting the kernel into homogeneities, we will have to deal with the terms of increasing power in the kernel. The next result shows how to deal with them. This is the main step towards Theorem \ref{mainResultSection2}. The extra parameter $\zeta$ is introduced because we want to apply the results on functions of the form $P(x-\zeta) d^p(x)$ for some polynomial $P$. It is going to be crucial to know the regularity in the additional parameter and hence need to track it carefully.
	
	\begin{lemma}\label{lemmaRegularity}
		Let $q \in \{x\in\R^n : x_n>0\}$ and $r>0$ be such that $\overline{B_r(q)}\subset B_1.$ Let $k\in\N,$ non-integer $\beta>1,$ $p\in(0, 2s)$, $k>\beta - 2s+p$. Assume that for every $\zeta$ varying over some bounded $C^\beta$ surface we have $\psi_\zeta \in C^{\beta-1}(\overline{B_1})\cap C^k(\overline{B_r(q)}),$ so that $\left|\left|D^j\psi_\zeta\right|\right|_{C^\beta_\zeta}\leq C_0$, for $j\geq 0$. For $j\in\N$, let $a_j\in C^{j+k-1}(\mathbb{S}^{n-1})$ satisfy 
		$$a_j(-\theta) = (-1)^{j+1} a_j(\theta),\quad \theta \in \mathbb{S}^{n-1}.$$
		Then the function defined as 
		$$I_j(x) := \operatorname{p.v.}\int_{B_1}(z_n)_+^{p-1} \frac{a_j(\left\langle z-x\right\rangle )}{|z-x|^{n+2s-j-1}}\psi_\zeta(z) dz$$ satisfies 
		$$I_j(x) = P_\zeta(x)x_n^{p-2s} + R_\zeta(x),$$ where $P_\zeta$ is a polynomial whose coefficients are bounded with $C||\psi_\zeta||_{C^{\beta-1}(B_1)}||a_j||_{C^{k+j-1}(\mathbb{S}^{n-1})},$ and are $C^\beta$ smooth as functions of $\zeta$ with bounded $C^\beta_\zeta$ norm, and $R_\zeta$ is of class $C^{k+j-1}$ in $B_{r/2}(q)$ with 
		$$|D^{k+j-1}R_\zeta(x)| \leq C||a_j||_{C^{k+j-1}(\mathbb{S}^{n-1})}\times$$
		$$\times\left(\sum_{|\gamma|=\lfloor\beta\rfloor }^{k-1} ||\psi_\zeta||_{C^{|\gamma|}(B_r(q))} r^{p-2s-k+1+|\gamma|} + ||\psi_\zeta||_{C^k(B_r(q))} r^{1+p-2s} + ||\psi_\zeta||_{C^{\beta-1}(B_1)}r^{\beta - k + p -2s}\right).$$ 
		The constant $C$ depends only on $n$, $s$, $\beta$ and $p$.
	\end{lemma}
	\begin{proof}
		Fix some point $x_0$ in $B_{r/2}(q)$ and write
		\begin{align*}
		\psi_\zeta(z) &= \sum_{|\gamma| \leq k-1} \partial^\gamma \psi_\zeta(x_0) (z-x_0)^\gamma + P_k(x_0,z) \\
		&= \sum_{|\gamma|\leq k-1}\partial^\gamma \psi_\zeta(x_0) \sum_{\alpha\leq\gamma} \binom{\gamma}{\alpha}(z-x)^\alpha (x-x_0)^{\gamma-\alpha}+P_k\\
		&= \sum_{|\alpha|\leq k-1} \Psi^k_\alpha(x_0,x) (z-x)^\alpha + P_k(x_0,z), \quad \quad z\in B_{r}(q),
		\end{align*}
		where $\Psi_\alpha^k (x_0,x) = \sum_{\gamma\geq\alpha}^{|\gamma|\leq k-1}\binom{\gamma}{\alpha}\partial^\gamma\psi_\zeta(x_0)(x-x_0)^{\gamma-\alpha}$ and $|P_k(x_0,z)|\leq ||\psi_\zeta||_{C^k(B_r(q))}|z-x_0|^k.$ Note that all the dependence in $\zeta$ comes through $\partial^\gamma\psi_\zeta$, which is by assumption of bounded $C^\beta_\zeta$ norm. %norm depends on C_j up to k or something like this...
		Moreover, $\Psi^k_\alpha$ does not depend on $z$ and hence it exits all the integrals over the variable $z$.
		Similarly, we expand $\psi_\zeta$ in $B_1\backslash B_r(q)$ using $C^{\beta-1}$ regularity
		\begin{align*}
		\psi_\zeta(z) &= \sum_{|\gamma| \leq \beta-1} \partial^\gamma \psi_\zeta(x_0) (z-x_0)^\gamma + P_{\beta-1}(x_0,z) \\
		&= \sum_{|\alpha|\leq \beta-1} \Psi^{\beta-1}_\alpha(x_0,x) (z-x)^\alpha + P_{\beta-1}(x_0,z), \quad \quad z\in B_1\backslash B_r(q),
		\end{align*}
		where $\Psi_\alpha^{\beta-1}(x_0,x) = \sum_{\gamma\geq\alpha}^{|\gamma|\leq \beta-1}\binom{\gamma}{\alpha}\partial^\gamma\psi_\zeta(x_0)(x-x_0)^{\gamma-\alpha}$ and $|P_{\beta-1}(x_0,z)|\leq ||\psi_\zeta||_{C^{\beta-1}(B_1)} |z-x_0|^{\beta-1}$.
		We plug these expansions into the definition of $I_j$ so that
		\begin{align}
		I_j(x) =& \int_{B_r(p)}(z_n)^{p-1} \frac{a_j(\left\langle z-x\right\rangle)}{|z-x|^{n+2s-j-1}}\psi_\zeta(z) dz + \int_{B_1\backslash B_r(q)}(z_n)_+^{p-1} \frac{a_j(\left\langle z-x\right\rangle)}{|z-x|^{n+2s-j-1}}\psi_\zeta(z) dz \nonumber \\ 
		=&\sum_{|\alpha|\leq \beta -1} \Psi_\alpha^{\beta-1}(x_0,x)\int_{B_1}(z_n)_+^{p-1} \frac{a_j(\left\langle z-x\right\rangle)\left\langle z-x\right\rangle^\alpha}{|z-x|^{n+2s-j-1-|\alpha|}}dz \label{niceTerm}\\
		&+ \sum_{\lfloor\beta\rfloor \leq |\alpha|\leq k-1} \Psi^{k}_\alpha (x_0,x)\int_{B_r(q)}(z_n)^{p-1} \frac{a_j(\left\langle z-x\right\rangle)\left\langle z-x\right\rangle^\alpha}{|z-x|^{n+2s-j-1-|\alpha|}}dz \label{term23}\\
		& + \sum_{ |\alpha|\leq \lfloor\beta\rfloor-1} \left(\Psi^{k}_\alpha (x_0,x)-\Psi^{\beta-1}_\alpha(x_0,x)\right)\int_{B_r(q)}(z_n)^{p-1} \frac{a_j(\left\langle z-x\right\rangle)\left\langle z-x\right\rangle^\alpha}{|z-x|^{n+2s-j-1-|\alpha|}}dz \label{additionalTerm23}\\
		\begin{split}
		&+\int_{B_r(q)}(z_n)^{p-1} \frac{a_j(\left\langle z-x\right\rangle)}{|z-x|^{n+2s-j-1}}P_k(x_0,z) dz\\
		&+\int_{B_1\backslash B_r(q)}(z_n)_+^{p-1} \frac{a_j(\left\langle z-x\right\rangle)}{|z-x|^{n+2s-j-1}}P_{\beta-1}(x_0,z) dz.
		\end{split}\label{reminders}
		\end{align}
		Consider first the term \eqref{niceTerm}. Note that $\Psi^{\beta-1}_\alpha(x_0,x)$ are polynomials in $x$, whose coefficients are derivatives of $\psi_\zeta$. To analyse the integral part, choose a multi-index $\delta$ with $|\delta| = j$ and calculate 
		\begin{align*}
		\partial^{\alpha+\delta}_x \int_{B_1}(z_n)_+^{p-1} \frac{a_j(\left\langle z-x\right\rangle)\left\langle z-x\right\rangle^\alpha}{|z-x|^{n+2s-j-1-|\alpha|}}dz &= \int_{B_1}(z_n)_+^{p-1} \frac{\tilde{a}_j(\left\langle z-x\right\rangle)}{|z-x|^{n+2s-1}}dz\\
		&= \tilde{L}((x_n)_+^p)(x) - \int_{\R^n\backslash B_1}(z_n)_+^{p-1} \frac{\tilde{a}_j(\left\langle z-x\right\rangle)}{|z-x|^{n+2s-1}}dz,
		\end{align*}
		where we used that $\tilde{a}_j$ represents a kernel of some operator $\tilde{L}$. 
		But due to Lemma \ref{lostCalculation}, $\tilde{L}((x_n)_+^p)(x) = c_{p}x_n^{p-2s}$ for $x_n>0$, where $c_p$ depends only on $p$ and $a_j$, and the remaining integrals is as smooth as $\tilde{a}_j$, which is 
		$C^{j+k-1-|\alpha|-j} = C^{k-1-|\alpha|}.$ Hence the original function equals $Q(x) x_n^{p-2s} + S(x)$ where $Q$ is a polynomial and $S$ is of class $C^{k-1+j}(B_{1/2})$. 
		Because $\Psi_\alpha^{\beta-1}$ are polynomials in $x$ with coefficients bounded with $||\psi_\zeta||_{C^{\beta-1}(B_1)}$ and $C^\beta_\zeta$ smooth, the term \eqref{niceTerm} contributes $P(x) x_n^{p-2s} + R_1(x)$, where $P$ is as stated in the claim, and $R_1$ satisfies 
		$|D^{j+k-1}R_1(x)|\leq C ||a_j||_{C^{k+j-1}(\mathbb{S}^{n-1})}||\psi_\zeta||_{C^{\beta-1}(B_1)},$ for $x \in B_{1/2}\cap\{x_n>0\}$. 
		
		Proceed now to term \eqref{term23}. Let us rewrite it for convenience:
		$$\sum_{\lfloor\beta\rfloor \leq |\alpha|\leq k-1} 
		\sum_{\gamma\geq \alpha} \binom{\gamma}{\alpha}\partial^\gamma\psi_\zeta(x_0)(x-x_0)^{\gamma-\alpha}
		\int_{B_r(q)}(z_n)^{p-1} \frac{a_j(\left\langle z-x\right\rangle)\left\langle z-x\right\rangle}{|z-x|^{n+2s-j-1-|\alpha|}}dz.$$
		Choose such multi-indices $\alpha,\gamma$, and another one with $|\delta| = j+k-1.$ Then 
		$$\partial_x^\delta \left( (x-x_0)^{\gamma-\alpha}
		\int_{B_r(q)}(z_n)^{p-1} \frac{a_j(\left\langle z-x\right\rangle)\left\langle z-x\right\rangle}{|z-x|^{n+2s-j-1-|\alpha|}}dz  \right)= $$ 
		$$= \sum_{\eta\leq \min(\gamma-\alpha, \delta)} \binom{\delta}{\eta} c_{\gamma-\alpha,\eta}(x-x_0)^{\gamma-\alpha-\eta} \partial^{\delta-\eta} 
		\int_{B_r(q)}(z_n)^{p-1} \frac{a_j(\left\langle z-x\right\rangle)\left\langle z-x\right\rangle}{|z-x|^{n+2s-j-1-|\alpha|}}dz. $$
		Note that since $|\eta|\leq |\gamma-\alpha| \leq k-1-|\alpha| $ we have that $|\delta-\eta|\geq k-1+j-(k-1-|\alpha|) =j+|\alpha|$. Let us analyse the integral part. Choose any $\epsilon\leq\delta-\eta$ of order $|\epsilon| = j+|\alpha|$, and compute first 
		\begin{equation}\label{wierdCase}
		\partial^\epsilon\int_{B_r(q)}(z_n)^{p-1} \frac{a_j(\left\langle z-x\right\rangle)\left\langle z-x\right\rangle}{|z-x|^{n+2s-j-1-|\alpha|}}dz = \int_{B_r(q)}(z_n)^{p-1} \frac{\tilde{a}_j(\left\langle z-x\right\rangle)}{|z-x|^{n+2s-1}}dz
		\end{equation}
		$$ = \tilde{L}((x_n)_+^{p}) -\int_{\R^n\backslash B_1}(z_n)^{p-1} \frac{\tilde{a}_j(\left\langle z-x\right\rangle)}{|z-x|^{n+2s-1}}dz-\int_{B_1\backslash B_r(q)}(z_n)^{p-1} \frac{\tilde{a}_j(\left\langle z-x\right\rangle)}{|z-x|^{n+2s-1}}dz$$
		$$ = c_p x_n^{p-2s} + g(x) - \int_{B_1\backslash B_r(q)}(z_n)_+^{p-1} \frac{\tilde{a}_j(\left\langle z-x\right\rangle)}{|z-x|^{n+2s-1}}dz,$$
		where $c_p$ is an explicit constant depending on $a_j$, obtained in Lemma \ref{lostCalculation}, and $g$ is or class $C^{k-1-|\alpha|}$. Differentiating the above equation $\delta-\eta-\epsilon$ more, we get 
		$$ \partial^{\delta-\eta} 
		\int_{B_r(q)}(z_n)^{p-1} \frac{a_j(\left\langle z-x\right\rangle)\left\langle z-x\right\rangle}{|z-x|^{n+2s-j-1-|\alpha|}}dz =$$ $$= \tilde{c}_p x_n^{p-2s-|\delta-\eta-\epsilon|} + \partial^{\delta-\eta-\epsilon}g(x) + \int_{B_1\backslash B_r(q)}(z_n)_+^{p-1} \frac{\hat{a}_j(\left\langle z-x\right\rangle)}{|z-x|^{n+2s-1 +|\delta-\eta-\epsilon|}}dz.$$ 
		Since $x\in B_{r/2}(q),$ the first term is bounded with $$C||a_j||_{C^{k+j-1}(\mathbb{S}^{n-1})} r^{p-2s-|\delta-\eta-\epsilon|} = C||a_j||_{C^{k+j-1}(\mathbb{S}^{n-1})} r^{p-2s-k+1 + |\eta| + |\alpha|},$$ and the second one just with $||a_j||_{C^{k+j-1}(\mathbb{S}^{n-1})}$. For the third one we use Lemma \ref{generalisedA9}
		to bound it with $||a_j||_{C^{k+j-1}(\mathbb{S}^{n-1})} r^{p-2s - |\delta-\eta-\epsilon|}.$ To apply the lemma, we need $2s-1+|\delta-\eta-\epsilon|\geq 0$. The only time, this does not happen is, when $|\delta-\eta-\epsilon|= 0$, and $s<\frac{1}{2}.$ But then, $\delta-\eta = \epsilon$, and the pole we get in the right-hand side in \eqref{wierdCase} is integrable, since $2s-1<0$. Therefore we can directly apply Lemma \ref{generalisedA9}, to end up with the same estimate.
		Putting it all together, we got the following estimate for the norm of $D^{k-1+j}$ of the term \eqref{term23}:
		$$\sum_{\lfloor\beta\rfloor \leq |\alpha|\leq k-1} 
		\sum_{\gamma\geq \alpha} \binom{\gamma}{\alpha}\partial^\gamma\psi_\zeta(x_0)\sum_{\eta\leq \min(\gamma-\alpha, \delta)} \binom{\delta}{\eta} c_{\gamma-\alpha,\eta}(x-x_0)^{\gamma-\alpha-\eta} \times$$
		$$
		||a_j||_{C^{k+j-1}(\mathbb{S}^{n-1})} r^{p-2s - |\delta-\eta-\epsilon|}		
		$$
		$$\leq C||a_j||_{C^{k+j-1}(\mathbb{S}^{n-1})}\sum_{\gamma,\alpha,\eta}
		\partial^\gamma\psi_\zeta(x_0) r^{|\gamma-\alpha-\eta|}r^{p-2s-k+1 + |\eta| + |\alpha|}$$ $$\leq  C||a_j||_{C^{k+j-1}(\mathbb{S}^{n-1})}\sum_{\lfloor\beta\rfloor \leq |\gamma|\leq k-1} ||\psi_\zeta||_{C^{|\gamma|}(B_r(q))} r^{p-2s-k+1+|\gamma|}.$$
		
		With \eqref{additionalTerm23} we proceed similarly. Note first, that 
		$$\Psi^{k}_\alpha (x_0,x)-\Psi^{\beta-1}_\alpha(x_0,x) = \sum^{\gamma\geq\alpha}_{\lfloor\beta\rfloor \leq|\gamma|\leq k-1} \binom{\gamma}{\alpha}\partial^\gamma\psi_\zeta(x_0)(x-x_0)^{\gamma-\alpha}, $$ and so 
		$$\eqref{additionalTerm23} = \sum_{ |\alpha|\leq \lfloor\beta\rfloor-1} \sum^{\gamma\geq\alpha}_{\lfloor\beta\rfloor \leq|\gamma|\leq k-1} \binom{\gamma}{\alpha}\partial^\gamma\psi_\zeta(x_0)(x-x_0)^{\gamma-\alpha}\int_{B_r(q)}(z_n)^{p-1} \frac{a_j(\left\langle z-x\right\rangle)\left\langle z-x\right\rangle}{|z-x|^{n+2s-j-1-|\alpha|}}dz.$$
		When we differentiate the expression $j+k-1$ times (choose multi-index $\delta$), perform the Leibnitz rule (with multi-index $\eta$), we end up with the same terms as when we treated \eqref{term23}, just that it also happens that
		$|\delta-\eta| < j+|\alpha|$. Then we just estimate the integral we get with Lemma \ref{generalisedA9}, to obtain the bound with the desired power of $r$. Therefore we get the same estimate as before, just the range of the index $\alpha$ is different
		$$|D^{j+k-1}\eqref{additionalTerm23}|\leq  C||a_j||_{C^{k+j-1}(\mathbb{S}^{n-1})}\sum_{\lfloor\beta\rfloor \leq |\gamma|\leq k-1} ||\psi_\zeta||_{C^{|\gamma|}(B_r(q))} r^{p-2s-k+1+|\gamma|}.$$
		
		Finally, to estimate \eqref{reminders}, use the argumentation as in \cite[Lemma 2.9]{AR20}
		to get
		\begin{align*}
			&\left|D^{k-1+j}|_{x_0} \int_{B_r(q)}(z_n)^{p-1} \frac{a_j(\left\langle z-x\right\rangle)}{|z-x|^{n+2s-j-1}}P_k(x_0,z) dz\right|\leq\\
			\leq& C||a_j||_{C^{k+j-1}}||\psi_\zeta||_{C^{k}(B_r(q))}\int_{B_r(q)}(z_n)^{p-1} \frac{1}{|z-x|^{n+2s-2}}dz\\
			\leq& C||a_j||_{C^{k+j-1}}||\psi_\zeta||_{C^{k}(B_r(q))}r^{1-2s+p} 
		\end{align*}
		in view of Lemma \ref{generalisedA9}. Similarly,		
		\begin{align*}
			&\left|D^{k-1+j}|_{x_0}\int_{B_1\backslash B_r(q)}(z_n)_+^{p-1} \frac{a_j(\left\langle z-x\right\rangle)}{|z-x|^{n+2s-j-1}}P_{\beta-1}(x_0,z) dz\right|\leq\\ 
			\leq& C||a_j||_{C^{k+j-1}}||\psi_\zeta||_{C^{\beta-1}(B_1)}\int_{B_1\backslash B_r(q)}(z_n)_+^{p-1} \frac{1}{|z-x|^{n+2s-1 + k -\beta}}dz \\
			\leq& C||a_j||_{C^{k+j-1}}|| \psi_\zeta||_{C^{\beta-1}(B_1)}r^{\beta - k - 2s+p},
		\end{align*}
		and the result follows.
	\end{proof}
	
	Next, we show how to connect the obtained estimates on the derivatives with the regularity up to the boundary. Of course we need to have suitable assumptions on the function $\psi$.

	\begin{corollary}\label{corollaryRegularity}
		Let $p\in(0,2s)$ and $\beta>1+2s-p$. Suppose for every $\zeta$ varying over some bounded $C^\beta$ surface, we have $\psi_\zeta\in C^{\beta-1}(B_1)\cap C^\infty(B_1 \cap \{x_n>0\})$ satisfying
		$$|D^k\psi_\zeta|\leq C_k d^{\beta-1-k},\quad \quad k>\beta-1,$$
		$$D^k\psi_\zeta \text{ is } C^{\beta}\text{ in variable }\zeta \quad k\in\N.$$
		For some $j\in \N$, let $a_j\in C^{j+\lfloor \beta-2s+p\rfloor}(\mathbb{S}^{n-1})$ satisfy
		$$a_j(\theta) = (-1)^{j+1}a_j(\theta),\quad \theta \in \S^{n-1}.$$
		Then the function defined as 
		$$I_j(x) := p.v.\int_{B_1}(z_n)_+^{p-1}\frac{a_j(\langle z-x\rangle )}{|z-x|^{n+2s-j-1}}\psi_\zeta(z)dz$$
		satisfies
		$$I_j(x) = P_\zeta(x)x_n^{p-2s} + R_\zeta(x),$$
		where $P_\zeta$ is a polynomial whose coefficients are bounded with $C||\psi_\zeta||_{C^{\beta-1}(B_1)}||a_j||_{C^{j+\lfloor \beta-2s+p\rfloor}(\S^{n-1})}$ and $C^\beta$ in variable $\zeta$. Furthermore, the function $R_\zeta$ is of class $C^{j+\beta-1-2s+p}(B_{1/2}\cap \{x_n\geq0\})$, with 
		$$||R_\zeta||_{C^{j+\beta-1-2s+p}} \leq C ||a_j||_{C^{j+\beta}(\mathbb{S}^{n-1})}||\psi_\zeta||_{C^{\beta-1}(B_1)},$$
		where the constant $C$ depends only on $n$, $s$, $\beta$ and $p$.
	\end{corollary}
	
	\begin{proof}
		Choose $k = \lfloor \beta-2s+p\rfloor +1$. From the above lemma, whenever $x\in B_r(x_0)$ and $d(x_0)=2r$, we get
		$$|D^{k+j-1}R_\zeta(x)| \leq C||a_j||_{C^{k+j-1}(\mathbb{S}^{n-1})}\times$$
		$$\times\left(\sum_{|\gamma|=\lfloor\beta\rfloor }^{k-1} ||\psi_\zeta||_{C^{|\gamma|}(B_r(q))} r^{p-2s-k+1+|\gamma|} + ||\psi_\zeta||_{C^k(B_r(q))} r^{1+p-2s} + ||\psi_\zeta||_{C^{\beta-1}(B_1)}r^{\beta - k + p -2s}\right),$$ 
		which we furthermore estimate with 
		$$\leq    C||a_j||_{C^{k+j-1}(\mathbb{S}^{n-1})}   r^{\beta - k + p -2s},$$
		where we used the assumption on $\psi_\zeta$. This gives, that every derivative of order $k+j-2$ is a $C^{\left\langle \beta-2s+p \right\rangle }$ function in $B_{1/2}\cap\{x_n\geq0\}$.
	\end{proof}

	The following develops analogous result as Lemma \ref{lemmaRegularity}, in the setting when the function $\psi$ has a zero of order one at origin. Then, the regularity estimate improves roughly by one power as well. We establish it in the similar way as before, just the cases we treat change a little. Note also, that this step is not done completely correct in \cite{AR20}, and deserves more attention. 
	
	\begin{lemma}\label{improvedLemmaRegularity}
		Let $q \in \{x\in\R^n : x_n>0\}$ and $r>0$ be such that $\overline{B_r(q)}\subset B_1.$ Take $k\in\N,$ non-integer $\beta>1,$ $p\in(0,2s)$, $k>\beta+1 - 2s+p$. Assume that for every $\zeta$ varying over some bounded $C^\beta$ surface  we have a function $\psi_\zeta \in C^{\beta-1}(\overline{B_1},\R^n)\cap C^k(\overline{B_r(q)},\R^n),$ satisfying $\left|\left|D^j\psi_\zeta\right|\right|_{ C^\beta_\zeta}\leq C_0,$ for $j\geq 0$. For $j\in\N$, let $a_j\in C^{j+k}(\mathbb{S}^{n-1})$ satisfy 
		$$a_j(-\theta) = (-1)^{j+1} a_j(\theta),\quad \theta \in \mathbb{S}^{n-1}.$$
		Then the function defined as 
		$$I_j(x) := \operatorname{p.v.}\int_{B_1}(z_n)^{p-1} \frac{a_j(\left\langle z-x\right\rangle )}{|z-x|^{n+2s-j-1}}z\cdot\psi_\zeta(z) dz$$ satisfies 
		$$I_j(x) = P_\zeta(x)x_n^{p-2s} + R_\zeta(x),$$ where $P_\zeta$ is a polynomial whose coefficients are bounded with $C||\psi_\zeta||_{C^{\beta-1}(B_1)}||a_j||_{C^{k+j-1}(\mathbb{S}^{n-1})}$ and $C^\beta$ functions of variable $\zeta$ with bounded $C^\beta_\zeta$ norm, and $R_\zeta$ is of class $C^{k+j-1}$ in $B_{r/2}(q)$ with 
		$$|D^{k+j-1}R_\zeta(x)| \leq C||a_j||_{C^{k+j-1}(\mathbb{S}^{n-1})}|q|\times$$
		$$\times\left(\sum_{|\gamma|=\lfloor\beta\rfloor }^{k-1} ||\psi_\zeta||_{C^{|\gamma|}(B_r(q))} r^{p-2s-k+1+|\gamma|} + ||\psi_\zeta||_{C^k(B_r(q))} r^{1+p-2s} + ||\psi_\zeta||_{C^{\beta-1}(B_1)}r^{\beta - k + p -2s}\right).$$ 
	\end{lemma}
	
	\begin{proof}
		The beginning is the same as in the lemma above, but now we deal with integrals of the form 
		\begin{align}
		\int_{B_1}z_i(z_n)_+^{p-1} \frac{a_j(\left\langle z-x\right\rangle)\left\langle z-x\right\rangle^\alpha}{|z-x|^{n+2s-j-1-|\alpha|}}dz, \label{term6} \\	
		(x-x_0)^{\gamma-\alpha}\int_{B_r(q)}z_i(z_n)^{p-1} \frac{a_j(\left\langle z-x\right\rangle)\left\langle z-x\right\rangle^\alpha}{|z-x|^{n+2s-j-1-|\alpha|}}dz, \label{term7}
		\\
		\int_{B_r(q)}z_i(z_n)^{p-1} \frac{a_j(\left\langle z-x\right\rangle)}{|z-x|^{n+2s-j-1}}P_k(x_0,z) dz, \label{term8}
		\\
		\int_{B_1\backslash B_r(q)}z_i(z_n)_+^{p-1} \frac{a_j(\left\langle z-x\right\rangle)}{|z-x|^{n+2s-j-1}}P_{\beta-1}(x_0,z) dz. \label{term9}
		\end{align}
		First we split $z_i = x_i + (z_i-x_i)$, and treat both terms analogously as in the previous lemma. Let us start with \eqref{term6}:
		$$\eqref{term6} = x_i\int_{B_1}(z_n)_+^{p-1} \frac{a_j(\left\langle z-x\right\rangle)\left\langle z-x\right\rangle^\alpha}{|z-x|^{n+2s-j-1-|\alpha|}}dz +
		\int_{B_1}(z_n)_+^{p-1} \frac{a_j(\left\langle z-x\right\rangle)\left\langle z-x\right\rangle^{\alpha+e_i}}{|z-x|^{n+2s-j-1-|\alpha+e_i|}}dz.$$
		Performing the same as in previous lemma, we see, that it equals $P_\zeta(x)x_n^{p-2s} + R_1(x)$, where $P_\zeta$ is a polynomial with the suitable bound and regularity in $\zeta$, and $R_1$ is as smooth as the kernel $a_j.$ 
		Since the power of the pole in the second integral is for one bigger than usual we need to assume, that $a_j$ is of one class smoother.
		
		The term \eqref{term7} splits into 
		$$(x-x_0)^{\gamma-\alpha}x_i\int_{B_r(q)}(z_n)^{p-1} \frac{a_j(\left\langle z-x\right\rangle)\left\langle z-x\right\rangle^\alpha}{|z-x|^{n+2s-j-1-|\alpha|}}dz
		+$$ 
		$$+ (x-x_0)^{\gamma-\alpha}\int_{B_r(q)}(z_n)^{p-1} \frac{a_j(\left\langle z-x\right\rangle)\left\langle z-x\right\rangle^{\alpha+e_i}}{|z-x|^{n+2s-j-1-|\alpha+e_i|}}dz,$$
		which we analyse just like \eqref{term23} and \eqref{additionalTerm23}. The first integrand brings $|x_i|$ more to the estimate, while the second brings $r$, so together the estimate improves for $C|q|$.
		
		The reminder terms, \eqref{term8} and \eqref{term9}, both still work. In analysis of the first one we get an integrable pole, so we use the other case of Lemma \ref{generalisedA9} which improves the estimate for $C|q|$, and for the other one we get a new restriction on $k$: $\beta+1-2s+p<k$, so that we can use Lemma \ref{generalisedA9}, and we get one extra power of $r$, which is majorised with $C|q|$.
	\end{proof}

	In the following result we show how to connect the computation of $L(\psi d^p)$ with the previous results, to get the desired statement. As in the model case \cite[Section 2]{AR20}, we split the kernel into homogeneities, we perform the flattening of the boundary, and carefully treat the obtained terms.  
	
	\begin{theorem}\label{mainResultSection2}
		Let $\Omega$ be a domain in $\R^n$, such that $0\in\partial\Omega$ and $\partial\Omega\cap B_1\in C^\beta,$ for some $\beta>1$, and $p\in(0, 2s)$. Let an integer $k> \beta-2s+p-1$ and the kernel $K$ be a $C^{2k+1}$ function satisfying \eqref{kernelConditions}. Let for every $\zeta\in\partial\Omega\cap B_1$ the function $\psi_\zeta\in C^{\beta-1}(\overline{B}_1,\R^n)\cap C^\infty(\Omega\cap B_1,\R^n)$ satisfies
		\begin{align}\label{conditionLikeD}
		\begin{split}
		\left|\left|D^j\psi_\zeta(x)\right|\right| _ {C^\beta_\zeta(\partial\Omega\cap B_1)}\leq C_0,\quad \forall x\in \Omega\cap B_1, j\geq 0,\text{ and}\\
		|D^j\psi_\zeta(x)| \leq C_j r^{\beta-1-j},\quad\quad j>\beta-1,\text{ }x\in B_r(x_0),
		\end{split}
		\end{align}
		whenever $B_{2r}(x_0)\subset \Omega\cap B_1$. Then 
		\begin{align*}
			L_\psi(d^p)(x) :=& p.v.\int_{B_1} d^{p-1}(y)K(y-x)\psi_\zeta(y)\cdot (y-x)dy\\ =& \varphi_\zeta(x) d^{p-2s}(x) + R_\zeta(x),
		\end{align*}
		where $\varphi_\zeta$ is composition of the boundary flattening map with a polynomial whose coefficients'  $C^\beta$ norm in variable $\zeta$ is bounded with a constant depending on $C_0$. We have $||\varphi_\zeta||_{C^\beta(B_1)}\leq C ||\psi_\zeta||_{C^{\beta-1}(B_1)}$, and $R_\zeta$ satisfies the following estimate:
		$$|D^k R_\zeta(x)|\leq C d^{\beta-1-2s+p-k}(x),$$
		for $x\in B_{1/2}$. 
		
		Moreover, if $\psi_\zeta = fF_\zeta$, where $f$ is a $C^{\beta}$ function satisfying $|D^jf(x)| \leq C_j r^{\beta-j},$ $j>\beta$, $ x\in B_r(x_0),$ $B_{2r}(x_0)\subset \Omega$, and vanishing at $0$, and $F_\zeta$ a vector function satisfying \eqref{conditionLikeD}, then $R_\zeta$ satisfies 
		$$|D^k R_\zeta(x)|\leq C |x|d^{\beta-1-2s+p-k}(x),$$
		for $x\in B_{1/2}$.
		The constant $C$  depends only on $k$, $s$, $||K||_{C^{2k+1}(\S^{n-1})}$, $\beta$ and $||\partial\Omega||_{C^\beta}.$
	\end{theorem}
	
	\begin{proof}
		Flattening the boundary (with $\phi$) just like in \cite[Section 2.2]{AR20}, and using the same notation as therein, we get
		$$p.v.\int_{B_1} d^{p-1}(y)K(y-x)\psi_\zeta(y)\cdot (y-x)dy = p.v.\int_{B_1} (z_n)_+^{p-1}(y)J(\phi(z)-\phi(\hat{x}))\cdot \rho_\zeta(z)dz =:I(\hat{x}).$$
		Remember that $\rho_\zeta(z)=\psi_\zeta(\phi(z))|\operatorname{det}D\phi(z)|.$ 
		Condition \eqref{conditionLikeD} is closed under multiplication of functions and composites. Since both $\phi$ and $D\phi$ satisfy it, so does $\rho_\zeta$. For simplicity we omit the hat script. We want to prove that $I(x) = P_\zeta(x)x_n^{p-2s} + \hat{R}_\zeta(x),$ where $P_\zeta$ is a polynomial and $\hat{R}_\zeta$ satisfies 
		$$|D^kR_\zeta(x)|\leq C d^{\beta-1-2s+p-k}(x).$$
		
		We proceed the same way as in the paper with Nicola. We fix a multi-index $\alpha$ of order $|\alpha|=k$, and a point $q\in \{x_n>0\}.$ We denote $r = d(q)/2$. We split $I(x) = I_1(x) + I_r(x,x)$, and proceed as in \cite[proof of Theorem 2.2]{AR20} to write
		\begin{align*}
			\partial^\alpha_x I_r(x,x) =& \sum_{\gamma\leq\alpha}\binom{\alpha}{\gamma}\sum_{i=0}^{|\gamma|} \partial^\gamma_x\partial^{\alpha-\gamma}_\xi \int_{B_r(q)}(z_n)_+^{p-1}\frac{b_i(\xi,\left\langle z-x\right\rangle )}{|z-x|^{n+2s-i-1}}\cdot \rho_\zeta(z)dz\\
			&+\sum_{\gamma\leq\alpha}\binom{\alpha}{\gamma} \partial^\gamma_x\partial^{\alpha-\gamma}_\xi \int_{B_r(q)}(z_n)_+^{p-1}\frac{R_{|\gamma|+1}(\xi,\left\langle z-x\right\rangle )}{|z-x|^{n+2s-1}}\cdot \rho_\zeta(z)dz.
		\end{align*}
		To do expansions as in \cite{AR20}, we need that the kernel is $|\alpha|+|\gamma|+1$ times differentiable, so since $\gamma\leq \alpha$, we need $2k+1$ regularity of the kernel.
		
		We estimate the error term in the same way as in \cite{AR20}, but we use  $|\rho_\zeta(z)|\leq C$:
		\begin{equation}
		\begin{split}
		\left|\partial^\gamma_x\partial^{\alpha-\gamma}_\xi \int_{B_r(q)}(z_n)_+^{p-1}\frac{R_{|\gamma|+1}(\xi,\left\langle z-x\right\rangle )}{|z-x|^{n+2s-1}}\cdot \rho_\zeta(z)dz\right|&\leq \\
		\leq Cr^{\beta-|\alpha|-2}\int_{B_r(q)}(z_n)^{p-1}|z-x|^{-n-2s+2}dz
		&\leq C r^{\beta-2s+p-1-k},
		\end{split}
		\end{equation}
		in view of Lemma \ref{generalisedA9}.
		
		When $|\gamma|>i+\beta-1-2s+p$, we write 
		$$\partial^\gamma_x \int_{B_r(q)}(z_n)_+^{p-1}\frac{\partial^{\alpha-\gamma}_\xi b_i(\xi,\left\langle z-x\right\rangle )}{|z-x|^{n+2s-i-1}}\cdot \rho_\zeta(z)dz = $$
		$$ = -\partial^\gamma_x \int_{B_1\backslash B_r(q)}(z_n)_+^{p-1}\frac{\partial^{\alpha-\gamma}_\xi b_i(\xi,\left\langle z-x\right\rangle )}{|z-x|^{n+2s-i-1}}\cdot \rho_\zeta(z)dz + \Gamma(x) +\partial^\gamma (P_1(x)x_n^{p-2s}),$$
		where we denoted 
		$$\Gamma(x) := \partial^\gamma \int_{B_1}(z_n)_+^{p-1}\frac{\partial^{\alpha-\gamma}_\xi b_i(\xi,\left\langle z-x\right\rangle )}{|z-x|^{n+2s-i-1}}\cdot \rho(z)dz-\partial^\gamma (P_1(x)x_n^{p-2s}),$$
		for suitable polynomial $P_1$ obtained in Lemma \ref{lemmaRegularity}. Note that the lemma gives that the coefficients of obtained polynomial are $C^\beta$ in $\zeta$ with bounded $C^\beta_\zeta$ norm.
		Due to condition on $|\gamma|$ above, we can apply Lemma \ref{lemmaRegularity} (we use $k'=|\gamma|-i+1$ in the assumption ot the lemma) to every summand of $\Gamma$, to say 
		$$|\Gamma(x)|\leq C||\partial^{\alpha-\gamma}_\xi b_i(\xi,\cdot)||_{C^{|\gamma|}(\mathbb{S}^{n-1})}  r^{\beta-|\gamma|+i-1-2s+p},$$
		where we already used the estimates on $\rho_\zeta$. Using this, together with the estimates on $||\partial^{\alpha-\gamma}_\xi b_i(\xi,\cdot)||_{C^{|\gamma|}(\mathbb{S}^{n-1})}$ in \cite[(26)]{AR20}, we end up with
		\begin{equation}\label{GammaTerm}
		|\Gamma(x)|\leq C r^{\beta-|\gamma|+i-1-2s+p} \left\lbrace\begin{array}{l l}
		1&\text{if }|\alpha|-|\gamma|+i+1<\beta,\\
		r^{\beta-i-1-|\alpha|+|\gamma|} & \text{otherwise,}
		\end{array}\right\rbrace
		\leq Cr^{\beta-1-2s+p-k},
		\end{equation}
		since $|\gamma|\leq k$ and $i\geq0$, and $\beta>1$. We postpone the analysis of the "annular" integral.
		
		When $|\gamma|<i+\beta -1 -2s+p$ write
		$$\int_{B_r(q)}(z_n)_+^{p-1}\frac{b_i(\xi,\left\langle z-x\right\rangle )}{|z-x|^{n+2s-i-1}}\cdot \rho_\zeta(z)dz = $$
		$$ = \int_{B_1}(z_n)_+^{p-1}\frac{b_i(\xi,\left\langle z-x\right\rangle )}{|z-x|^{n+2s-i-1}}\cdot \rho_\zeta(z)dz-P_2(x)x_n^{p-2s}+ $$ 
		$$+P_2(x)x_n^{p-2s} - \int_{B_1\backslash B_r(q)}(z_n)_+^{p-1}\frac{b_i(\xi,\left\langle z-x\right\rangle )}{|z-x|^{n+2s-i-1}}\cdot \rho_\zeta(z)dz,$$ 
		for some suitable polynomial $P_2$ obtained in Lemma \ref{lemmaRegularity}. 
		Notice, that the first term (the integral minus the polynomial) is smoother in $B_{1/2}$ than the derivation order, due to Corollary \ref{corollaryRegularity}, so it brings a bounded term.
		
		Let us now deal with the "annular" integrals. In the case $i>\beta-|\alpha|+|\gamma|-1$ we do the same as in \cite{AR20}, to get
		\begin{align*}
		\begin{split}
		&\left|\partial^\gamma_x\partial^{\alpha-\gamma}_\xi \int_{B_1\backslash B_r(q)}(z_n)_+^{p-1}\frac{b_i(\xi,\left\langle z-x\right\rangle )}{|z-x|^{n+2s-i-1}}\cdot \rho_\zeta(z)dz\right|\leq\\
		&\leq Cr^{\beta-i-1-|\alpha|+|\gamma|}\int_{B_1\backslash B_r(q)}(z_n)_+^{p-1}|z-x|^{-n-2s+i+1-|\gamma|}|\rho_\zeta(z)|dz\\
		&\leq Cr^{\beta-2s+p-1-k},
		\end{split}
		\end{align*}
		in view of Lemma \ref{generalisedA9} (note that $i\leq |\gamma|$).
		
		So we are left with 
		$$\partial^\alpha_x|_pI_r(x,x)- \partial^\alpha (P_1(x)+P_2(x))x_n^{p-2s} = $$
		$$= \sum_{\gamma\leq\alpha}\binom{\alpha}{\gamma}\sum_{i=0}^{\lfloor\beta\rfloor -|\alpha|+|\gamma|-1}\partial^\gamma_x\partial^{\alpha-\gamma}_\xi \int_{B_1\backslash B_r(q)}(z_n)_+^{p-1}\frac{b_i(\xi,\left\langle z-x\right\rangle )}{|z-x|^{n+2s-i-1}}\cdot \rho_\zeta(z)dz + \Theta_\alpha(r,q),$$
		where $|\Theta_\alpha(r,q)|\leq C r^{\beta-1-2s+p-k}.$ 
		
		Finally, the same steps as in \cite[Estimate (27)]{AR20} give
		\begin{align*}
		\begin{split}
		\left|\partial^\alpha I_1 (q) -  \sum_{\gamma\leq\alpha}\binom{\alpha}{\gamma}\sum_{i=0}^{\lfloor\beta\rfloor -|\alpha|+|\gamma|-1}\partial^\gamma_x|_q\partial^{\alpha-\gamma}_\xi|_q \int_{B_1\backslash B_r(q)}(z_n)_+^{p-1}\frac{b_i(\xi,\left\langle z-x\right\rangle )}{|z-x|^{n+2s-i-1}}\cdot \rho_\zeta(z)dz\right|\leq\\
		\leq C r^{\beta-|\alpha|}\int_{B_1\backslash B_r(q)}(z_n)_+^{p-1}|z-q|^{-n-2s}|\rho_\zeta(z)|dz\leq Cr^{\beta-1-2s+p-k},
		\end{split}
		\end{align*}
		where we estimated $|\rho_\zeta(z)|\leq C$ and applied Lemma \ref{generalisedA9}.
		
		Therefore we got that $L_\psi(\phi(x)) = P(x)x_n^{p-2s} + \hat{R}(x),$ with $\hat{R}$ satisfying $|D^k\hat{R}|\leq C d^{\beta-1-2s+p-k}$ for $k>\beta-1-2s+p$. But since $\phi^{-1}$ is a diffeomorphism satisfying $|D^j(\phi^{-1})|\leq C d^{\beta-j},$ when $j>\beta$, we have that 
		$$L_\psi(x) = P(\phi^{-1}(x))d^{p-2s}(x) + \hat{R}(\phi^{-1}(x)) = \varphi_\zeta(x)d^{p-2s}(x) + R_\zeta(x),$$
		where $|D^kR_\zeta|\leq C d^{\beta-1-2s+p-k},$ which we get by explicit calculation of the composite:
		$$\partial^\alpha (R\circ\phi) = \sum_{j=1}^{|\alpha|}c_j D^jR (\partial^{\alpha_1}\phi,\ldots,\partial^{\alpha_j}\phi),\quad\quad\text{where }|\alpha_i|\geq 1, \text{ and }\alpha_1+\ldots+\alpha_j = \alpha.$$
		The expression $D^jR ( v_1,\ldots,v_j)$ means evaluation of $j$-linear form on $j$ vectors.
		
		Let us now turn to the "moreover" case; when $\psi = fF_\zeta$ with $f\in C^{\beta}(B_1)$ vanishing at $0$. Then $\rho_\zeta$ is also of such form, and since $\phi$ also satisfies regularity condition \eqref{conditionLikeD}, we have that $\rho_\zeta(z) = g(z)G_\zeta(z)$, where $g$ vanishes at $0$, and $g$ and $G_\zeta$ satisfy the same conditions on growth of the derivatives as $f$ and $F_\zeta$. Next, we apply Lemma \ref{lemmaA8} to $g$, to write $g(z) = z\cdot h(z)$ for a $C^{\beta-1} (B_1,\R^n)$ vector function $h$, which by regularity condition of $g$ satisfies \eqref{conditionLikeD}. We proceed with the estimation. Note that now, we have additional $z$ in the function $\rho_\zeta$, which improves all estimates with $|q|$. We do the same steps as above with some minor differences. 
		
		In the error term, we change the estimate $|\rho_\zeta(z)|\leq C$ to $|\rho_\zeta(z)|\leq C|z|,$  which leads to the estimate
		$$Cr^{\beta-|\alpha|-2}\int_{B_r(q)}(z_n)^{p-1}|z-x|^{-n-2s+2}|z|dz \leq C|q| r^{\beta-1-2s+p-k},$$ in view of Lemma \ref{generalisedA9}.
		
		Then we want to use Lemma \ref{improvedLemmaRegularity}  instead of Lemma \ref{lemmaRegularity}, which leads to splitting cases on $|\gamma|>i+\beta-2s+p$ and $|\gamma|<i+\beta-1-2s+p$. In the estimate of $|\Gamma|$, we get additional $|q|$ from application of Lemma \ref{improvedLemmaRegularity}. 
		Note that the case $|\gamma|=i+\lfloor \beta-2s+p\rfloor $ is problematic (note that this is "never" also $i=|\gamma|$, since $\beta > 1+2s-p$ in practice). In this case, we act as usual. First we split the integral on the ball $B_1$ minus on the "annular" region (we deal with this one later). Now on $B_1$, we use the fact, that $\rho_\zeta(z) = z\cdot h(z)G_\zeta(z),$ write it in components, to end up with the integrals of the form
		$$\int_{B_1}(z_n)_+^{p-1}\frac{\tilde{b}_i(\xi,\left\langle z-x\right\rangle )}{|z-x|^{n+2s-i-1}}\tilde{\rho}_\zeta(z)z_ldz, $$
		where $\tilde{\rho}_\zeta$ still satisfies condition \eqref{conditionLikeD}.
		We proceed with splitting $z_l = (z_l-x_l)+x_l$, and treat the two cases separately. The one with $z_l-x_l$ decreases the power of the pole, which is the same as increasing $i$ by one, which we already treated. We differentiate the term with $x_l$. When we derive $x_l$, the derivative on the integral term decreases by one, and the integral is smoother than the derivation, so it gives a bounded term.	The last term which remains is 
		\begin{equation}\label{problematicChild}
		x_l\partial^\gamma \int_{B_1}(z_n)_+^{p-1}\frac{\tilde{b}_i(\xi,\left\langle z-x\right\rangle )}{|z-x|^{n+2s-i-1}}\tilde{\rho}_\zeta(z)dz,
		\end{equation}
		which we treat just like the one that lead to \eqref{GammaTerm}. Note that this time we have additional $x_l$ in the front, which improves the estimate with $|q|$.
		
		The "annular" integrals are treated similarly. We split the cases in the same way as before. When $i>\beta-|\alpha|+|\gamma| -1 $ we now estimate $|\rho_\zeta(z)|\leq C|z|$ and apply Lemma \ref{generalisedA9} to get the improved estimate with one power of $r$. Notice that we can not use that lemma only in the case $i=|\gamma|$. 
		Then we have to return to the integral on the region $B_r(q)$. We expand 
		$$b_i(\xi,\left\langle z-x\right\rangle )\cdot\rho_\zeta(z)= \sum_j \sum_l b_i^j(\xi,\left\langle z-x\right\rangle )G^j(z)h^l(z)z_l =$$ $$= \sum_j \sum_l b_i^j(\xi,\left\langle z-x\right\rangle )G^j(z)h^l(z)(x_l + (z_l-x_l)),$$
		and plug it into the integral
		$$\int_{B_r(q)}(z_n)_+^{p-1}\frac{b_i(\xi,\left\langle z-x\right\rangle )}{|z-x|^{n+2s-i-1}}\cdot \rho_\zeta(z)dz. $$
		We treat the cases as follows: when the pole, obtained after the estimate \cite[(26)]{AR20}, is integrable  we directly apply Lemma \ref{generalisedA9} to the integral in the region $B_r(q)$  and obtain additional power or $r$. When the pole is not integrable, we first make the transformation to $B_1\backslash B_r(q)$ and then apply Lemma \ref{generalisedA9}. We get additional $x_l$ in front, which brings $|q|$ to the result.
	\end{proof}

	With one last step we obtain a more concrete statement. Notice, that it is stated in a broader setting than Theorem \ref{mainResultSection2}, due to its applications in other sections. 
	
	\begin{corollary}\label{aboveCorollary}
		Let $p\in(0,2s)$ and $\beta > 1+2s-p$. Let $\partial\Omega\cap B_1$ be $C^\beta$, and let $K$ be a $C^{2k+1}(\mathbb{S}^{n-1})$ kernel satisfying \eqref{kernelConditions}, for some integer $k>\beta-1-2s+p$. Let for every $\zeta\in\partial\Omega$ the function $\eta_\zeta\in C^{\beta}(\overline{\Omega}\cap B_1)\cap C^\infty(\Omega\cap B_1)$ satisfy 
		$$|D^j\eta_\zeta(x)|\leq C_j d^{\beta-j},\quad\quad j>\beta\quad \text{and}$$
		$$ \left|\left|D^j\eta_\zeta\right|\right|_{C^\beta_\zeta}\leq C_0,\quad j\geq 0.$$
		
		Then, $L(\eta_\zeta d^p) = \varphi_\zeta d^{p-2s} + R_\zeta$, where $\varphi_\zeta$ is composition of the boundary flattening map with a polynomial whose coefficients'  $C^\beta_\zeta$ norm is bounded with a constant depending on $C_0$. We have $\varphi_\zeta\in C^{\beta}(\overline{\Omega}\cap B_{1/2})$ with $||\varphi_\zeta||_{C^{\beta}(\overline{\Omega}\cap B_{1/2})}\leq ||\eta_\zeta||_{C^\beta(B_1)}$, and $R_\zeta\in C^{\beta-1-2s+p}(\overline{\Omega}\cap B_{1/2})$, with 
		$$|D^kR_\zeta(x)|\leq C d^{\beta-1-2s+p-k}(x),\quad\quad \text{in }\overline{\Omega}\cap B_{1/2}.$$
		
		Moreover, if at some boundary point $z$ we have $\eta_\zeta(z)=0$, the estimate on $R_\zeta$ improves to
		$$|D^kR_\zeta(x)|\leq C|x-z| d^{\beta-1-2s+p-k}(x),\quad\quad \text{in }\overline{\Omega}\cap B_{1/2}.$$
		The constant $C$  depends only on $k$, $s$, $||K||_{C^{2k+1}(\S^{n-1})}$, $\beta$ and $||\partial\Omega||_{C^\beta}.$
	\end{corollary}
	
	\begin{proof}
		We rewrite $L(\eta_\zeta d^p)$ in the same way as in \cite[Proof of Corollary 2.3]{AR20}:
		\begin{align*}
		L(\eta_\zeta d^p)(x) =& -\frac{1}{2s}p.v.\int_{\R^n} \eta_\zeta(y)\nabla (d^p)(y) \cdot (y-x) K(y-x)dy\\
		& - \frac{1}{2s}p.v.\int_{\R^n}d^p(y)\nabla \eta_\zeta (y)\cdot (y-x)K(y-x)dy \\
		=& p.v.\int_{\R^n} d^{p-1}(y)K(y-x)\psi_\zeta(y)\cdot (y-x)dy,
		\end{align*}
		where we denoted 
		$$\psi_\zeta= -\frac{p}{2s}\eta_\zeta\nabla d - \frac{1}{2s}d\nabla\eta_\zeta.$$
		By assumption on $\eta_\zeta,$ the vector function $\psi_\zeta\in C^{\beta-1}(\overline{\Omega}\cap B_1)$, smooth in the interior, and satisfies 
		$|D^j\psi(x)|\leq C_j d^{\beta-1}(x),$ for $j>\beta-1$, as well as $\beta$ regularity of all the derivatives with respect to $\zeta.$
		
		Therefore splitting the integral above in the regions $B_1$ and $\R^n\backslash B_1$, we deduce the result from the above theorem.
		
		Note that the moreover case in the theorem gives the moreover part of the corollary, since both $d$ and $\eta_\zeta$ are $C^\beta$ and vanish at $z$.
	\end{proof}
	
	With not much work, we now prove Theorem \ref{goalOfSection2}.
	
	\begin{proof}[Proof of Theorem \ref{goalOfSection2}] 
		It is a special case of the above corollary.
	\end{proof}

%	\begin{corollary}
%		When all the above assumptions hold, but  $1<\beta<1+2s-p$, we still get the estimate
%		$$|R(x)|\leq C d^{\beta-1-2s+p}(x),\quad\quad x\in \overline{\Omega}\cap B_{1/2},$$
%		or 
%		$$|R(x)|\leq C |x-z|d^{\beta-1-2s+p}(x),\quad\quad x\in \overline{\Omega}\cap B_{1/2},$$
%		when $\eta(z)=0$.
%	\end{corollary}

	To complete the result, we show how to generalise the statement to the powers of distance greater than $2s$. Thanks to the clear representation of the function $\varphi$ and the improvement of the estimate in the "moreover" cases above, we are able to deduce the corollary below. In short words, we treat the surplus $d^{\lfloor p\rfloor}$ as a part of the function $\eta$. That is the reason, why we allow the function $\eta$ to be less smooth than $C^\infty$. Note that the condition $\langle p \rangle<2s$ does not play any role in the set-up $s>1/2$.
	
	\begin{corollary}\label{surjectivityLemma}
		Let $\partial\Omega\cap B_1$ be $C^\beta$, $\zeta\in\partial\Omega\cap B_1$, $\eta$ be a polynomial, 
		and $p\in\R$ with $p> 2s$,  $p-2s\not\in\N$, 
		$p\leq \lfloor \beta\rfloor+ 2s$. Additionally, if $s\leq 1/2$, then we need $p\not\in\N$ and $\left\langle p\right\rangle <2s$. Assume $K\in C^{2\beta+1}(\mathbb{S}^{n-1})$ is the kernel of $L$ satisfying \eqref{kernelConditions}.
		
		Then
		$$L(\eta(\cdot - \zeta) d^p)(x) = \varphi_\zeta(x) d^{p-2s}(x) + R_\zeta(x),$$
		where $\varphi_\zeta\in C^\beta(\overline{\Omega})$  whose all derivatives  in variable $\zeta$ are of bounded $C^\beta_\zeta$ norm. Furthermore, we have
		$||\varphi_\zeta||_{C^\beta(\overline{\Omega}\cap B_{1/2})}\leq C ||\eta(\cdot-\zeta)||_{C^\beta(\overline{\Omega})}$, and $R_\zeta\in C^{\beta-1+\left\langle p-2s \right\rangle}(\overline{\Omega}\cap B_{1/2})$.
	\end{corollary}
	\begin{proof}
		%We split the cases as follows. When $2s>\left\langle p\right\rangle > 2s-1$, w
		We apply the above corollary on $\tilde{\eta}_\zeta = \eta(\cdot-\zeta) d^{\lfloor p\rfloor }.$ We get 
		$L(\eta(\cdot-\zeta) d^p) = L(\tilde{\eta}_\zeta d^{\left\langle p\right\rangle })= \tilde{\varphi}_\zeta d^{\left\langle p\right\rangle -2s} + R_\zeta.$ Note that $\lfloor p\rfloor\geq 1$, and so $\tilde{\eta}_\zeta$ vanishes at every boundary point, so we can use the moreover case of the above corollary everywhere. Let us also stress that $\tilde{\eta}_\zeta$ satisfies the condition \eqref{conditionLikeD}, since $\eta$ is a polynomial, $\zeta$ varies over the $C^\beta$ boundary, and $d$ satisfies the same condition. Hence $\tilde{\varphi}_\zeta$ is a composition of a polynomial whose coefficients are $C^\beta_\zeta$, with a boundary flattening map. 
		We now argue that $R_\zeta\in C^{\beta-2s+\left\langle p \right\rangle}$. We choose $|x_0-z|=2r$, $x,y\in B_r(x_0)$ and $|\gamma| = \lfloor \beta-2s+\left\langle p \right\rangle\rfloor $. Compute 
		\begin{align*}
			|\partial^\gamma R_\zeta(x)-\partial^\gamma R_\zeta(y)| &\leq  ||D^{|\gamma|+1}R_\zeta||_{L^\infty(B_r(x_0))}|x-y| \\
			&\leq C |x-z|r^{ \beta-2s+\left\langle p \right\rangle-1-|\gamma|-1} r^{1-\left\langle \beta-2s+\left\langle p \right\rangle\right\rangle}|x-y|^{\left\langle \beta-2s+\left\langle p \right\rangle\right\rangle}\\
			&\leq C |x-y|^{\left\langle \beta-2s+\left\langle p \right\rangle\right\rangle},
		\end{align*}
		where we used the moreover case of Corollary \ref{aboveCorollary}.
		
%		In the case $\left\langle p\right\rangle < 2s-1,$ we need to apply the corollary on $\tilde{\eta}_\zeta = \eta(\cdot-\zeta) d^{\lfloor p\rfloor-1 }.$ We get 
%		$L(\eta(\cdot-\zeta) d^p) = L(\tilde{\eta}_\zeta d^{\left\langle p\right\rangle+1 })= \tilde{\varphi}_\zeta d^{\left\langle p\right\rangle+1 -2s} + R_\zeta.$
%		Since $p>2s$, the condition $\left\langle p\right\rangle < 2s-1$ implies $p\geq 2$, so we can use the moreover case again, to get
%		$R_\zeta\in C^{\beta-(2s-1)+\left\langle p \right\rangle}(\overline{\Omega}\cap B_{1/2}).$ (The case $p\in\N$, works only when $p>1/2$. Then we split $p=(p-1) + 1$, and do as in the first case above.)
		
		We can write both regularities as $R_\zeta\in C^{\beta-1+\left\langle p-2s \right\rangle}(\overline{\Omega}\cap B_{1/2}).$
		
		To extract the correct power of $d$ on the right-side, proceed as follows. For simplicity we work with $0$ as the boundary point. We use \cite[Lemma 3.5]{AR20} to do blow-ups on every compact ball inside $\Omega$. 
		%Let us treat only the case when $\left\langle p\right\rangle >2s-1$. Other cases are done in a similar manner. 
		We start with defining $u_r(x) = \frac{1}{r^{\left\langle p\right\rangle }}\eta(rx-\zeta)d^p(rx),$ for $r<1$. We have that $u_r$ converge to $0$ in $L^\infty_\loc(\R^n)$. We can also estimate $|\eta(rx-\zeta)d^p(rx)|\leq C|x|^p$, which implies the convergence of $\int_{\R^n} \frac{|u_r(x)|}{1+|x|^{n+2s+k}}dx$ to $0$, if we choose $k=\lceil p-2s\rceil$. On every compact ball $B$ inside $\Omega$ we have that
		$$L(u_r)(x)-P_r(x) = \tilde{\varphi}_\zeta(rx) \frac{1}{r^{\left\langle p\right\rangle-2s}}d^{\left\langle p\right\rangle-2s}(rx) +  \frac{1}{r^{\left\langle p\right\rangle-2s}} (R(rx)-\tilde{P}_r) \longrightarrow \tilde{\varphi}_\zeta(0) (x\cdot\nu)_+^{\left\langle p\right\rangle-2s}$$
		in $L^\infty(B),$ for $\tilde{P}_r$ being the suitable Taylor polynomials of $R, $ of order $k-1$, and $\nu$ being the unit normal of $\Omega$ at zero. Therefore, the lemma gives us that 
		$$0 = L(0) \stackrel{k}{=} \tilde{\varphi}_\zeta(0) (x\cdot\nu)_+^{\left\langle p\right\rangle-2s} \quad\quad \text{in }B,$$
		which implies that $\tilde{\varphi}_\zeta(0)=0$. Now we change the order of blow-up to $1+\left\langle p\right\rangle $ to conclude also that $\nabla \tilde{\varphi}_\zeta(0) = 0$. We can go on with the procedure as long as the order of blow-up is strictly lower than $p$, so that we have the local uniform convergence of $u_r$ to zero. In the last step the order of blow up has to be taken $p-\varepsilon$, so that $p-\varepsilon<p$, and $p-\varepsilon> \lfloor p\rfloor$, to get that the Taylor polynomial of $\tilde{\varphi}_\zeta$ of order $\lfloor p\rfloor$ vanishes. 
		We also need $p -2s-\varepsilon < \beta -1 + \left\langle p-2s\right\rangle $ (blow-up order has to be smaller than the regularity of $R_\zeta$), so that the blow up of $R_\zeta$ vanishes, after subtracting suitable Taylor polynomials. Hence, when $p\not\in\N$, we conclude that $\tilde{\varphi}_\zeta$ has zero of order $\lfloor p\rfloor $ at every boundary point. This means that before the pre-composition with the boundary-flattening map, $\tilde{\varphi}_\zeta$ was divisible by $x_n^{\lfloor p\rfloor }$, and all the coefficients were bounded with $||\eta(\cdot-\zeta)||_{C^{\beta}}$, and $C^\beta$ in variable $\zeta$. Therefore, if we set $\varphi_\zeta = \frac{\tilde{\varphi}_\zeta}{d^{\lfloor p\rfloor }}$ it is still a $C^\beta(\overline{\Omega}\cap B_{1/2})$ function which satisfies $||\varphi_\zeta||_{C^{\beta}}\leq C ||\eta(\cdot-\zeta)||_{C^{\beta}}$.
		
		Notice also, that in the case $p\in\N$ and $s>1/2$, we can do the above procedure with taking one power of $d$ less into $\eta$. So then we work with $L((\eta(\cdot - \zeta) d^{ p-1}) d^{1}),$ and we obtain the same result.
		%In other cases we proceed analogously, to prove the desired result
	\end{proof}
	
	\begin{remark}
		Note that in the case when $p\in\N+s$, the polynomial part falls off, so we have that $L(\eta d^p) \in C^{\beta - s}(\overline{\Omega}),$ together with the suitable estimate.
	\end{remark}

	It seems that the regularity of the reminder could be improved to $\beta-1+p-2s$. To establish it, we would need to improve the estimates in Lemma \ref{lemmaRegularity}, assuming $\psi$ has a zero of some higher order at $0$. Since the proof is already very cumbersome, and what we have proven is enough for our purposes, we do not dig into the improvement.

	\section{Equations for non-local operators}

	\subsection{Existence and computation}
	
	In this subsection we aim to answer the question of solvability  of $L(Pd^p ) = q d^{p-2s}$, where $q$ is a known polynomial. Briefly, the answer consists of two steps. First we perform a blow-up argument using \cite[Lemma 3.5]{AR20}, which reduces the problem to the flat case. Then we investigate the flat case, which is explicitly computable.
	
	We start with computation in the flat case. We need a preliminary result, which says when can we evaluate $L$ on a function with polynomial growth in the generalised sense, recall Definition \ref{definitionOfGeneralDistance}.
	
	\begin{lemma}\label{existance}
		Let $L$ be an operator whose kernel $K$ satisfy \eqref{kernelConditions} and is $C^{k+1}(\mathbb{S}^{n-1})$. Let for some $\varepsilon>0$ a function $u\colon \R^n\to\R$ be $C^{2s+\varepsilon}(B_2)$ with $$\left|\left|\frac{u}{1+|x|^{k+s+\alpha}}\right|\right|_{L^\infty(\R^n)}<\infty,$$  for some $\alpha<s$.
		Then there exists a function $f\colon B_1\to\R$ and polynomials $P_R\in\textbf{P}_{k-1}$, so that $L(u\chi_R) -P_R\to f$ in $L^\infty(B_1)$, with 
		$$||f||_{L^\infty(B_1)}\leq C\left(\left|\left|\frac{u}{1+|x|^{k+s+\alpha}}\right|\right|_{L^\infty(\R^n)} + ||u||_{C^{2s+\varepsilon}(B_2)}\right).$$ 		
	\end{lemma}
	
	\begin{proof}
		Denote $\tilde{u} = u\chi_2$. Then we can compute $L\tilde{u} = \tilde{f}$, and we have $||\tilde{f}||_{L^\infty(B_1)}\leq C(||u||_{C^{2s+\varepsilon}(B_2)})$. For $R>2$ we have $L(u\chi_R) = L(u\chi_2)+ L(u\chi_{B_R\backslash B_2}) = \tilde{f}+g_R$. Choose now $\gamma\in\N_0^n$ with $|\gamma|=k$. Then
		$$|\partial^\gamma g_R(x)| = |\partial^\gamma\int_{B_R\backslash B_2}u(z)K(z-x)dz|=|\int_{B_R\backslash B_2}u(z)\partial^\gamma K(z-x)dz|$$
		$$\leq C \int_{\R^n\backslash B_2}\frac{u(z)}{|z-x|^{n+k+2s}}dz
		\leq C\left|\left|\frac{u}{1+|x|^{k+s+\alpha}}\right|\right|_{L^\infty(\R^n)}<\infty.$$
		Also, the same estimate with the dominated convergence theorem gives that $\partial^\gamma g_R(x)\rightarrow f_\gamma$ as $R\to\infty,$
		where we denoted $f_\gamma(x)= \int_{\R^n\backslash B_2}u(z)\partial^\gamma K(z-x)dz$. (This convergence is uniform for $x\in B_1$, which we can check with an $\varepsilon$, $R_0$ calculus, and some estimates of the kernel.) With another dominated convergence argument we see, that whenever $\gamma + e_i = \gamma'+e_j$ then $\partial_if_\gamma = \partial_j f_{\gamma'}$. Since $B_1$ is simply connected, we can integrate $k$ times to get a function $f_0$ such that $\partial^\gamma f_0 = f_\gamma.$ Denote with $T^j\phi$ the $j$-th Taylor polynomial of $\phi$ centred at $0$. Then 
		$$||g_R +\tilde{f}- T^{k-1}g_R - f_0+T^{k-1}f_0-\tilde{f}||_{L^\infty(B_1)} \leq ||D^{k}(\ldots)||_{L^\infty(B_1)} \rightarrow 0,$$
		and so $f$ can be taken as $f_0+\tilde{f}$. 
		Since $||f_0||_{L^\infty(B_1)}\leq C\left|\left|\frac{u}{1+|x|^{k+s+\alpha}}\right|\right|_{L^\infty(\R^n)}$ by construction, we have
		$$||f||_{L^\infty(B_1)}\leq C\left(\left|\left|\frac{u}{1+|x|^{k+s+\alpha}}\right|\right|_{L^\infty(\R^n)} + ||u||_{C^{2s+\varepsilon}(B_2)}\right)$$
		as wanted.
	\end{proof}

	\begin{remark}\label{existenceRemark}
		Note that if $u$ is smooth as the growth (in $B_2$), the polynomials $P_R$ can be taken as the Taylor polynomial of $Lu\chi_R$ of the correct order centred at $0$. If we take a smooth cut-off $\chi_2$ above, then $\tilde{f}$ becomes smooth enough. Note also, that if additionally for some positive $\varepsilon$ we have $u\in C^{k+1+2s+\varepsilon}(B_2)$, the convergence is actually happening in $C^{k+1}(B_1)$.
	\end{remark}
	
	Next, we compute $L(Pd^p)$ in the flat case, for any polynomial $P$. All the equations should be understood in the sense of Definition \ref{definitionGeneralised}, when the function inside the operator grows too much at infinity. Let $\Omega = \left\lbrace x_n>0\right\rbrace$. Then $d(x) = \max\left\lbrace x_n,0\right\rbrace = (x_n)_+$.

	Because $L$ is linear, it is enough to compute $L(M(x_n)_+^p)$ for every monomial $M$. 	
	
	\begin{proposition}
		Suppose $p-2s\not\in\N$. Let $\gamma \in \N^{n-1}\times 0$ and $k>|\gamma|+p-2s$. If $L$ is an operator whose kernel is $C^k(\mathbb{S}^{n-1})$ and satisfies \eqref{kernelConditions}, we have that $$L\left(x^\gamma(x_n)_+^p\right) \stackrel{k}{=} \sum_{\alpha\leq\gamma}c_{\gamma,\alpha,s,p}x^\alpha x_n^{p-2s+|\gamma-\alpha|} \quad \text{ in } \left\lbrace x_n>0\right\rbrace .$$
		Furthermore, if and only if $p-s\not\in\N$, then the coefficient $c_{\gamma,\gamma,s,p}\neq0$.
	\end{proposition}
	
	\begin{proof}
		We prove it with induction on $|\gamma|$. When $|\gamma|=0$, we have that $L((x_n)_+^p)$ is a $p-2s$-homogeneous function dependent only on $x_n$. Hence it equals $c_{s,p}x_n^{p-2s}.$ Additionally, when $p-s\not\in\N$, then $c_{s,p}\neq 0$, which we conclude from the Liouville theorem \cite[Theorem 3.10]{AR20}. 
		
		Take now $|\gamma|>0$. Let $\gamma_i>0$. Differentiating on $x_i$ and using the induction we get 
		$$\partial_iL\left(x^\gamma (x_n)_+^p\right) =\gamma_iL\left(x^{\gamma-e_i}(x_n)_+^p\right) = \gamma_i \sum_{\alpha\leq\gamma-e_i}c_{\gamma-e_i,\alpha,s,p}x^\alpha x_n^{p-2s+|\gamma-e_i-\alpha|}.$$ Integrating back we get
		\begin{align*}
			L\left(x^\gamma (x_n)_+^p\right)& = \sum_{\alpha\leq\gamma-e_i}\frac{\gamma_i c_{\gamma-e_i,\alpha,s,p}}{\alpha_i+1}x^{\alpha+e_i} x_n^{p-2s+|\gamma-e_i-\alpha|}+f(x_1,\ldots,\hat{x_i},\ldots,x_n)\\
			& = 		 \sum_{e_i\leq\alpha\leq\gamma}\frac{\gamma_i c_{\gamma-e_i,\alpha-e_i,s,p}}{\alpha_i}x^{\alpha} x_n^{p-2s+|\gamma-\alpha|}+f(x_1,\ldots,\hat{x_i},\ldots,x_n),
		\end{align*}
		for some function $f$ independent of $x_i$. Differentiating on other $x_j,$ for $ j<n$ and comparing the terms, we get
		$$L\left(x^\gamma (x_n)_+^p\right) =		 \sum_{0<\alpha\leq\gamma}c_{\gamma,\alpha,s,p}x^{\alpha} x_n^{p-2s+|\gamma-\alpha|}+f(x_n)=:h(x)+f(x_n).$$
		
		This means that for some polynomials $P_R$ of degree  $\lfloor|\gamma|+p-2s\rfloor$, we have
		$$L(x^\gamma (x_n)_+^{p}\chi_{B_R}) (x)- P_R(x) \longrightarrow h(x)+f(x_n).$$ If $|\gamma|+p-2s$ is a natural number, the polynomials are of one degree bigger.
		Let us now use the homogeneity of $L(x^\gamma (x_n)_+^{p}\chi_{B_R}) (x)$ in $(x,R)$ and of $h$. Concretely, evaluate the above formula at $\lambda R$ and $\lambda x$. We get
		$$f(\lambda x_n) + h(\lambda x) = \lim\limits_{R\to\infty} L(x^\gamma (x_n)_+^{p}\chi_{B_{\lambda R}}) (\lambda x)- P_{\lambda R}(\lambda x) $$
		$$=
		\lambda^{|\gamma|+p-2s}\lim\limits_{R\to\infty} L(x^\gamma (x_n)_+^{p}\chi_{B_R}) (x)  - P_R(x) + \lim\limits_{R\to\infty} \lambda^{|\gamma|+p-2s}P_R(x) - P_{\lambda R}(\lambda x) $$
		$$ = \lambda^{|\gamma|+p-2s} (f(x_n)+h(x)) + P_\lambda (x),$$
		for some polynomial $P_\lambda$ of the same degree. But since $h$ is homogeneous, this implies
		$$f(\lambda x_n) = \lambda^{|\gamma|+p-2s}f(x_n) + P_\lambda(x),$$
		and hence $P_\lambda$ depends only on $x_n$. Differentiating it $\lceil|\gamma|+p-2s\rceil$, we conclude that $f$ equals a sum of a polynomial and a $(|\gamma|+p-2s)$-homogeneous function. This proves the claim, because the equalities in the generalised sense hold up to polynomials of this order. Here we need, that $|\gamma|+p-2s$ is not a natural number, since otherwise we could get logarithmic terms. 
		
		Finally, differentiating the result $\gamma$-times, we get the desired inequality (See \cite[Theorem 3.10]{AR20}). 
	\end{proof}

	Observing the outcome, we deduce the following.
	
	\begin{corollary}\label{computationFlat}
		When $p-2s\not\in\N$, for every polynomial $P$ we have $L\left(P\cdot (x_n)_+^p\right) = qx_n^{p-2s}$ for some polynomial $q$ of the same degree. Furthermore, if $P$ is homogeneous, so is $q$, and if $p-s\not\in\N$, then if $P$ is non-zero, also $q$ is.
	\end{corollary}
\begin{proof}
	The corollary is a direct consequence of the above proposition.
\end{proof}
	\begin{remark}
		The mapping $P\mapsto q$ is a linear map, let us denote it $\Phi_{e_n}$. When $p-s\not\in\N$, it is also bijective. If we order the monomials first by order and then lexicographically, we get a lower triangular matrix with non-zero diagonals.
		\footnote{Note also, that if $p-2s = m\in\N$, we can get terms of the form $f(x) = x^m\log x,$ since it satisfies the condition 
			$$f(\lambda x) = (\lambda x)^m\log(\lambda x) = \lambda^m x^m\log x + \lambda^m\log \lambda \cdot x^m = \lambda^n f(x) + P_{\lambda}(x).$$
		This is why we need $p-2s\not\in\N$ in the assumptions.}
	\end{remark}
	
	Let us connect the above computation with the computation of $L(P\cdot(xe)_+^p)$ in $\left\lbrace xe>0\right\rbrace $, for a polynomial $P$ and a unit vector $e$. For a polynomial $P$ we denote $\Phi_e (P):= q$, when $L(P(xe)_+^p) = q (xe)^{p-2s}$ in $\{xe>0\}.$ Similarly, for a linear map $Q$, we denote $\Phi_e^Q (P):= q$, when $L^Q(P(xe)_+^p) = q (xe)^{p-2s}$ in $\{xe>0\},$ where the kernel of  $L^Q$ is the kernel of $L$ pre-composed with $Q$.
	
	\begin{lemma}\label{orthogonalTransformationLemma}
		Let $p-2s\not\in\N$, let $L$ be an operator whose kernel $K$ satisfies the condition \eqref{kernelConditions} and let $e\in \S^{n-1}.$ Let $Q$ be an orthogonal matrix which maps $e_n$ into $e$. Then the following formula holds:
		\begin{equation*}\label{linearMap}
		\Phi_e^I(P) = \Phi_{e_n}^Q(P\circ Q)\circ Q^T,
		\end{equation*} 
		where $\Phi_e^I$ and $\Phi_{e_n}^Q$ are as described above.
	\end{lemma}

	\begin{proof}		
	In general we have
	\begin{align*}
		Lu(x) &= \int_{\R^n} (u(x)-u(x+y))K(y)dy = \int_{\R^n}\big (u(QQ^Tx)-u(Q(Q^Tx+Q^Ty)\big)K(QQ^Ty)dy\\
		&=\int_{\R^n}\big(u\circ Q(Q^Tx) u\circ Q(Q^Tx+y)\big)K\circ Q(y)dy = L^{Q}(u\circ Q)(Q^Tx),
	\end{align*}
	where the kernel of the operator $L^{Q}$ is $K\circ Q$. To get the same equality for the equation for function with polynomial growth, we just use the above equality on every ball $B_R$.
	 
	Let now $Q$ be an orthogonal matrix whose last column equals $e$, meaning that we find an orthonormal basis of $\R^n$ so that the last vector is $e$. Then $Qxe = xQ^Te = x e_n$. So if we denote $u(x) = P(x)(xe)_+^p$ then $u\circ Q (x) = P\circ Q (x) (xe_n)_+^p.$ Since $Q$ is linear, $L^{Q}$ has the same properties as $L$ needed for computation of $L(P(x_n)_+^p)$, we conclude 
	$$L(P(xe)^p_+)(x) = L^{Q}  (P\circ Q (x_n)_+^p)(Q^Tx) = \Phi^Q_{e_n}(P\circ Q) (Q^Tx)\cdot (Q^Txe_n)^{p-2s} =q(x)(xe)^{p-2s}$$
	where $q$ is still polynomial, because pre-composition with linear map preserves polynomials. Therefore we have
	\begin{equation*}
	\Phi_e^I(P) = \Phi_{e_n}^Q(P\circ Q)\circ Q^T,
	\end{equation*} 
	as wanted.
\end{proof}
	
	We are now in position to state the result, which answers the question of the beginning of this subsection. The result is of great use when proving the main statement of the paper, Theorem \ref{1.1}.
	
	\begin{theorem}\label{surjectivityResult}
		Assume $L$ is an operator with kernel $K$ satisfying \eqref{kernelConditions}. Let $\Omega\subset \R^n$ be a domain with $\partial\Omega\cap B_1\in C^\beta$ and $p>\varepsilon_0$, $p-2s\not\in\N$, $p-s\not\in\N$. Let for every $z\in \partial\Omega\cap B_1$ we have a polynomial $Q_z\in \textbf{P}_{\lfloor \alpha\rfloor}$, for some $\alpha\leq\beta-1$, such that the coefficient in front of $x^\gamma$ is $C^{\alpha-|\gamma|}$ as a function in $z$ with
		$\left|\left|(Q_z)^{(\gamma)}\right|\right|_{C^{\alpha-|\gamma|}_z(\partial\Omega\cap B_1)}\leq C_0$. Let also $\lfloor \alpha \rfloor+\lfloor p-2s\rfloor\lor 0 < \beta-1.$
		
		Then there exists a polynomial $\tilde{Q}_z\in\textbf{P}_{\lfloor \alpha\rfloor},$ with  $\left|\left|(\tilde{Q}_z)^{(\gamma)}\right|\right|_{C^{\alpha-|\gamma|}_z(\partial\Omega\cap B_1)}\leq CC_0$, so that
		$$L(\tilde{Q}_z(\cdot-z)d^p) = Q_z(\cdot-z)d^{p-2s} + R_z+\eta_z,$$
		where the function $R_z\in C^{\beta-1-\langle p-2s\rangle - \textbf{1}_{(p<2s)}}(\overline{\Omega}\cap B_{1/2})$ and $\eta_z\in C^\beta(\overline{\Omega}\cap B_{1/2})$ with $|\eta_z(x)|\leq C |x-z|^{\lfloor \alpha\rfloor + 1}$.
	\end{theorem}

	\begin{proof}
		We analyse the map $\phi_z\colon \textbf{P}_{\lfloor \alpha\rfloor}\to\textbf{P}_{\lfloor \alpha\rfloor},$ defined as $\Phi_z(P) = q$, when $L(P(\cdot-z)d^p) = \varphi_z d^{p-2s}+ R_z$, as before (see Corollary \ref{surjectivityLemma}) and $q = T^{\lfloor \alpha\rfloor}_z(\varphi_z).$ The map $\Phi_z$ is linear. 
		
		Let us show that $\Phi_z$ is surjective. Therefore we choose a multi-index $\gamma$ of order less or equal than $\lfloor\alpha\rfloor$. 
		With straight forward computation we see that
		$$L((\cdot-z)^\gamma d^p)(x) = L((\cdot)^\gamma d^p_z)(x-z) = \varphi_\gamma(x-z) d^p_z(x-z)+R_\gamma(x-z),$$
		where $d^p_z(x) = d^p(x+z).$ We want to perform a blow-up of the function $(x)^\gamma d^p_z(x)$, to get the $|\gamma|-$th Taylor polynomial of $\varphi_\gamma.$ We use \cite[Lemma 3.5]{AR20}. 
		
		Define $u_r(x) = \frac{1}{r^{|\gamma|+p}}u(rx)$, where $u(x) = x^\gamma d_z^p(x)$. We have the $C^p_\loc (\R^n)$ convergence of $u_r \to u_0$, for $u_0 = x^\gamma (x\nu_z)^p.$ 
		Together with the estimate $|u_r(x)|\leq C |x|^{|\gamma|}|x|^p$, we get all the assumptions of the lemma, with $k' = \lceil |\gamma|+p\rceil$.
		True: $u_r\to u_0$ in $L^\infty_\loc$ follows from the convergence above and $\int_{\R^n } \frac{|u_r|}{1+|x|^{n+2s+k}}<C$ independently or $r$, due to the growth estimate. Consequently, $\int_{\R^n}\frac{|u_r-u_0|}{1+|x|^{n+2s+k}} \to 0$ follows from both the growth estimate and the convergence; we obtain it with splitting the integral on $B_R$ and the complement. On $B_R$ we have uniform convergence, on the complement the integral is small.   
		Taking into account the right-hand side as well, we get 
		$$L(u_r)(x) = \frac{1}{r^{|\gamma|}} \varphi_\gamma(rx) \frac{1}{r^{p-2s}} d^{p-2s}(rx) + \frac{1}{r^{|\gamma|+p}} R_\gamma(rx).$$
		So for suitable rescaled Taylor polynomials of $R_\gamma$ of order $\lceil|\gamma|+p-2s\rceil$, we get that $L(u_r)-P_r$ converges in $L^\infty_\loc(\R^n)$, similarly as in the proof of Corollary \ref{surjectivityLemma}.
		
		The lemma gives that $L(u_0) \stackrel{k}{=} T_0^{|\gamma|}(\varphi_\gamma) (x\nu_z)_+^{p-2s}$, where $\nu_z$ is the unit normal to $\partial\Omega$ at $z$. Now we use Corollary \ref{computationFlat} and Lemma \ref{orthogonalTransformationLemma}, to get that
		$$T^{|\gamma|}_0(\varphi_\gamma) = \phi^I_{\nu_z}(x^\gamma) = \Phi^Q_{e_n}(\varphi_\gamma\circ Q)\circ Q^T,$$
		for any orthogonal matrix $Q$ mapping $e_n$ into $\nu_z$.
		Since $T_0^{|\gamma|}(\phi_\gamma)$ is the projection of $\Phi_z(x^\gamma)$ on $\textbf{P}_{|\gamma|}$, we deduce, that $\Phi_z$ is in some basis a lower triangular operator (with ordering the monomials first by homogeneity and then lexicographically), and surjective (due to the assumption $p-s\not\in\N$) - and hence bijective, since the dimensions of the domain and codomain agree. 
		
		Let us now turn to the regularity of entries of $\Phi_z$. Since $P(x-z)$ is a polynomial in $x$, all its derivative are $C^\beta$ smooth in $z$. Therefore, Corollary \ref{surjectivityLemma} renders that when $L(P(\cdot-z)d^p) = \varphi_z d^{p-2s}+R_z$, the derivatives of $\varphi_z$ at any point are $C^\beta$ in $z$ as well. Since $\varphi_z$ is a $C^\beta$ function, then $\partial^\gamma \varphi_z(z)$ is $C^{\beta-|\gamma|}$ regular in $z$. Therefore the coefficient of $\Phi_z$ in the $\gamma$-th row is $C^{\beta-|\gamma|}_z$. Now we define the inverse map, 
		$$\Psi_z := \Phi_z^{-1},$$
		which has the same properties, since $\Phi_z$ is a lower triangular and surjective. Note that the entry of inverse of the lower diagonal matrix depends only on the entries which lie in the rows of lower or equal index in the original matrix.
		
		Finally, take $Q_z$ as in the assumptions. We define $\tilde{Q}_z := \Psi_z(Q_z)$ and $\eta_z := \varphi_z - T^{\lfloor\alpha\rfloor}_z(\varphi_z),$ where
		$$L(\tilde{Q}_z(\cdot -z)d^p) = \varphi_z d^{p-2s}+R_z,$$
		which proves the claim.
	\end{proof}

	\subsection{Regularity estimates}
	In the second part of this section, we establish some regularity results for equations with polynomial growth, where the right-hand side explodes at the boundary. The ideas are  taken from the regularity results in \cite{RS16,AR20,AuR20}.
	
	The strategy for treating functions with polynomial growth is often through cut-off. Then we have a term which is compactly supported and another one which grows at infinity, but vanishes around the point of evaluation. We start with the cut-off lemma, which is a generalisation of \cite[Lemma 3.6]{AR20}. 

	\begin{lemma}\label{generalised3.6}
		Assume $U\subset B_1$ is a $C^{\beta}$ domain with $\beta>1$. For some $k\in\N$ let the kernel $K$ of operator $L$ be $C^{k+1}(S^n)$ and satisfy condition \eqref{kernelConditions}. Let $u$ be the solution of 
		$$\left\lbrace \begin{array}{rcl l}
		Lu&\stackrel{k}{=}&f& \text{in }U,\\
		\displaystyle\left|\left|\frac{u}{1+|\cdot|^{k+s+\alpha}}\right|\right|_{L^\infty(\R^n)}&\leq &C&\text{with }\alpha<s,
		\end{array}\right. $$
		with $|f|\leq K_0 d^{\varepsilon-2s}$, for some $\varepsilon\in (0,s)$. 
		
		Then the function defined as 
		$$\tilde{u}:= u\chi_{B_2},$$
		satisfies $L\tilde{u}=\tilde{f}$, with
		$$||\tilde{f}d^{2s-\varepsilon}||_{L^\infty(U)}\leq C\left(||fd^{2s-\varepsilon}||_{L^\infty(U)} + \left|\left|\frac{u}{1+|\cdot|^{k+s+\alpha}}\right|\right|_{L^\infty(\R^n)} \right).$$
	\end{lemma}
	\begin{proof}
		We compute
		$$L\tilde{u}\stackrel{k}{=}Lu-L\hat{u},$$
		for $\hat{u}:=u\chi_{B_2^c}$. From Lemma \ref{existance}, and Remark \ref{existenceRemark}, we know that  $L\hat{u}\stackrel{k}{=}\hat{f}$, for some $\hat{f}\in C^{k+1}(U),$ satisfying
		$$||\hat{f}||_{L^\infty(U)}\leq C \left|\left|\frac{u}{1+|\cdot|^{k+s+\alpha}}\right|\right|_{L^\infty(\R^n)}.$$
		Therefore,
		$$L\tilde{u}\stackrel{k}{=}f-\hat{f}.$$ 
		But since $\tilde{u}$ is compactly supported it follows from the definition of evaluation of $L$ in the generalised sense, that 
		$$L\tilde{u} = f-\hat{f}-P,$$
		for some polynomial $P$ of degree $k$.
		
		Now, we split $\tilde{u} = u_1+u_2+u_3$, so that 
		$$\left\lbrace \begin{array}{rl l}
		L(u_1)=&f-\hat{f}& \text{in }U,\\
		u_1=&0&\text{in } U^c,
		\end{array}\right.  \quad 
		\left\lbrace \begin{array}{rl l}
		L(u_2)=&0& \text{in }U,\\
		u_2=&\tilde{u}&\text{in } U^c,
		\end{array}\right.\quad\text{and}\quad
		\left\lbrace \begin{array}{rl l}
		L(u_3)=&P& \text{in }U,\\
		u_3=&0&\text{in } U^c.
		\end{array}\right.
		$$
		The existence of $u_2$ and $u_3$ is provided, and then we define $u_1 = \tilde{u}-u_2-u_3$. By the above proposition %(maybe do it a bit better - when the function is $0$ outside, you really get the bound $|u|\leq ||fd^{2s-\varepsilon}|| d^{\varepsilon}$), 
		we have 
		$$||u_1||_{L^{\infty}(U)}\leq C_U(||fd^{2s-\varepsilon}||_{L^\infty(U)} +  ||\hat{f}||_{L^\infty(U)}).$$
		By standard elliptic estimates,
		$$||u_2||_{L^{\infty}(U)}\leq||\tilde{u}||_{L^{\infty}(\R^n)},$$
		and so applying \cite[Lemma 3.7]{AR20}, we get
		\begin{align*}
			 ||P||_{L^\infty(U)}\leq C||u_1||_{L^\infty(U)}&\leq ||\tilde{u}||_{L^\infty(U)}+||u_1||_{L^\infty(U)}+||u_2||_{L^\infty(U)} \\
			&\leq C\left(||fd^{2s-\varepsilon}||_{L^\infty(U)} + \left|\left|\frac{u}{1+|\cdot|^{k+s+\alpha}}\right|\right|_{L^\infty(\R^n)} \right).
		\end{align*}
				
		Hence defining $\tilde{f}:= f-\hat{f}-P$, the lemma is proven.
	\end{proof}
	
	This immediately gives the boundary regularity result with which we conclude this section.
	
	\begin{lemma}\label{boundaryRegGeneralised}
		Let $L$ be an operator whose kernel $K$ satisfies \eqref{kernelConditions}. Let $\Omega$ be a domain of class $C^\beta$, $\beta>1$. Let for some $k\in\N$ a function $u$ be the solution to 
		$$\left\lbrace \begin{array}{rcl  l}
		Lu&\stackrel{k}{=}&f& \text{in }\Omega\cap B_1,\\
		u&=&0&\text{in } B_1\backslash \Omega,\\
		\displaystyle\left|\left|\frac{u}{1+|\cdot|^{k+s+\alpha}}\right|\right|_{L^\infty(\R^n)}&\leq& C&\text{with }\alpha<s,
		\end{array}\right. $$
		with $|f|\leq C d^{\varepsilon-2s}$ and $\varepsilon\in(0,s)$.
		
		Then we have
		$$||u||_{C^\varepsilon(\overline{B}_{1/2})}\leq C\left(||fd^{2s-\varepsilon}||_{L^\infty(U)} + \left|\left|\frac{u}{1+|\cdot|^{k+s+\alpha}}\right|\right|_{L^\infty(\R^n)} \right),$$
		where $C$ does not depend on $u$.
	\end{lemma}

	\begin{proof}
		We define $\tilde{u} := u\chi_{B_2}$. It suffices to prove the claim for $\tilde{u}$, since functions agree on $B_{1/2}$. By \cite[Theorem 1.1]{AuR20},\footnote{Even though the result is stated with the right-hand side equal to zero, the proof is done under the assumptions we have here.} we get 
		$$||u||_{C^\varepsilon(B_{1/2})}\leq \left(||\tilde{f}d^{2s-\varepsilon}||_{L^\infty(\Omega\cap B_1)} + ||\tilde{u}||_{L^\infty(\R^n)}\right),$$
		and by Lemma \ref{generalised3.6}, we get the result.
	\end{proof}

	\section{Expansion results}
	
	In this section we investigate how well can functions be approximated with the generalised distance function, depending on the regularity of the right-hand side and the regularity of the boundary of the domain. We sort the results according to the increasing regularity of the boundary.  
	
	The proofs are all done in the same way. We perform a blow-up with compactness argument. First proofs are done with all the details, and in the later ones we just explain what are the differences from the previous cases.
	
	We start with a result, similar to \cite[Proposition 3.2]{RS16}. 
	
	\begin{lemma}\label{1function1.2}
		Let $\alpha= \varepsilon_0 + \varepsilon \in(0,s)$, and $\Omega$ be a $C^{1,\alpha}$ domain. Let $0\in\partial \Omega,$ and let $\partial\Omega\cap B_1$ be a graph of some $C^{1,\alpha}$ function of $C^{1,\alpha}$ norm less than $1$. 
		
		Let $L$ be an operator with kernel $K$ satisfying \eqref{kernelConditions}. Suppose that for every $z\in\partial\Omega\cap B_{3/4}$ there is a function $g_z$, so that 
		$$\left\lbrace \begin{array}{r c l l}
		L(u-g_z)&=&f_z,& \text{in }\Omega\cap B_1\\
		u-g_z&=&0,&\text{in } B_1\backslash \Omega,
		\end{array}\right. $$
		whit $f_z$ satisfying $|f_z(x)|\leq C_0 d^{\varepsilon_0 - s}|x-z|^{\varepsilon}$ in $\Omega\cap B_1$. We set $K_0 = C_0 + \sup_z||u-g_z||_{L^\infty(\R^n)}<\infty$. 
		
		Then, for every $z\in\partial\Omega\cap B_{1/2}$ there exists a constant $Q_z$ satisfying $|Q_z|\leq CK_0$ so that 
		$$|u(x)-g_z(x)-Q_zd^s(x)|\leq CK_0|x-z|^{\alpha + s}.$$
		
		Moreover, whenever $d(x_0) = |x_0-z|=2r$, we have $$\left[u-g_z-Q_zd^s \right]_{C^{2s}(B_r(x_0))} \leq CK_0 r^{\alpha-s}$$
		and if $d(x_1) = 2r_1 = |x_1-z_1|$, we have
		$$\left[u-g_z-Q_zd^s \right]_{C^{\alpha+s}(B_{r_1}(x_1))} \leq CK_0 \left(\frac{|x_1-z|}{r_1}\right)^{\alpha+s}.$$
		The constant $C$ depends only on $n,s,\alpha$ and ellipticity constants.
	\end{lemma}

	\begin{proof}
		First, we prove the existence claim. Due to the assumption on the norm of the boundary defining function, we can prove the claim for $z=0$. For simplicity we denote $u-g_0 = \tilde{u}.$ We proceed in the same way as in \cite[Proposition 3.2]{RS16}. We can assume that $K_0=1$ and that the normal vector of $\partial \Omega$ at $0$ is $e_n$. Let us assume by contradiction there exist $\Omega_k,L_k,u_k,f_k$ satisfying the assumptions, so that for every $Q_k$ we have 
		$$\sup_{r>0}\frac{1}{r^{s+\alpha}}||u_k-Q_k d_k^s||_{L^\infty(B_r)}>k.$$
		Now we define the constants $Q_{k,r} = \frac{\int_{B_r}u_kd^s}{\int_{B_r}d^{2s}}$, and the monotone function 
		$$\theta(r) = \sup_{k}\sup_{\rho>r}\frac{1}{\rho^{\alpha+s}}||u_k-Q_{k,r} d_k^s||_{L^\infty(B_\rho)},$$
		which by \cite[Lemma 3.3]{RS16} converges to $\infty$ as $r\downarrow 0$. Hence there are sequences $r_m,k_m$, such that 
		$$\theta(r_m)\geq \frac{1}{r_m^{\alpha+s}}||u_{k_m}-Q_{k_m,r_m} d_{k_m}^s||_{L^\infty(B_{r_m})}\geq \frac{1}{2}\theta(r_m).$$
		With them, we define the blow-up sequence
		$$v_m(x):=\frac{1}{r_m^{s+\alpha}\theta(r_m)}(u_{k_m}(r_mx) - Q_{r_m,k_m} d_{k_m}^s(r_mx)),$$
		which by the definition of the constants $Q_{r,k}$ satisfies
		$$\int_{B_1} v_m(x)d^s(r_mx)dx =0,$$
		and by the definition of $\theta$ also $||v_m||_{L^\infty(B_1)}\geq \frac{1}{2}$.
		
		With the same arguments as in \cite[Proposition 3.2]{RS16}, we get the estimates $|Q_{k,r} - Q_{k,2r}|\leq Cr^\alpha \theta(r)$ and $|Q_{k,r} - Q_{k,Rr}|\leq C(Rr)^\alpha \theta(r)$, where both $C$ are independent of $k$, which furthermore give that $||v_m||_{L^\infty(B_R)}\leq CR^{s+\alpha}.$

		We proceed to computing $L(v_m)(x) = \frac{r_m^{s-\alpha}}{\theta(r_m)}\big(f(r_mx)+ Q_{r_m,k_m} L(d_m^s)(x)\big),$ which we estimate using \cite[Proposition 2.3]{RS16} with
		$$|L(v_m)(x)|\leq \frac{r_m^{s-\alpha}}{\theta(r_m)}d_m^{\varepsilon_0-s}(r_mx)|r_mx|^\varepsilon + \frac{Q_{r_m,k_m}}{\theta(r_m)} r_m^{s-\alpha}d_m^{\alpha-s}(r_mx)$$
		$$\leq \frac{C}{\theta(r_m)}d_m^{\varepsilon_0-s}(x)|x|^\varepsilon + \frac{Q_{k_m,r_m}}{\theta(r_m)}d_m^{\alpha-s}(x),$$
		which converges to $0$ as $m\to\infty$ in every compact set in $\left\lbrace x_n>0 \right\rbrace $, since $\theta(r_m)^{-1}\to0$ and $\frac{Q_{r_m,k_m}}{\theta(r_m)}\to 0$. This implies that $|L(v_m)(x)|\leq C_M d^{\varepsilon_0-s}(x)$ on $B_M$, which is needed to get the uniform bound (in $m$) on $||v_m||_{C^s(B_M)}$, using \cite[Proposition 3.1]{RS16}. 
		
		The rest of the proof is analogous to the original one. We conclude that $v_m$ converge to $\kappa (x_n)_+^s,$ for some $\kappa \in \R$, which together with passing the integral quantity to the limit gives $\kappa = 0$. But this is a contradiction with the non-triviality of $||v||_{L^\infty(B_1)}$.
		
		Let us now show, that the growth control proven above gives the interior regularity we want. For this, we use \cite[Corollary 3.6]{RS16b} and \cite[Proposition 2.3]{RS16}. Concretely, take $2r = d(x_0) = |x_0-z|$. Then the mentioned corollary, used on $(u-g_z-Q_zd^s)(x_0+r\cdot)$, gives
		$$r^{2s}\left[ u-g_z-Q_zd^s\right]_{C^{2s}(B_r(x_0))}\leq C(K_0r^{\alpha+s}+ K_0r^{2s}r^{\alpha-s})\leq CK_0r^{\alpha+s}. $$
		We bound the first term using Lemma \ref{zoomGrowthLemma} while for the other one we use the assumption on $L(u-g_z)$ and Lemma \ref{LonDistance}.
		
		Let now $d(x_1) = 2r_1 = |x_1-z_1|$. Then denoting $u_r(x) := (u-g_z-Q_zd^s)(x_1+2r_1x)$, and applying the same two results, we get
		$$r_1^{\alpha+s}\left[ u-g_z-Q_zd^s\right]_{C^{\alpha+s}(B_{r_1}(x_1))}\leq C(K_0|x_1-z|^{\alpha+s}+ K_0r_1^{\varepsilon_0+s}|x_1-z|^{\varepsilon})\leq CK_0|x_1-z|^{\alpha+s}, $$
		again with aid of Lemma \ref{zoomGrowthLemma} and Lemma \ref{LonDistance}.
	\end{proof}

	This result, together with  the generalisation of \cite[Proposition 4.1]{AR20} cover all cases when doing expansions with $d^s$ varying the regularity of $\Omega$. The result states as follows:

	\begin{lemma}\label{1function2.2}
		Let $\beta>1+s$, $\beta\not\in\N$, and let $\Omega\subset \R^n$ be $C^\beta$ domain. Assume $0\in\partial \Omega,$ and let $\partial\Omega\cap B_1$ be a graph of some $C^{\beta}$ function of $C^{\beta}$ norm less than $1$.   Let $L$ be an operator whose kernel $K$ is $C^{2\beta+1}(\S^{n-1})$ and satisfies \eqref{kernelConditions}. Let $u\colon \R^n \to \R$ and assume for every $z\in\partial\Omega\cap B_{3/4}$ we have a function $g_z$ so that 
		$$\left\lbrace\begin{array}{rcll}
		L(u-g_z)&=&f_z&\text{in  } \Omega\\
		u-g_z &=&0 & \text{in  } \Omega^c\cap B_1.
		\end{array}\right.$$
		For some $\varepsilon>0$, let us have $|f_z(x)-P_z(x)|\leq C|x-z|^{\beta - 1+s-\varepsilon}d^{\varepsilon-2s}$, for $x\in B_1\cap\Omega$, where $C$ does not depend on the boundary point $z$, for some $P_z\in\textbf{P}_{\left[ \beta-1\right] }$. Furthermore, let also 
		$$\left[f_z \right] _{C^{\beta-1-s}(B_{\frac{3r}{2}}(x_1))}\leq C\left(\frac{|x_1-z|}{r_1}\right)^{\beta - s} ,$$
		whenever $d(x_1) = 2r$, independently of $z,r,x_1$. Then for every $z\in\partial\Omega\cap B_{1/2}$ there exists a polynomial $Q_z\in \textbf{P}_{\lfloor\beta-1\rfloor}$, so that 
		$$|u(x)-g_z(x) - Q_z(x) d^s(x)|\leq C|x-z|^{\beta-1+s} \quad \text{for every }x\in B_1, $$ 
		and $C$ is independent of $z$.
		Also for if $x_1\in\Omega\cap B_1$, $d(x_1) = 2r_1$, then
		$$\left[ u-g_z-Q_zd^s\right] _{C^{\beta-1+s}(B_r(x_1))}\leq C\left(\frac{|x_1-z|}{r_1}\right)^{(\beta-1+ s)\lor (\beta-s)}$$
		with $C$ depending only on $n,s,\beta$ and $||K||_{C^{2\beta+1}(\S^{n-1})}$.
	\end{lemma}
	
	\begin{proof}
		First we argue that without loss of generality we work at $z=0$, because of the assumption on the norm of the domain. Then the proof goes exactly the same as in \cite[Proposition 4.1]{AR20} denoting $\tilde{u}= u-g_0$. Notice that in the paper, we needed that $Lu\in C^{\beta-1-s}$ to find a polynomial $P$ such that $Lu-P = O(|x|^{\beta-1-s})$.
		Now by assumption we have $Lu-P = O(|x|^{\beta-1+s-\varepsilon}d^{\varepsilon-2s})$, which gives that the blow up sequence satisfies
		$$|L_mv_m(x) -P_m(x)|\leq \frac{C}{\theta(r_m)}|x|^{\beta-1+s-\varepsilon}d_m^{\varepsilon -2s}(x),$$
		which still converges locally uniformly to $0$ in the half-space. Instead of \cite[Proposition 3.8]{AR20} we now use Lemma \ref{boundaryRegGeneralised}, which allows us to apply Arzela-Ascoli Theorem, so that $v_m$ converges locally uniformly in $\R^n$ to some function $v$. Hence by \cite[Lemma 3.5]{AR20} in every compact set in $\{x_n>0\}$ the limit function satisfies $L_\star v \stackrel{k}{=} 0$. The conclusion of the proof is then the same.
		
		Second part, where we needed the regularity of $f$ is for the interior regularity estimate. Using \cite[Proposition 3.9]{AR20}, similarly as at the end of the previous lemma, we obtain
		\begin{align*}
			r_1^{\beta-1+s}\left[ u-g_z-Q_zd^s\right]_{C^{\beta-1+s}(B_{r_1}(x_1))}&\leq C\left(|x_1-z|^{\beta-1+s}+ r_1^{\beta-1+s}\left(\frac{|x_1-z|}{r_1}\right)^{\beta - s}  \right)\\
			&\leq C\left(\frac{|x_1-z|}{r_1}\right)^{(\beta-1+s)\lor(\beta - s)},
		\end{align*}
		as desired.
	\end{proof}

	In the similar manner we also want to give the complete answer to expansions of one solution with respect to another, non-trivial one. Replacing the distance function with a solution basically increases the order of approximation by one. Now the cases split in three categories with respect to the smoothness of the right-hand side. First one is when the right-hand side $f$ is only bounded by $|f|\leq C d^{\varepsilon-s}$, for some $\varepsilon\in(0,s)$, the second one is $f\in C^{\varepsilon-s}$, $\varepsilon\in(s,1)$, and the third one is $\varepsilon>1$. In the first two cases we only need that the boundary $\partial\Omega$ is $C^{1},$  while in the third case we need to assume that the boundary is $C^\beta$, with $\beta = \varepsilon+s$. The model result here is \cite[Proposition 4.4]{AR20}, which deals with the third case.
	
	First we prove the basic claims in all settings. Let us begin with the case $\varepsilon<s$.

	\begin{lemma}\label{2functions1.1}
		Let $L$ be an operator with kernel $K$ satisfying \eqref{kernelConditions}.  Let $\Omega$ be a domain with $0\in\partial\Omega$, so that $\partial \Omega\cap B_1$ is a graph of a function whose $C^{1}$ norm is smaller than $1$. For $i=1,2$, let $u_i$ be the solution  to
		$$\left\lbrace \begin{array}{rcll}
		Lu_i&=&f_i& \text{in }\Omega\cap B_1\\
		u_i&=&0&\text{in } B_1\backslash \Omega,
		\end{array}\right. $$
		for some $f_i$ satisfying $|f_i|\leq C_0d^{\varepsilon-s}, $ $\varepsilon\in(0,s)$. Furthermore assume the existence of $C_2,c_2>0$, such that $C_2d^s\geq u_2\geq c_2d^s$ in $B_1$. Denote $K_0 = C_0 + ||u_1||_{L^\infty(\R^n)}+||u_2||_{L^\infty(\R^n)}$.
		
		Then for every $z\in\partial\Omega\cap B_{1/2}$ there exist a constant $Q_z$, such that 
		$$|u_1(x)-Q_zu_2(x)|\leq CK_0|x-z|^{\varepsilon+s},\quad x\in B_{1/2}(z)$$
		and 
		$$\left[u_1-Q_zu_2 \right]_{C^{s+\varepsilon}(B_r(x_0))} \leq CK_0 ,$$
		whenever $d(x_0)=|x_0-z|=2r.$ The constant $C$ depends only on $n,s,c_2,C_2,\varepsilon$ and ellipticity constants.
	\end{lemma}

	\begin{proof}
		We start with the expansion estimate part. Without loss of generality assume $K_0=1$, $z=0$ and the normal vector $\nu_0=e_n.$ Assume by contradiction that there exist $L_j,\Omega_j,u_{i,j},f_{i,j}$ satisfying the assumptions of the lemma, but for every $Q_j$ we have $$\sup_{r>0}\frac{1}{r^{s+\varepsilon}}||u_{1,j}-Q_ju_{2,j}||_{L^\infty(B_r)}>j.$$
		We define 
		$$Q_{r,j} := \frac{\int_{B_r}u_{1,j}u_{2,j}}{\int_{B_r}u_{2,j}^2},$$
		so that $\int_{B_r}(u_{1,j}-Q_{r,j}u_{2,j})u_{2,j} dx = 0$. Define now the monotone quantity 
		$$\theta(r) := \sup_j\sup_{\rho>r} \frac{1}{\rho^{\varepsilon+s}}||u_{1,j}-Q_{\rho,j}u_{2,j}||_{L^\infty(B_\rho)},$$
		which by contradiction assumption and \cite[Lemma 4.5]{AR20} converges to infinity as $r$ goes to zero. We find the realising sequences $r_m,j_m$ so that 
		$$m\leq\frac{\theta(r_m)}{2}\leq \frac{1}{r_m^{\varepsilon+s}}||u_{1,j_m}-Q_{r_m,j_m}u_{2,j_m}||_{L^\infty(B_{r_m})}\leq \theta(r_m).$$ %(The $m\leq ...$ is there to force $\theta(r_m)$ to infinity and $r_m$ to $0$.) 
		The blow-up sequence is defined with
		$$v_m(x):=\frac{1}{\theta(r_m)r_m^{\varepsilon+s}}\big(u_{1,j_m}(r_mx) - Q_{j_m,r_m}u_{2,j_m}(r_mx)\big).$$
		By construction, we have both $||v_m||_{L^\infty(B_1)}\geq \frac{1}{2}$ and $\int_{B_1} v_m(x)u_{2,j_m}(r_mx)dx = 0$. In the limit of some subsequence of $v_m$, these quantities will contradict each other.
		
		We proceed with some estimations
		\begin{align*}
			|Q_{r,j}-Q_{2r,j}|&\leq C\frac{2^s}{r^s}||Q_{r,j}u_{2,j}-Q_{2r,j}u_{2,j}||_{L^\infty(B_r\cap\{d>r/2\})}\\
			&\leq \frac{C}{r^s}\big(||u_{1,j}-Q_{r,j}u_{2,j}||_{L^\infty(B_r)} + ||u_{1,j}-Q_{2r,j}u_{2,j}||_{L^\infty(B_{2r})}\big)\\
			&\leq Cr^{-s}\big (\theta(r)r^{s+\varepsilon}+ \theta(2r)(2r)^{s+\varepsilon}\big )\leq C\theta(r)r^\varepsilon,
		\end{align*}
		where we used the monotonicity of $\theta$.
		Summing the geometric series, this gives
		\begin{align*}
			|Q_{r,j}-Q_{2^kr,j}|\leq\sum_{i=0}^{k-1}|Q_{2^{i+1}r,j}-Q_{2^{i}r,j}|&\leq \sum_{i=0}^{k-1}C\theta(2^ir)(2^ir)^\varepsilon\\
			&\leq C\theta(r)r^\varepsilon C_\varepsilon 2^{\varepsilon k}\leq C\theta(r)(2^kr)^\varepsilon,
		\end{align*}
		which furthermore gives $$ |Q_{r,j}-Q_{Rr,j}|\leq C_\varepsilon\theta(r)(Rr)^\varepsilon.$$
		The latter estimate implies the growth of $v_m$ at infinity, as follows
		\begin{align*}
			||v_m||_{L^\infty(B_R)} =& \left|\left|\frac{1}{\theta(r_m)r_m^{\varepsilon+s}}(u_{1,j_m}-Q_{j_m,r_m}u_{2,j_m})\right|\right|_{L^\infty(B_{Rr_m})}\\
			\leq& \frac{1}{\theta(r_m)r_m^{\varepsilon+s}}\left|\left|u_{1,j_m}-Q_{j_m,Rr_m}u_{2,j_m}\right|\right|_{L^\infty(B_{Rr_m})}+\\
			& + \frac{1}{\theta(r_m)r_m^{\varepsilon+s}} \left|Q_{r_m,j_m}-Q_{Rr_m,j_m}\right|\cdot\left|\left|u_{2,j_m}\right|\right|_{L^\infty(B_{Rr_m})}\\
			\leq& \frac{\theta(Rr_m)(Rr_m)^{\varepsilon+s}}{\theta(r_m)r_m^{\varepsilon+s}} + \frac{C_\varepsilon\theta(r_m)(Rr_m)^\varepsilon C_2(Rr_m)^s}{\theta(r_m)r_m^{\varepsilon+s}}\leq CR^{s+\varepsilon}.
		\end{align*}
		We proceed to computing the $L_{j_m}$ on the blow-up sequence:
		$$L_{j_m}v_m(x) = \frac{1}{\theta(r_m)r_m^{\varepsilon+s}}r_m^{2s}(f_{1,j_m}-Q_{j_m,r_m}f_{2,j_m})(r_mx) =\frac{r_m^{s-\varepsilon}}{\theta(r_m)}(f_{1,j_m}-Q_{j_m,r_m}f_{2,j_m})(r_mx),$$
		so when estimating the modulus we get
		$$|L_{j_m}v_m(x)|\leq\frac{C}{\theta(r_m)}r_m^{s-\varepsilon}(1 + Q_{j_m,r_m})d_{j_m}^{\varepsilon-s}(r_mx) \leq\frac{C}{\theta(r_m)}(1+Q_{j_m,r_m})d_{j_m}^{\varepsilon-s}(x).$$
		Let us now show, that $\frac{Q_{j,r}}{\theta(r)}\to 0$ uniformly in $j$, as $r\to 0$. Let $2^{-k-1}\leq r \leq 2^{-k}$. Then
		\begin{align*}
			|Q_{j,r}|&\leq |Q_{j,2^kr}|+\sum_{i=0}^{k-1}|Q_{j,2^ir}-Q_{j,2^{i+1}r}|\leq C + \sum_{i=0}^{k-1} C \theta(2^ir)(2^ir)^\varepsilon\\
			&\leq C(1 + \sum_{i=0}^{k-1}  \theta(2^{i-k})(2^{i-k})^\varepsilon)\leq  C\left(1 + \sum_{i=0}^{k-1}  \theta(2^{-i})(2^{-i})^\varepsilon\right).
		\end{align*}
		Now we divide with $\theta(r)$, which is better then dividing with $\theta(2^{-k})$, we get a series $\sum_{i=0}^{k-1}  \frac{\theta(2^{-i})}{\theta(2^{-k})}(2^{-i})^\varepsilon$, which is of the form "convergent series times coefficients which all go to $0$", and so we can show (with an $\epsilon,r_0$ calculus) that it goes to $0$ indeed. Above we also used the uniform bound $|Q_{r,j}|\leq C$, $r\in(1/2,1)$ which follows by the definition of $Q_{r,j}$, uniform boundedness of $u_{1,j}$, $u_{2,j}$ and the uniform boundedness from below on $u_{2,j}.$ As a consequence, we get that $L_{j_m}v_m\to 0$ uniformly in every compact set in $\{x_n>0\}$ as well as $|L_{j_m}v_m|\leq Cd_{j_m}^{\varepsilon-s}.$
		This bound, together with previously obtained growth control of $v_m$, allows us to apply \cite[Proposition 3.1]{RS16}, to get a uniform bound  $\left[ v_m\right] _{C^s(B_M)}\leq C(M) $.
		
		The same argumentation as in \cite[Proposition 3.2]{RS16} we get that a subsequence still denoted $v_m$ converges locally uniformly to $v$ in $\R^n$, and $v$ solves $L_*v = 0$ in $\{x_n>0\},$ and $v=0$ in $\{x_n\leq 0\},$ which implies that $v(x)=\kappa(x_n)_+^s$ (see \cite[Proposition 5.1]{RS16c}.) But passing the integral quantity to the limit, gives $\kappa= 0$, which gives contradiction with $||v||_{L^\infty(B_1)}\geq \frac{1}{2}.$
		
		To get the interior regularity from the claim, we use the growth estimate just proven together with \cite[Corollary 3.6]{RS16b}, as follows. Take $d(x_0) = |x_0-z|=2r$, and denote $v_r(x) = (u_1-Q_zu_2)(x_0+2rx)$. The growth estimate implies $v_r(x)\leq CK_0|2rx+x_0-z|^{\varepsilon+s}\leq Cr^{\varepsilon+s}(1+|x|)^{\varepsilon+s}$.
		We apply the corollary, to get
		$$	\left[ v_r\right] _{C^{2s}(B_{1/2})}\leq C(CK_0r^{\varepsilon+s}+ r^{2s}CK_0r^{\varepsilon-s}),$$
		which after passing to smaller exponent and the rescaling gives 
		$$\left[ u_1-Q_zu_2\right] _{C^{\varepsilon+s}(B_{r}(x_0))}\leq CK_0,$$
		as wanted.
	\end{proof}
	
	Notice, that we need the boundary to be at least $C^1,$ so that the blow-up of the domain converges to the half space. 
	%TODO: add a comment about the non-triviality assumption 
	
	We now state the analogous result, when $\varepsilon\in (s,1)$. 
	
	\begin{lemma}\label{2functions2.1}
		Let $L$ be an operator with kernel $K$ satisfying \eqref{kernelConditions}.  Let $\Omega$ be a domain with $0\in\partial\Omega$, so that $\partial \Omega\cap B_1$ is a graph of a function whose $C^{1}$ norm is smaller than $1$. For $i=1,2$, let $u_i$ be the solution  to
		$$\left\lbrace \begin{array}{rcll}
		Lu_i&=&f_i& \text{in }\Omega\cap B_1\\
		u_i&=&0&\text{in } B_1\backslash \Omega,
		\end{array}\right. $$
		for some $f_i\in C^{\varepsilon-s}(\Omega\cap B_1), $ $\varepsilon\in(s,1)$. Furthermore assume the existence of $C_2,c_2>0$, such that $C_2d^s\geq u_2\geq c_2d^s$ in $B_1$. Denote $K_0 = ||f_1||_{C^{\varepsilon-s}}+||f_2||_{C^{\varepsilon-s}}+ ||u_1||_{L^\infty(\R^n)}+||u_2||_{L^\infty(\R^n)}$.
		
		Then for every $z\in\partial\Omega\cap B_{1/2}$ there exist a constant $Q_z$, such that 
		$$|u_1(x)-Q_zu_2(x)|\leq CK_0|x-z|^{\varepsilon+s},\quad x\in B_{1/2}(z)$$
		and 
		$$\left[u_1-Q_zu_2 \right]_{C^{s+\varepsilon}(B_r(x_0))} \leq CK_0 ,$$
		whenever $d(x_0)=|x_0-z|=2r.$ The constant $C$ depends only on $n,s,c_2,C_2,\varepsilon$ and ellipticity constants.
	\end{lemma}
	
	\begin{proof}
		The proof goes along the same lines of the proof of the lemma above, just the converging arguments changes a little. To obtain the uniform $C^s$ bound on the blow-up sequence $v_m$, we use \cite[Proposition 3.8]{AR20} instead, and since the growth of the limit function is too large, we use Liouville theorem \cite[Theorem 3.10]{AR20}, to get that the limit $v(x) = p(x)(x_n)_+^s,$ for some polynomial $p$ of degree one. But then the growth control of $v$ ensures that the polynomial is indeed of degree zero.
		
		To obtain the interior regularity estimate we now apply \cite[Proposition 3.9]{AR20}.
	\end{proof}

	The third case, $\varepsilon>1$ is done in \cite[Proposition 4.4]{AR20}.

	Proceed now to the generalisations. First we state the result with the least regularity assumptions.
	
	\begin{lemma}\label{2functions1.2}
		Let $L$ be an operator whose kernel $K$ satisfies \eqref{kernelConditions}. Let $\Omega$ be a domain with $0\in\partial\Omega$, so that $\partial \Omega\cap B_1$ is a graph of a function whose $C^{1}$ norm is smaller than $1$. For $i=1,2$ and $z\in\partial\Omega\cap B_{3/4}$, let $u_i-g_{i,z}$ be  solutions  to
		$$\left\lbrace \begin{array}{rcll}
		L(u_i-g_{i,z})&=&f_{i,z}& \text{in }\Omega\cap B_1,\\
		u_i-g_{i,z}&=&0&\text{in } B_1\backslash \Omega,
		\end{array}\right. $$
		for some $f_{i,z}$ satisfying $|f_{i,z}(x)|\leq C_0d^{\varepsilon_0-s}|x-z|^\varepsilon, $ $\varepsilon_0+\varepsilon\in(0,s)$. Furthermore assume the existence of $C_2,c_2>0$, such that $C_2d^s\geq u_2-g_{2,z}\geq c_2d^s$ in $B_1$. Denote $$K_0 = C_0 + \sup_z||u_1-g_{1,z}||_{L^\infty(\R^n)}+ \sup_z||u_2-g_{2,z}||_{L^\infty(\R^n)}.$$
		
		Then for every $z\in\partial\Omega\cap B_{1/2}$ there exist a constant $Q_z$, such that 
		$$|u_1(x)-g_{1,z}(x)-Q_z(u_2(x)-g_{2,z}(x))|\leq CK_0|x-z|^{\varepsilon_0+\varepsilon+s},\quad x\in B_{1/2}(z)$$
		and 
		$$\left[u_1-g_{1,z}-Q_z(u_2-g_{2,z}) \right]_{C^{s+\varepsilon_0+\varepsilon}(B_r(x_0))} \leq CK_0 ,$$
		whenever $d(x_0)=|x_0-z|=2r.$ The constant $C$ depends only on $n,s,c_2,C_2,\varepsilon$ and ellipticity constants.
	\end{lemma}

	\begin{proof}
		The proof works in the same way as the the one in Lemma \ref{1function1.2}, but following the proof of Lemma \ref{2functions1.1}, just that the exponent changes to $\varepsilon_0+\varepsilon$ instead of $\varepsilon$. 
	\end{proof}

	Similarly happens with Lemma \ref{2functions2.1}.
	
	\begin{lemma}\label{2functions2.2}
		Let $L$ be an operator whose kernel $K$ satisfies condition \eqref{kernelConditions}. Let $\Omega$ be a domain with $0\in\partial\Omega$, so that $\partial \Omega\cap B_1$ is a graph of a function whose $C^{1}$ norm is smaller than $1$. For $i=1,2$ and $z\in\partial\Omega\cap B_{3/4}$, let $u_i-g_{i,z}$ be  solutions  to
		$$\left\lbrace \begin{array}{rcll}
		L(u_i-g_{i,z})&=&f_{i,z}& \text{in }\Omega\cap B_1,\\
		u_i-g_{i,z}&=&0&\text{in } B_1\backslash \Omega.
		\end{array}\right. $$
		Let there exist constants $P_z$, so that $|f_{i,z}(x)-P_z|\leq C_0d^{\varepsilon_0-s}|x-z|^\varepsilon, $ $\varepsilon,\varepsilon_0>0,$ $\varepsilon_0<s$, $\varepsilon_0+\varepsilon\in(s,1)$.
		Assume also that whenever $d(x_0)=2r=|x_0-z|$, we have the bound
		$$ \left[ f_z\right] _{C^{\varepsilon_0+\varepsilon-s}(B_{\frac{3r}{2}})}\leq C_0. $$
		Furthermore assume the existence of $C_2,c_2>0$, such that $C_2d^s\geq u_2-g_{2,z} \geq c_2d^s$ in $B_1$. Denote $K_0 = C_0 + \sup_z||u_1-g_{1,z}||_{L^\infty(\R^n)}+\sup_z||u_2-g_{2,z}||_{L^\infty(\R^n)}$.
		
		Then for every $z\in\partial\Omega\cap B_{1/2}$ there exist a constant $Q_z$, such that 
		$$|u_1(x)-g_{1,z}(x)-Q_z(u_2(x)-g_{2,z}(x))|\leq CK_0|x-z|^{\varepsilon+\varepsilon_0+s},\quad x\in B_{1/2}(z)$$
		and 
		$$\left[ u_1-g_{1,z}-Q_z(u_2-g_{2,z}) \right]_{C^{s+\varepsilon+\varepsilon_0}(B_r(x_0))} \leq CK_0 ,$$
		whenever $d(x_0)=|x_0-z|=2r.$ The constant $C$ depends only on $n,s,c_2,C_2,\varepsilon$ and ellipticity constants.
	\end{lemma}

	\begin{proof}
		The proof is again done with the blow up argument as in the proof of Lemma \ref{2functions2.1}, just that the assumption $f\in C^\varepsilon$ is replaced with approximation estimate ($|f_{i,z}(x)-P_z|\leq C_0 |x-z|^\varepsilon d^{\varepsilon_0-s}$) together with the interior regularity ($ \left[ f_z\right] _{C^{\varepsilon+\varepsilon_0-s}(B_{\frac{3r}{2}})}\leq C_0 $). The convergence argument is done in the same way as in Lemma \ref{2functions1.2}, just that for the equation of the limit function we use \cite[Lemma 3.5]{AR20} on every compact set in $\{x_n>0\}$. 
		
		For the interior regularity estimate we again use \cite[Proposition 3.9]{AR20}.
	\end{proof}
	
	We conclude the section with the last generalised result.
	
	\begin{lemma}\label{2functions3.2}
		Let $\beta>1$, such that $\beta\not\in\N$ and $\beta\pm s\not\in\N$. Let $L$ be an operator whose kernel $K$ satisfies the condition \eqref{kernelConditions} and is $C^{2\beta+1}(\S^{n-1})$. Let $\Omega\subset \R^n$ be a bounded domain of class $C^\beta$, with $0\in\partial\Omega$.  Let $z\in\partial\Omega\cap B_{3/4}$ and $u_1,u_2\in L^\infty(\R^n)$ solutions of 
		$$\left\lbrace\begin{array}{rcl l}
		L(u_i-g_{i,z})&=&f_{i,z}&\text{in  } \Omega\cap B_1\\
		u_i-g_{i,z} &=&0 & \text{in  } \Omega^c\cap B_1,
		\end{array}\right.$$
		for some functions $g_{i,z}$. Suppose that  $\partial\Omega\cap B_1$ is a graph of some $C^\beta$ function whose $C^\beta$ norm is bounded by one. Assume that for some $c_1>0$,
		$$u_1(x)-g_{1,z}(x)\geq c_1 d^s(x),\quad \text{for any }x\in B_1(z).$$
		Let there exist two polynomials $P_{i,z}\in\textbf{P}_{\left[ \beta-s\right] }$, so that for some $\varepsilon>0,$ we have $|f_{i,z}(x) - P_{i,z}(x)|\leq C_1|x-z|^{\beta+s-\varepsilon}d^{\varepsilon-2s}, \quad x\in B_1(z)$, with $C_1$ independent of $z$. Assume also 
		$$\left[f_{i,z} \right] _{C^{\beta-s}(B_{\frac{3r}{2}}(x_0))}\leq C_2 ,$$ whenever $d(x_0) = 2r = |x_0-z|>0$, independently of $z,r,x_0$. Then for every $z\in\partial\Omega\cap B_{\frac{1}{2}}$ there exists a polynomial $Q_z\in\textbf{P}_{\lfloor\beta\rfloor }$, such that
		$$|u_2(x) - g_{2,z}(x) - Q_z(x)(u_1(x)-g_{1,z}(x))|\leq C|x-z|^{\beta+s},\quad \text{for any }x\in B_1(z).$$
		Moreover, when $d(x_0) = 2r = |x_0-z|>0$,
		$$\left[u_2- g_{2,z} - Q_z(u_1- g_{1,z}) \right] _{C^{\beta+s}(B_{r}(x_0))}\leq C .$$ 
		The constant $C$ depends only on $n,s,c_1,C_1,C_2$ and $||K||_{C^{2\beta+1}(\S^{n-1})}$.
	\end{lemma}
	
	\begin{proof}
		Again we can assume $z=0$ without loss of generality. Denote $\tilde{u}_1 = u_1-g_{1,0}$ and $\tilde{u}_2=u_2-g_{2,0}$. Applying Lemma \ref{1function2.2} (or Lemma \ref{1function1.2} when $\beta$ is to small) to $\tilde{u}_1$, we see with the same argument as in the paper, that it is equivalent finding $\tilde{Q} = \tilde{q}^{(0)}+\tilde{Q}^{(1)}(x)$ such that 
		$$|\tilde{u}_2 (x)- \tilde{q}^{(0)} \tilde{u}_1(x) - \tilde{Q}^{(1)}(x) d^s(x)|\leq C|x|^{\beta+s},\quad \text{for any }x\in B_1.$$ 
		Just like in Lemma \ref{1function2.2} we assumed all the necessary interior regularity, which replaces the assumption $f_i\in C^{\beta-s}(\overline{\Omega})$. Again it is needed to find a polynomial so that when subtracting it from the operator on the blow-up sequence we get something small, and to get a uniform bound on the $C^\varepsilon$ norm of the blow-up sequence (where we potentially use Lemma \ref{boundaryRegGeneralised}), and the interior regularity estimate in the claim. Note also that the equation of the limit function is obtained through \cite[Lemma 3.5]{AR20} on every compact subset of $\{x_n>0\}$, just like in the proof of Lemma \ref{1function2.2}.
	\end{proof}

	\section{Regularity estimates for solutions to linear equations}
	
	In this section, we study the behaviour near the boundary of solutions to the linear equation \eqref{linearEquation}	
	when $s>\frac{1}{2}$. We denote 
	$$\varepsilon_0=2s-1.$$
	In the first part we establish the optimal $C^s$ regularity of the solutions up to the boundary and in the second we study quotients of solutions with $d^s$ and quotients of two solutions.

	We deal with the weak solutions according to the following definition.
	
	\begin{definition}
		Let $\Omega\subset\R^n$ be a bounded domain. Let $b\in\R^n$, $f\in L^2(\Omega)$ and let $u\colon\R^n \to \R$ satisfy the growth control $|u(x)|\leq C(1+|x|^{2s-\delta}),$ for some $\delta>0.$ Assume $L$ is an operator whose kernel $K$ satisfies \eqref{kernelConditions}. 
		We say that $u$ is a weak solution of 
		$$Lu + b\cdot\nabla u = f\quad\text{in }\Omega,$$ 
		whenever
		$$\int_{\R^n}\int_{\R^n}(u(x)-u(y))(\varphi(x)-\varphi(y))K(x-y) dydx - \int_\Omega u(x)b\cdot\nabla \varphi(x) dx = \int_\Omega f(x)\varphi(x)dx,$$
		holds for every $\varphi\in C^\infty_c(\Omega)$.
	\end{definition}

	\begin{remark}\label{regularisationRemark}
		Due to the fact that $K(x-y) = K((x-t) - (y-t))$ the following fact holds.
		If $\eta_\varepsilon\in C^\infty_c(\R^n)$ is a  mollifier and $u$ a solution to $Lu+b\cdot\nabla u = f$ in $\Omega$ in the weak sense, then $u_\varepsilon:= u\ast \eta_\varepsilon$ satisfies $Lu_\varepsilon +b\cdot\nabla u_\varepsilon = f\ast \eta_\varepsilon$ in $\Omega_\varepsilon$ in the weak sense, for $\Omega_\varepsilon = \{x\in \Omega;\text{  }\operatorname{dist}(x,\partial\Omega)>\varepsilon\}$.
	\end{remark}
	
	\subsection{Interior and boundary regularity}
	
	We present a version of interior regularity estimate that we need in our work. Even though it follows from already known results we provide a short proof.
	
	\begin{lemma}\label{interiorRegularityLemma}
		Let $s\in (1/2,1)$, and let $L$ be an operator whose kernel satisfies condition \eqref{kernelConditions}. Let $u$ be any solution of 
		\begin{equation*}
		Lu +b\cdot\nabla u= f\quad \text{ in }B_1,
		\end{equation*}
		with $f\in L^\infty(B_1)$ and $b\in\R^n$. Then for every $\delta>0,$
		$$||u||_{C^{2s}(B_{1/2})}\leq C\left(||f||_{L^\infty(B_1)}+ \sup_{R\geq 1} R^{\delta-2s}||u||_{L^\infty(B_R)}\right).$$
		The constant $C$ depends only on $n,s,\delta$ and ellipticity constants.
	\end{lemma}

	\begin{proof}
		Let us first assume that $u$ is a smooth solution. Then, by interior regularity estimates for the equation 
		$$Lu = f-b\cdot\nabla u\quad \text{ in }B_1,$$
		see \cite[Corollary 3.6]{RS16b}, we get 
		$$||u||_{C^{2s}(B_{1/2})}\leq C\left(||f||_{L^\infty(B_1)}+ \sup_{R\geq 1} R^{\delta-2s}||u||_{L^\infty(B_R)} + ||\nabla u||_{L^\infty(B_1)}\right).$$
		Since $s>\frac{1}{2}$, the interpolation inequality gives that for every $\varepsilon>0$ there is a constant $C$ so that
		$$||u||_{L^\infty(B_1)}\leq \varepsilon \left[u\right]_{C^{2s}(B_1)} + C_\varepsilon ||u||_{L^\infty(B_1)}.$$
		Combining the two inequalities, we get that for every $\varepsilon>0$,
		$$||u||_{C^{2s}(B_{1/2})}\leq \varepsilon \left[u\right]_{C^{2s}(B_1)}+ C_\varepsilon\left(||f||_{L^\infty(B_1)}+ \sup_{R\geq 1} R^{\delta-2s}||u||_{L^\infty(B_R)} \right),$$
		which thanks to \cite[Lemma 2.26]{FR20} gives
		$$||u||_{C^{2s}(B_{1/2})}\leq C\left(||f||_{L^\infty(B_1)}+ \sup_{R\geq 1} R^{\delta-2s}||u||_{L^\infty(B_R)}\right).$$
		Using a mollifier and taking the limit proves the general case, see Remark \ref{regularisationRemark}. 
	\end{proof}

	We proceed with the comparison principle for equation \eqref{linearEquation}.

	\begin{lemma}\label{comparisonPrinciple}
		Let $s\in (1/2,1)$, and let $L$ be an operator whose kernel satisfies condition \eqref{kernelConditions} and let $\Omega$ be a bounded domain. Let $u\in C(\R^n)$ be a solution to
		\begin{equation*}
		\left\lbrace\begin{array}{rcl l}
		Lu+b\cdot\nabla u&=&f&\text{in  } \Omega\\
		u &\geq&0 & \text{in  } \Omega^c,
		\end{array}\right.
		\end{equation*}
		with $f\in L^\infty(\Omega)$, $f\geq 0.$
		
		Then $u\geq 0$ in $\R^n$.
	\end{lemma}
	\begin{proof}
		Since the weak formulation of equations is preserved under mollification, the proof is the same as the proof of \cite[Lemma 4.1]{DRRV20}. 
%		Due to Lemma \ref{interiorRegularityLemma} the function $u$ is differentiable at every point in $\Omega$. Suppose that $u$ attains global minimum at $x_0\in\Omega$. Since $u$ is differentiable at $x_0,$ we have that $\nabla u(x_0)=0,$ and since $u(x_0)-u(y)\leq 0$ for all $y\in\R^n$, we have $Lu(x_0) \leq 0$. But since $Lu(x_0)+b\nabla u(x_0) = f(x_0) \geq 0$, we have that $Lu(x_0)=0$ which implies that $u(x_0) = u(y)$ for all $y\in\R^n$. Hence $u\equiv u(x_0)$ and since $u\geq0$ in $\Omega^c$, we get $u\geq0$.
	\end{proof}

	In order to achieve the optimal boundary regularity $C^s$, we need a suitable barrier -- a supersolution which grows as $d^s$ near the boundary. The construction is the same as in \cite[Lemma 2.8]{RS16}, we only need to estimate the gradient terms, which behave nicely thanks to the assumption $2s>1.$

	\begin{lemma}\label{barrier}
		Let $s\in (1/2,1)$, and let $L$ be an operator whose kernel satisfies condition \eqref{kernelConditions} and let $\Omega$ be a $C^{1,\alpha}$ domain. Then for a given $\varepsilon\in(0,\alpha\land \varepsilon_0)$ there exists a $\rho_0>0$ and a function $\phi$ satisfying
		\begin{equation*}
		\left\lbrace\begin{array}{rcl l}
		L\phi+b\cdot\nabla \phi&\leq&-d^{\varepsilon-s}&\text{in  } \Omega\cap\{d\leq\rho_0\}\\
		C^{-1}d^s\leq \phi &\leq&Cd^s & \text{in  } \Omega\\
		\phi&=&0&\text{in }\Omega^c.
		\end{array}\right.
		\end{equation*}
		The constant $C$ depends only on $n,s,\varepsilon,\Omega$ and ellipticity constants.
	\end{lemma}
	
	\begin{proof}
		By Lemma \ref{LonDistance}, we have that 
		$$|Ld^s|\leq C d^{(\alpha-s)\lor 0},$$
		and by \cite[Lemma 2.7]{RS16}, we have 
		$$Ld^{s+\varepsilon} \geq c_0d^{\varepsilon-s} - C_0,$$
		for some $c_0>0$. Consider the function
		$$\phi_1 = d^s-cd^{s+\varepsilon},$$
		with $c>0$ small enough.
		Then we have
		\begin{align*}
			L\phi_1+b\nabla\phi_1 &\leq Cd^{(\alpha-s)\lor 0} +C|\nabla d^{s}| -cc_0d^{\varepsilon-s} + C|\nabla d^{s+\varepsilon}|\leq\\
			&\leq -cc_0d^{\varepsilon-s}+C\left(d^{(\alpha-s)\lor 0} + Cd^{s-1}+Cd^{s+\varepsilon-1}\right)\leq \frac{cc_1}{2} d^{\varepsilon-s},
		\end{align*}
		in $\Omega\cap\{d\leq \rho_0\}$ provided that $\varepsilon<\varepsilon_0 \land \alpha$.
		By construction we have
		$$\phi_1=0\quad \text{ in }\Omega^c,$$
		as well as
		$$	C^{-1}d^s\leq \phi_1\leq Cd^s.$$
		Hence a suitable rescaling of $\phi_1$ is the searched function.	
	\end{proof}
	
	Having all the above ingredients, we can now prove the following.
	
	\begin{proposition}\label{optimalBoundaryRegularity}
		Let $s\in (1/2,1)$, and let $L$ be an operator whose kernel satisfies condition \eqref{kernelConditions} and let $\Omega$ be a $C^{1,\alpha}$ domain. Let $u$ be a solution to \eqref{linearEquation}, with 
		$$|f|\leq C_0d^{\varepsilon-s}\quad \text{in }\Omega\cap B_1,$$
		for some positive $\varepsilon$. 
		
		Then
		$$||u||_{C^s(B_{1/2})}\leq C\left(C_0 + \sup_{R\geq 1} R^{\delta-2s}||u||_{L^\infty(B_R)}\right).$$
		The constant $C$ depends only on $n,s,\varepsilon,\delta,\Omega$ and ellipticity constants.
	\end{proposition}

	\begin{proof}
		The wanted regularity follows from Lemma \ref{interiorRegularityLemma}, Lemma \ref{barrier} and Lemma \ref{comparisonPrinciple} exactly as in \cite[Proposition 3.1]{RS16}. 
	\end{proof}

%	\begin{proof}
%		Due to the homogeneity of the equation we may assume that
%		$$\left(||d^{s-\varepsilon}||_{L^\infty(B_1)} + \sup_{R\geq 1} R^{\delta-2s}||u||_{L^\infty(B_R)}\right)=1.$$
%		Then the truncated function $\tilde{u} = u\chi_{B_1}$ satisfies
%		$$|L\tilde{u}|\leq C d^{\varepsilon-s},\quad\text{for}x\in \Omega\cap B_{3/4} ,$$
%		$\tilde{u}\leq 1$ and $\tilde{u}\equiv 0$ in $\R^n\backslash B_1$. 
%		
%		Let $\tilde{\Omega}$ be a bounded $C^{1,\alpha}$ domain satisfying:  
%	\end{proof}
%TODO delete after xavi approves the proof.

	Before proceeding to the finer description of the boundary behaviour of solutions, let us derive the following estimates.

	\begin{lemma}\label{followingRegularityLemma}
		Let $s\in (1/2,1)$, and let $L$ be an operator whose kernel satisfies condition \eqref{kernelConditions} and let $\Omega$ be a $C^{1}$ domain. Let $u$ be a solution to \eqref{linearEquation}, with $|u|\leq C_0d^s$ in $\Omega\cap B_1$, and $f\in C^{\theta}(\overline{\Omega}),$ for some $\theta>0.$ 
		
		Then we have
		\begin{equation}\label{interiorEstimate}
		\left[  u \right]_{C^{\gamma}(B_{r}(x_0))}\leq C_\gamma r^{s-\gamma},
		\end{equation}
		whenever $d(x_0) = 2r$,
		as long as $\gamma \leq \theta+2s-1$. Furthermore, 
		\begin{equation}\label{gradientEstimate}
			|\nabla u|\leq Cd^{s-1}.
		\end{equation}
		The constants $C$ and $C_\gamma$ depend only on $n,s,C_0,||f||_{C^{\theta}(\overline{\Omega})},$ and ellipticity constants.
	\end{lemma}
	\begin{proof}
		Choose a point $x_0\in\Omega\cap B_{1/2}$, denote $2r=d(x_0)$ and apply interior estimates (Lemma \ref{interiorRegularityLemma}) on $u_r(x):=u(x_0+rx).$
		With aid of Lemma \ref{zoomGrowthLemma} and the bound on growth of $u$ near the boundary, we get
		$$r^{2s}\left[ u\right] _{C^{2s}(B_r(x_0))}\leq C (r^{2s}+ r^s),$$
		which tells 
		$$\left[ u\right] _{C^{2s}(B_r(x_0))}\leq C r^{-s}.$$
		Rewriting it for the gradient we get
		$$\left[ \nabla u\right] _{C^{2s-1}(B_r(x_0))}\leq C r^{-s}.$$
		Using Lemma \ref{growthLemma}, we deduce \eqref{gradientEstimate}. 
		
		To get higher order interior estimate \eqref{interiorEstimate}, we successively apply higher order interior estimates \cite[Proposition 3.9]{AR20}. Concretely assume we already have estimate 
		$$\left[  u\right] _{C^{\gamma'}(B_r(x_0))}\leq C r^{s-\gamma'},$$
		for some $\gamma'>1.$ Then the gradient satisfies $\left[ \nabla u\right] _{C^{\gamma'-1}(B_r(x_0))}\leq C r^{s-\gamma'},$ and so applying \cite[Proposition 3.9]{AR20},\footnote{We have to apply the interior estimates on smaller and smaller balls, but for simplicity let us always write the estimates with $B_r(x_0)$.} we get
		\begin{align*}
		r^{\gamma'+\varepsilon_0}\left[  u\right] _{C^{\gamma'+\varepsilon_0}(B_r(x_0))}&\leq C\left(\left|\left|\frac{u(x_0+r\cdot)}{1+|\cdot|^{\gamma'+\varepsilon_0}}\right|\right|_{L^\infty(\R^n)} + r^{\gamma'+\varepsilon_0}\left(\left[f-b\cdot\nabla u \right]_{C^{\gamma'-1}(B_r(x_0))} \right)\right)\\
		&\leq C(r^{s}+r^{s-\gamma'}),
		\end{align*}
		in view of Lemma \ref{zoomGrowthLemma}, provided $\gamma'-1\leq \theta$, and so
		$$\left[  u\right] _{C^{\gamma'+\varepsilon_0}(B_r(x_0))}\leq Cr^{s-\gamma'-\varepsilon_0}.$$
		With similar argumentation as in Lemma \ref{growthLemma}, we get the wanted estimate for all parameters between $\gamma'$ and $\gamma'+\varepsilon_0$ as well.
	\end{proof}
	\begin{remark}
		Note that the assumption $|u|\leq Cd^s$ can be omitted if we have that the domain is $C^{1,\alpha}.$ Then the growth control follows from the boundary regularity.
	\end{remark}

	\subsection{Boundary Schauder and boundary Harnack estimates}\label{harnackSection}
	
	We continue to study of the behaviour near the boundary of solutions to \eqref{linearEquation}. 
	
	The strategy is to gradually expand the solution in terms of the powers of the distance function with use of the expansion results established in the previous section.
	We improve the accuracy of the expansion of the solution in three steps. First we show how the already established expansion translates to the gradient, which we treat as the right-hand side. Then we need to correct the expansion of the solution, in such a way that its operator evaluation becomes small. In this step, the key ingredient is Theorem \ref{surjectivityResult}. Finally, we apply one of the expansion results from previous section.
	
	Since the procedure differs in three different cases of the regularity of the data we split the result according to the relevant settings.
	Let us point out, that in the proofs we are doing expansions with functions of the form $Qd^{s+k\varepsilon_0+l},$ where $Q$ is a polynomial. We use Theorem \ref{surjectivityResult}, where the power of the distance has to satisfy $s+k\varepsilon_0+l-2s\not\in \N$ and $s+k\varepsilon_0+l-s\not\in \N$. This is equivalent to $s\not\in\Q$. This is the reason we need to assume that $s\in(\frac{1}{2},1)\backslash \mathbb{Q}.$ 
	
	Let us start with the following:
	
	\begin{proposition}\label{harnack1.1}
		Let $s>\frac{1}{2},$ $s\not\in\Q$, and assume $\Omega\subset\R^n$ is a $C^\beta$ domain with $1<\beta<1+s$, $\beta+ s\not\in\N$. Let $L$ be an operator whose kernel $K$ satisfies conditions \eqref{kernelConditions}. Let $u\in L^\infty(\R^n)$ be a solution to \eqref{linearEquation}
		with  $|f|\leq C_1 d^{\beta-1-s}$ and $b\in\R^n.$  
		
		Then
		$$\left|\left|\frac{u}{d^s}\right|\right|_{C^{\varepsilon_0\land (\beta-1)}(\overline{\Omega}\cap B_{1/2})}\leq C\left(C_1+||u||_{L^\infty(\R^n)}\right),$$
		and
		$$\left|\left|\frac{u}{d^s}\right|\right|_{C^{\beta-1}(\partial\Omega\cap B_{1/2})}\leq
		C\left(C_1+||u||_{L^\infty(\R^n)}\right) .$$ 
		The constant $C$ depends only on $n,s,\beta,\Omega$ and ellipticity constants.
		 
		Moreover, there exist constants $Q_z^j$, for $j=0,\ldots,k,$ where $k=\lceil\frac{\beta-1}{\varepsilon_0}\rceil-1$, with $Q_z^j\in C^{\beta-1 - j\varepsilon_0}_z(\partial\Omega\cap B_{1/2}),$ so that for
		$$\tilde{u}_z:= u-Q_z^{k}d^{s+k\varepsilon_0}+\ldots+Q^1_z d^{s+\varepsilon_0},$$
		we have $|L\tilde{u}_z|\leq |x-z|^{\beta-2s}d^{s-1},$
		$$\left|\tilde{u}_z-Q_z^0d^s\right|\leq C|x-z|^{\beta-1+s},$$
		and
		$$\left[ \tilde{u}_z-Q_z^0d^s\right]_{C^{\beta-1+s}(B_{r_1}(x_1))}\leq C\left(\frac{|x_1-z|}{r_1}\right)^{\beta-1+s},$$
		whenever $d(x_1)=2r_1$. The constant $C$ depends only on $n,s,\beta,\Omega$ and ellipticity constants.
	\end{proposition}
	
	\begin{proof}
		Because of the homogeneity of the equation, we can assume that $C_1+||u||_{L^\infty(\R^n)}=1$.
		Notice, that $u$ solves
		$$\left\lbrace\begin{array}{rcl l}
		Lu&=&f-b\cdot\nabla u&\text{in  } \Omega\cap B_1\\
		u &=&0 & \text{in  } \Omega^c\cap B_1,
		\end{array}\right.$$
		where $|f-b\cdot\nabla u|\leq Cd^{s-1},$ thanks to \eqref{gradientEstimate}. Therefore we are in position to apply Lemma \ref{1function1.2}, to get that for every $z\in\partial\Omega\cap B_{3/4}$ there is a constant $Q_z$, such that 
		$$|u-Q_zd^s|\leq C|x-z|^{(s+\varepsilon_0)\land (s+\beta-1)}.$$
		The interior regularity result \cite[Corollary 3.6]{RS16b} together with Lemma \ref{zoomGrowthLemma}  render 
		$$\left[u-Q_zd^s \right]_{C^{2s}(B_r(x_0))} \leq Cr^{\varepsilon_0\land(\beta-1)-s},$$
		whenever $d(x_0)=|x_0-z|=2r$. The corollary can be applied due to Lemma \ref{LonDistance}, which assures that $|L(d^s)|\leq C d^{\beta-1-s}$. 
		In particular, this implies that 
		$$\left|\left|\frac{u}{d^s}\right|\right|_{C^{\varepsilon_0\land (\beta-1)}(\overline{\Omega}\cap B_{1/2})}\leq C,$$
		thanks to Lemma \ref{divisionLemma}.
		Moreover, Lemma \ref{newCoefRegularity} gives that $|Q_z-Q_{z'}|\leq C|z-z'|^{\varepsilon_0\land (\beta-1)}$. Note that $Q_z=\frac{u}{d^s}(z)$, and so the proposition is proven if $\beta-1<\varepsilon_0.$ Therefore we continue assuming $\varepsilon_0<\beta-1$.
		
		Reading the interior regularity estimate for the gradient, we get
		$$\left[\nabla w_i-sQ^i_z\nabla d d^{s-1} \right]_{C^{2s-1}(B_r(x_0))} \leq Cr^{\varepsilon_0-s},$$
		which  by Lemma \ref{growthLemma} gives ($2\varepsilon_0-s<0$)
		$$|\nabla w_i(x)-sQ^i_z\nabla d(x) d^{s-1}(x)|\leq Cd^{2\varepsilon_0-s}(x)$$ 
		on the cone $\mathcal{C}_z := \bigcup\big\{ B_r(x_0);\text{ where }d(x_0)=2r=|x-z|\big\}$.
		This furthermore implies 
		$$|\nabla u(x)-sQ_z\nabla d(z) d^{s-1}(x)|\leq C|x-z|^{\varepsilon_0}d^{s-1}(x),$$
		since $\nabla d\in C^{\varepsilon_0}(\overline{\Omega})$, and on such cone $|z-x|$ is comparable to $d(x)$.
		To derive the same inequality on the neighbourhood of the boundary point, we do the following
		\begin{align*}
		|\nabla u(x)-sQ_z\nabla d(z) d^{s-1}(x)|\leq&\text{ }|\nabla u(x)-sQ_{z'}\nabla d(z') d^{s-1}(x)|+\\
		& +s|Q_{z'}\nabla d(z')-Q^i_{z}\nabla d(z)| d^{s-1}(x)\\
		\leq&\text{ } d^{2\varepsilon_0-s}(x) + C|z-z'|^{\varepsilon_0}d^{s-1}(x)\\
		\leq&\text{ } C|x-z|^{\varepsilon_0}d^{s-1}(x),
		\end{align*}
		and so 
		\begin{equation*}
		|b\cdot\nabla u - sQ_zb\cdot\nabla d(z) d^{s-1}(x)|\leq C|x-z|^{\varepsilon_0}d^{s-1}(x).
		\end{equation*}
		Then we use Lemma \ref{LonDistance}, to find a constant $\tilde{Q}_z = -sQ_zb\cdot\nabla d(z) (c_{s+\varepsilon_0}|\nabla d(z)|^{2s})^{-1}$, which is $C^{\varepsilon_0}_z$, so that 
		$$L(u - \tilde{Q}_zd^{s+\varepsilon_0}) = f - b\cdot\nabla u + \tilde{Q}_zc_{s+\varepsilon_0}|\nabla d(x)|^{2s}d^{s-1}(x) + R(x),$$
		where $|R|\leq d^{2\varepsilon_0-s}$. Combining it with the expansion of the gradient of $u$ and the regularity of $\nabla d$, we get
		$$\big|L(u - \tilde{Q}_zd^{s+\varepsilon_0})\big| \leq C|x-z|^{\varepsilon_0}d^{s-1}(x).$$
		We apply Lemma \ref{1function1.2}, to get a constant $Q_z$, such that 
		$$\left|u-\tilde{Q}_zd^{s+\varepsilon_0}-Qd^s\right|\leq C |x-z|^{2\varepsilon_0\land(\beta-1)+s},$$
		and Lemma \ref{newCoefRegularity} gives that $Q_z$ is $C^{2\varepsilon_0\land(\beta-1)}$ in variable $z$.
		
		We proceed in a similar manner to get the expansion
		$$\tilde{u}_z:= u-Q_z^{k}d^{s+k\varepsilon_0}+\ldots+Q^1_z d^{s+\varepsilon_0},$$
		so that $|L\tilde{u}_z| \leq C|x-z|^{\beta-2s}d^{s-1},$ and $Q_z^{j} \in C^{(k-j+1)\varepsilon_0}_z(\partial\Omega\cap B_{1/2}).$ Note that this is not less regularity than stated.
		It remains to apply Lemma \ref{1function1.2} one last time, to get a constant $Q_z^0$, so that 
		$$|\tilde{u}_z-Q_z^0d^s|\leq C |x-z|^{\beta-1+s},$$
		and
		$$\left[ \tilde{u}_z-Q_z^0d^s\right]_{C^{\beta-1+s}(B_{r_1}(x_1))}\leq C\left(\frac{|x_1-z|}{r_1}\right)^{\beta-1+s},$$
		whenever $d(x_1)=2r_1$.
		Finally, Lemma \ref{newCoefRegularity} gives that $||\frac{u}{d^s}||_{C^{\beta -1}(\partial\Omega\cap B_{1/2})}\leq C,$ as wanted.
		
	\end{proof}

	Let us now show, how the arguments change, when $(k+1)\varepsilon_0$ just exceeds $s$. %One could call it "Break the ass."
	Basically, the difference is that we can not use Lemma \ref{LonDistance}, since it does not provide enough regularity. Instead we apply Theorem \ref{surjectivityResult}, while the rest remains the same. 

	\begin{proposition}\label{harnack1.2}
		Let $s>\frac{1}{2},$ $s\not\in\Q$, and assume $\Omega\subset\R^n$ is a $C^\beta$ domain with $1+s<\beta\leq 1+k_0\varepsilon_0,$ where $k_0$ is the least integer such that $k_0\varepsilon_0>1.$ Let $L$ be an operator whose kernel $K$ is $C^{2\beta+1}(\S^{n-1})$ and satisfies conditions \eqref{kernelConditions}. Let $u\in L^\infty(\R^n)$ solve
		\eqref{linearEquation}
		with $f\in C^{\beta-1-s}(\overline{\Omega})$ and $b\in\R^n.$ Then
		$$\left|\left|\frac{u}{d^s}\right|\right|_{C^{\beta-1}(\partial\Omega\cap B_{1/2})}\leq
			C\left(||f||_{C^{\beta-1-s}(\Omega\cap B_1)}+||u||_{L^\infty(\R^n)}\right)$$ 
		The constant $C$ depends only on $n,s,\beta,\Omega$ and $||K||_{C^{2\beta+1}(\S^{n-1})}.$
		
		Moreover, there exist polynomials $Q_z^j\in\textbf{P}_{\lfloor\beta-1-j\varepsilon_0\rfloor}$,\footnote{The only case when polynomial is non-constant is when $j=1$ and $\beta>2$.} for $j=0,\ldots,k,$ where $k=\lceil\frac{\beta-1}{\varepsilon_0}\rceil-1$, with the $\gamma$-th coefficient $(Q_z^j)^{(\gamma)}\in C^{\beta-1 - j\varepsilon_0-|\gamma|}_z(\partial\Omega\cap B_{1/2}),$ so that for
		$$\tilde{u}_z:= u-Q_z^{k}d^{s+k\varepsilon_0}+\ldots+Q^1_z d^{s+\varepsilon_0},$$
		there is a polynomial $P_z$ so that
		\begin{align*}
			|L\tilde{u}_z-P_z|\leq& C|x-z|^{\beta-1-\varepsilon_0}d^{s-1},\\
			\left[ L\tilde{u}_{z}\right]_{C^{\beta-1-s}(B_{\frac{3}{2}r_1}(x_1))}\leq& C\left(\frac{|x_1-z|}{r_1}\right)^{\beta-s},\\
			\left|\tilde{u}_z -Q_z^0d^s\right|\leq &C|x-z|^{\beta-1+s},\quad\text{and}\\
			\left[ \tilde{u}_{z}-Q_z^0d^s\right]_{C^{\beta-1+s}(B_{\frac{3}{2}r_1}(x_1))}\leq& C\left(\frac{|x_1-z|}{r_1}\right)^{\beta-1+s},
		\end{align*}
		whenever $d(x_1)=2r_1$. The constant $C$ depends only on $n,s,\beta,\Omega$ and ellipticity constants.
	\end{proposition}
	
	\begin{proof}
		We want to continue with the expansion of $u$, where we stopped in the previous proposition. Note that $u$ admits \eqref{interiorEstimate} and \eqref{gradientEstimate}.
		Hence, for $k$ such that $k\varepsilon_0<s<(k+1)\varepsilon_0$, the previous proposition indeed gives existence of constants $Q_z^j\in C^{(k-j)\varepsilon_0}_z(\partial\Omega\cap B_{3/4})$, for $j=0,\ldots, k-1$, so that 
		$$\left|u-Q^{k-1}_zd^{s+(k-1)\varepsilon_0}-\ldots-Q^0_zd^s\right|\leq C|x-z|^{k\varepsilon_0+s}$$
		and 
		$$ \left[ u-Q^{k-1}_zd^{s+(k-1)\varepsilon_0}-\ldots-Q^0_zd^s\right]_{C^{k\varepsilon_0 +s}(B_{\frac{3}{2}r_1}(x_1))}\leq C\left( \frac{|x_1-z|}{r_1} \right)^{k\varepsilon_0+s}, $$
		whenever $d(x_1)=r_1.$  Notice that $k\varepsilon_0+s>1$, since $(k+1)\varepsilon_0>s$. With aid of Lemma \ref{upToBoundaryRegularity} we conclude that in the cone $\mathcal{C}_z$ we have
		$$\left|\nabla\big(u-Q^{k-1}_zd^{s+(k-1)\varepsilon_0}-\ldots-Q^0_zd^s\big) \right|\leq C|x-z|^{k\varepsilon_0+s-1},$$
		as well as
		$$\left[ \nabla\left(u-Q^{k-1}_zd^{s+(k-1)\varepsilon_0}-\ldots-Q^0_zd^s\right) \right]_{C^{k\varepsilon_0 +s-1}(B_{\frac{3}{2}r_1}(x_1))}\leq C\left( \frac{|x_1-z|}{r_1} \right)^{k\varepsilon_0+s}, $$
		whenever $d(x_1)=2r_1.$
		Taking into account that $\nabla d\in C^{k\varepsilon_0}(\overline{\Omega})$ and $|\nabla d ( x)-\nabla d (z)|\leq C|x-z|^{k\varepsilon_0}$, the above implies that for $x\in \mathcal{C}_z,$ we have
		$$\left|\nabla u - ((k-1)\varepsilon_0+s)Q_z^{k-1}\nabla d(z) d^{(k-1)\varepsilon_0+s-1}-\ldots - sQ^0_z\nabla d(z)d^{s-1}\right|\leq C|x-z|^{k\varepsilon_0+s-1},$$
		as well as 
		$$\left[ \nabla u - ((k-1)\varepsilon_0+s)Q_z^{k-1}\nabla d(z) d^{(k-1)\varepsilon_0+s-1}-\ldots - sQ^0_z\nabla d(z)d^{s-1}\right]_{C^{k\varepsilon_0+s-1}}\leq$$$$\leq C\left(\frac{|x_1-z|}{r_1}\right)^{k\varepsilon_0+s}, $$
		in $B_{\frac{3}{2}r_1}(x_1)$,
		thanks to Lemma \ref{reminderRegularity}.
		As in the previous proposition we derive the expansion of the gradient in the neighbourhood, namely for $x\in \Omega\cap B_1,$
		$$\left|\nabla u - ((k-1)\varepsilon_0+s)Q_z^{k-1}\nabla d(z) d^{(k-1)\varepsilon_0+s-1}-\ldots - sQ^0_z\nabla d(z)d^{s-1}\right|\leq C|x-z|^{k\varepsilon_0}d^{s-1}(x).$$ 
		Using Theorem \ref{surjectivityResult}
		we now find constants $\tilde{Q}_{z}^{j}\in C^{(k+1-j)\varepsilon_0}(\partial\Omega)$, $j=1,\ldots,k$, so that 
		$$L(\tilde{Q}_{z}^{j}d^{j\varepsilon_0+s}) = ((j-1)\varepsilon_0+s)Q_{z}^{j-1}\cdot b\cdot\nabla d(z) d^{(j-1)\varepsilon_0+s-1} + R_j + \eta_j,$$
		where $R_j\in C^{k\varepsilon_0-2s+j\varepsilon_0+s}$ and $\eta_j = \phi_{j\varepsilon_0+s} d^{(j-1)\varepsilon_0+s-1}$. The result gives that $\phi_{j\varepsilon_0+s}\in C^{1+k\varepsilon_0}(\overline{\Omega}),$ and by construction $\phi_{j\varepsilon_0+s}(z)=0$. Hence we have $|\eta_j|\leq C|x-z|d^{(j-1)\varepsilon_0+s-1}\leq |x-z|^{k\varepsilon_0}d^{s-1}(x),$ and by Lemma \ref{reminderRegularity}  also $\left[ \eta_j\right]_{C^{k\varepsilon_0+s-1}(B_{\frac{3}{2}r_1}(x_1))}\leq C\left(\frac{|x_1-z|}{r_1}\right)^{k\varepsilon_0}$. So if we define 
		$$\tilde{u}_z := u - \tilde{Q}_{z}^{k}d^{k\varepsilon_0+s}-\ldots-\tilde{Q}_{z}^{1}d^{\varepsilon_0+s},$$ 
		we have
		\begin{equation*}
		|L\tilde{u}_{z}-P_z|\leq C|x-z|^{k\varepsilon_0}d^{s-1}(x),
		\end{equation*}
		for a suitable Taylor polynomial $P_z$ and
		\begin{equation*}
		\left[ L\tilde{u}_{z}\right]_{C^{k\varepsilon_0+s-1}(B_{\frac{3}{2}r_1}(x_1))}\leq C\left(\frac{|x_1-z|}{r_1}\right)^{k\varepsilon_0+s}.
		\end{equation*}
		Applying Lemma \ref{1function2.2}, we get a polynomial $Q_z$ of degree $\lfloor(k+1)\varepsilon_0\land (\beta-1)\rfloor$, so that
		$$|\tilde{u}_z -Q_zd^s|\leq C|x-z|^{(k+1)\varepsilon_0\land(\beta-1)+s},$$
		and
		$$\left[ \tilde{u}_{z}-Q_zd^s\right]_{C^{(k+1)\varepsilon_0\land(\beta-1)+s}(B_{\frac{3}{2}r_1}(x_1))}\leq C\left(\frac{|x_1-z|}{r_1}\right)^{(k+1)\varepsilon_0\land(\beta-1)+s}.$$
		With aid of Lemma \ref{newCoefRegularity}, we get the wanted regularity of the coefficients, in particular the zero order satisfies $$||(Q_z)^{(0)}||_{C^{(k+1)\varepsilon_0\land(\beta-1)}(\partial\Omega\cap B_{1/2})} = \left|\left|\frac{u}{d^s}\right|\right|_{C^{(k+1)\varepsilon_0\land(\beta-1)}(\partial\Omega\cap B_{1/2})} \leq C.$$
		If $\beta-1\leq(k+1)\varepsilon_0$ we are done, otherwise we continue following the same steps.
	\end{proof}

	When regularity becomes greater than $2$, it becomes harder to translate the expansion of the function to the expansion of its gradient, since the polynomials in the expansions start getting non-constant terms. This is where Lemma \ref{A10} and Corollary \ref{generalisedGrowthLemma} become useful. Note also, that these results give a slightly worse estimate, which is the reason for establishing the boundary regularity result - Lemma \ref{boundaryRegGeneralised} which plays the key role when establishing Lemma \ref{1function2.2} and Lemma \ref{2functions3.2}.

	\begin{proposition}\label{harnack1.3}
		Let $s>\frac{1}{2},$ $s\not\in\Q$, and assume $\Omega\subset\R^n$ is a $C^\beta$ domain with $\beta>k\varepsilon_0,$ for $k\varepsilon_0>1>(k-1)\varepsilon_0$, and $\beta\pm s\not\in\N$. Let $L$ be an operator whose kernel $K$ is $C^{2\beta+1}(\S^{n-1})$ and satisfies conditions \eqref{kernelConditions}. Let $u\in L^\infty(\R^n)$ be a solution to
		\eqref{linearEquation}
		with $f\in C^{\beta-1-s}(\overline{\Omega})$ and $b\in\R^n.$ Then
		$$\left|\left|\frac{u}{d^s}\right|\right|_{C^{\beta-1}(\partial\Omega\cap B_{1/2})}\leq
		C\left(||f||_{C^{\beta-1-s}(\Omega\cap B_1)}+||u||_{L^\infty(\R^n)}\right)$$ 
		The constant $C$ depends only on $n,s,\beta,\Omega$ and $||K||_{C^{2\beta+1}(\S^{n-1})}.$
		
		Moreover, there exist polynomials $Q_z^{j,l}\in\textbf{P}_{\lfloor\beta-1-j\varepsilon_0-l\rfloor}$, for $j,l\geq0,$ such that $j\varepsilon_0+l< \beta-1$, with the $\gamma$-th coefficient $(Q_z^{j,l})^{(\gamma)}\in C^{\beta-1 - j\varepsilon_0-l-|\gamma|}_z(\partial\Omega\cap B_{1/2}),$ so that for
		$$\tilde{u}_z := u - \sum_{j\geq1,l\geq0}Q^{j,l}_{z} d^{j\varepsilon_0+l+s}$$
		there is a polynomial $P_z$ so that
		\begin{align*}
		|L\tilde{u}_z-P_z|\leq& C|x-z|^{\beta-1-\varepsilon_0+s}d^{-1},\\
		\left[ L\tilde{u}_{z}\right]_{C^{\beta-1-s}(B_{\frac{3}{2}r_1}(x_1))}\leq& C\left(\frac{|x_1-z|}{r_1}\right)^{\beta-s},\\
		|\tilde{u}_z -Q_z^{0,0}d^s|\leq &C|x-z|^{\beta-1+s},\quad\text{and}\\
		\left[ \tilde{u}_{z}-Q_z^{0,0}d^s\right]_{C^{\beta-1+s}(B_{\frac{3}{2}r_1}(x_1))}\leq& C\left(\frac{|x_1-z|}{r_1}\right)^{\beta-1+s},
		\end{align*}
		whenever $d(x_1)=2r_1$. The constant $C$ depends only on $n,s,\beta,\Omega$ and ellipticity constants.
	\end{proposition}

	\begin{proof}
		We apply the previous proposition, to get polynomials $Q_z^j\in\textbf{P}_{\lfloor(k-j)\varepsilon_0\rfloor}$ with the $\gamma$-th coefficient 
		$(Q^j_z)^{(\gamma)}\in C^{(k-j)\varepsilon_0-|\gamma|}$, for $j=0,\ldots,k-1$ so that
		$$|u-Q^{k-1}_zd^{(k-1)\varepsilon_0+s}-\ldots-Q^0_zd^s|\leq C|x-z|^{k\varepsilon_0+s},$$
		and
		$$\left[ u-Q^{k-1}_zd^{(k-1)\varepsilon_0+s}-\ldots-Q^0_zd^s\right]_{C^{k\varepsilon_0+s}(B_{\frac{3}{2}r_1}(x_1))}\leq C\left(\frac{|x_1-z|}{r_1}\right)^{k\varepsilon_0+s},$$
		whenever $d(x_1)=2r_1.$
		
		We proceed with applying Corollary \ref{generalisedGrowthLemma}, we get
		$$\left|\nabla \left(u-Q^{k-1}_zd^{(k-1)\varepsilon_0+s}-\ldots-Q^0_zd^s\right)\right|\leq C|x-z|^{k\varepsilon_0+s}d^{-1}(x),$$ 
		while 
		$$\left[ \nabla \left(u-Q^{k-1}_zd^{(k-1)\varepsilon_0+s}-\ldots-Q^0_zd^s\right)\right]_{C^{k\varepsilon_0+s-1}(B_{r_1}(x_1))}\leq C\left(\frac{|x_1-z|}{r_1}\right)^{k\varepsilon_0+s}$$
		follows straightforward from the interior estimate above.
		Taking into account that $\nabla d - T^1_z(\nabla d)$ is $C^{k\varepsilon_0}\cap O(|x-z|^{k\varepsilon_0}),$ we get the refined expansion of the form
		$$\left|\nabla u - \sum_{j,l\geq0}^{j\varepsilon_0+l<k\varepsilon_0} P_z^{j,l}d^{j\varepsilon_0+l+s-1}\right|\leq C|x-z|^{k\varepsilon_0+s}d^{-1}(x),$$
		where $P_{z}^{j,l}$ is a polynomial of degree $\lfloor k\varepsilon_0-j\varepsilon_0-l\rfloor$, with the $\alpha-$th coefficient being $C^{k\varepsilon_0-j\varepsilon_0-l-|\alpha|}_z$ smooth (multiplying a polynomial with this property with the Taylor polynomial of $\nabla d$ preserves this property). Using Lemma \ref{reminderRegularity} we also get the interior part of the estimate
		$$\left[ \nabla u - \sum_{j,l\geq0}^{j\varepsilon_0+l<k\varepsilon_0} P_{z}^{j,l}d^{j\varepsilon_0+l+s-1}\right]_{C^{k\varepsilon_0+s-1}(B_{r_1}(x_1))}\leq C\left(\frac{|x_1-z|}{r_1}\right)^{k\varepsilon_0+s}.$$
		Using Theorem \ref{surjectivityResult}
		we find polynomials $\tilde{Q}^{j,l}_{z},$ $j\geq 1,l\geq0$ of degree $\lfloor (k+1)\varepsilon_0-j\varepsilon_0-l\rfloor$, with the $\alpha-$th coefficient being $C^{(k+1)\varepsilon_0-j\varepsilon_0-l-|\alpha|}_z$ smooth, so that
		$$L(\tilde{Q}^{j,l}_{z} d^{j\varepsilon_0+l+s}) = P^{j-1,l}_{z} d^{(j-1)\varepsilon_0+l+s-1} + R_{j,l}+\eta_{j,l},$$
		with $R_{j,l}\in C^{k\varepsilon_0-s+j\varepsilon_0+l}(\overline{\Omega})$ (since $j\geq 1,$ the power is greater than $k\varepsilon_0 + s-1$) and $\eta_{j,l} = \phi_{j,l}d^{(j-1)\varepsilon_0+l+s-1}$, and $\phi_{j,l}\in C^{1+k\varepsilon_0}\cap O(|x-z|^{\lceil (k+1)\varepsilon_0-j\varepsilon_0-l\rceil})$ so the worst is when $j=1,l=0$, when we get $|\eta_{j,l}|\leq C|x-z|^{k\varepsilon_0}d^{s-1}$, which gives $\left[ \eta_{j,l}\right]_{C^{k\varepsilon_0+s-1}(B_{r_1}(x_1))}\leq C\left(\frac{|x_1-z|}{r_1}\right)^{k\varepsilon_0} $. 
		Hence, defining 
		$$\tilde{u}_z := u - \sum_{j\geq1,l\geq0}\tilde{Q}^{j,l}_{z} d^{j\varepsilon_0+l+s}$$
		and refining the estimates with Lemma \ref{reminderRegularity} we have
		$$|L\tilde{u}_z - P_z|\leq C |x-z|^{k\varepsilon_0+s}d^{-1}(x),$$
		for a suitable Taylor polynomial $P_z$ and
		$$\left[ L\tilde{u}_z\right]_{C^{k\varepsilon_0+s-1}(B_{r_1}(x_1))}\leq C\left(\frac{|x_1-z|}{r_1}\right)^{k\varepsilon_0+s},$$
		whenever $d(x_1)=2r_1.$ 
		We apply Lemma \ref{1function2.2}, we get a polynomial $Q_z$ of degree $\lfloor(k+1)\varepsilon_0\land (\beta-1)\rfloor$, so that
		$$|\tilde{u}_z -Q_zd^s|\leq C|x-z|^{(k+1)\varepsilon_0\land(\beta-1)+s},$$
		and
		$$\left[ \tilde{u}_{z}-Q_zd^s\right]_{C^{(k+1)\varepsilon_0\land(\beta-1)+s}(B_{\frac{3}{2}r_1}(x_1))}\leq C\left(\frac{|x_1-z|}{r_1}\right)^{(k+1)\varepsilon_0\land(\beta-1)+s}.$$
		With aid of Lemma \ref{newCoefRegularity}, we get the wanted regularity of the coefficients of $Q_z$, in particular the zero order satisfies $$\left|\left|(Q_z)^{(0)}\right|\right|_{C^{(k+1)\varepsilon_0\land(\beta-1)}(\partial\Omega\cap B_{1/2})} = \left|\left|\frac{u}{d^s}\right|\right|_{C^{(k+1)\varepsilon_0\land(\beta-1)}(\partial\Omega\cap B_{1/2})} \leq C.$$
		If $\beta-1\leq(k+1)\varepsilon_0$ we are done, otherwise we continue following the same steps.
	\end{proof}

	We can now also provide the proof of Theorem \ref{Schauder}.
	
	\begin{proof}[Proof of Theorem \ref{Schauder}] 
		It is a special case of Proposition \ref{harnack1.3}.
	\end{proof}

	Now we turn to the ratio of two solutions to \eqref{linearEquation}. In the following result we establish the optimal regularity of the quotient up to the boundary.

	\begin{proposition}\label{harnack2.1}
		Let $s>\frac{1}{2},$ and assume $\Omega\subset\R^n$ is a $C^1$ domain. Let $L$ be an operator whose kernel $K$  satisfies conditions \eqref{kernelConditions}. For $i=1,2$ let $u_i\in L^\infty(\R^n)$ satisfy
		$$|u_i|\leq C_1 d^{s},$$
		and solve
		$$\left\lbrace\begin{array}{rcl l}
		Lu_i+b\cdot\nabla u_i&=&f_i&\text{in  } \Omega\cap B_1\\
		u_i& =&0 & \text{in  } \Omega^c\cap B_1,
		\end{array}\right.$$
		with $|f_i|\leq K_0d^{s-1}$ and $b\in\R^n.$ Assume  $C_2d^s\geq u_2\geq c_2 d^s$ for some positive $C_2,c_2$. Then
		$$\left|\left|\frac{u_1}{u_2}\right|\right|_{C^{\varepsilon_0}(\overline{\Omega}\cap B_{1/2})}\leq
		C\left(K_0+||u||_{L^\infty(\R^n)}\right)$$ 
		The constant $C$ depends only on $n,s,\beta,\Omega,C_1,C_2,c_2,$ and ellipticity constants.
	\end{proposition}
	
	\begin{proof}
		In particular, $u_i$ solve 
		$$\left\lbrace\begin{array}{rcl l}
		Lu_i&=&f_i-b\cdot\nabla u_i&\text{in  } \Omega\cap B_1\\
		u_i& =&0 & \text{in  } \Omega^c\cap B_1,
		\end{array}\right.$$
		where $|f_i-b\cdot\nabla u_i|\leq Cd^{s-1},$ thanks to \eqref{gradientEstimate}. Hence Lemma \ref{2functions1.1} assures that for every $z\in\partial\Omega\cap B_{1/2}$  we get the existence of $Q_{z}$, so that
		$$|u_1(x)-Q_{z}u_2(x)|\leq C|x-z|^{\varepsilon_0+s},\quad x\in B_{1/2}(z)$$
		and 
		$$\left[u_1-Q_{z}u_2 \right]_{C^{s+\varepsilon_0}(B_r(x_0))} \leq C ,$$
		whenever $d(x_0)=|x_0-z|=2r.$ The latter, together with assumptions on $u_2$ allow us to apply Lemma \ref{divisionLemma} to get 
		$$\left[u_1/u_2 \right]_{C^{\varepsilon_0}(B_r(x_0))} \leq C ,$$ 
		which gives that the quotient is $C^{\varepsilon_0}(\overline{\Omega}\cap B_{1/2})$
	\end{proof}

	Similarly as in the case of the quotient with the distance function, the higher order expansions provide the higher regularity only on the boundary. But in the case of two solutions, we can make one more step of expansions, and then apply Lemma \ref{2functions3.2}. This yields an improvement of the regularity of size $\varepsilon_0$ with respect to the case of the quotient with the distance function. 

	\begin{proposition}\label{harnack2.2}
		Let $s>\frac{1}{2},$ $s\not\in\Q$, and assume $\Omega\subset\R^n$ is a $C^\beta$ domain with $\beta> 1$ and $\beta\pm s\not\in\N$. Let $L$ be an operator whose kernel $K$ is $C^{2\beta+1}(\S^{n-1})$ and satisfies conditions \eqref{kernelConditions}. For $i=1,2$, let $u_i\in L^\infty(\R^n)$ be two solutions to
		$$\left\lbrace\begin{array}{rcl l}
		Lu_i+b\cdot\nabla u&=&f_i&\text{in  } \Omega\cap B_1\\
		u_i &=&0 & \text{in  } \Omega^c\cap B_1,
		\end{array}\right.$$
		with $K_i:=||f_i||_{C^{\beta-1-s +\varepsilon_0}(\overline{\Omega}\cap B_1)}<\infty$ and $b\in\R^n.$\footnote{We can even allow them to be different for $i=1,2$.} When $\beta<1+s-\varepsilon_0$ we assume $K_i:=||fd^{1+s-\varepsilon_0-\beta}||_{L^\infty(\Omega\cap B_1)}<\infty$ instead. Assume that $C_2d^s \geq u_2\geq c_2d^s$, for some positive $C_2,c_2.$ Then
		\begin{align*}
			\left|\left|\frac{u_1}{u_2}\right|\right|_{C^{\beta-1+\varepsilon_0}(\partial\Omega\cap B_{1/2})}\leq
			C\left(K_1+||u_1||_{L^\infty(\R^n)}\right).
		\end{align*}
		The constant $C$ depends only on $n,s,\beta,\Omega,C_2,c_2$ and $||K||_{C^{2\beta+1}(\S^{n-1})}.$
	\end{proposition}

	\begin{proof}
		We can assume that $K_1+||u_1||_{L^\infty(\R^n)}=1$.
		Depending on $\beta$, we apply the relevant result amongst Proposition \ref{harnack1.1}, Proposition \ref{harnack1.2}, or Proposition \ref{harnack1.3}, to get polynomials $Q^{j,l}_{i,z}\in\textbf{P}_{\lfloor\beta-1-j\varepsilon_0-l\rfloor}$, for $j,l\geq0,$ such that $j\varepsilon_0+l< \beta-1$, with the $\gamma$-th coefficient $(Q_{i,z}^{j,l})^{(\gamma)}\in C^{\beta-1 - j\varepsilon_0-l-|\gamma|}_z(\partial\Omega\cap B_{1/2}),$
		so that
		$$
		\left|u_i - \sum_{j\geq0,l\geq0}Q^{j,l}_{i,z} d^{j\varepsilon_0+l+s}\right|\leq C|x-z|^{\beta-1+s},$$
		and
		$$\left[ u_i - \sum_{j\geq0,l\geq0}Q^{j,l}_{i,z} d^{j\varepsilon_0+l+s}\right]_{C^{\beta-1+s}(B_{\frac{3}{2}r_1}(x_1))}\leq C\left(\frac{|x_1-z|}{r_1}\right)^{\beta-1+s},
		$$
		whenever $d(x_1)=2r_1$. Since the regularity of $f_i$ is $C^{\beta-1-s+\varepsilon_0}(\overline{\Omega}),$ we are able to make one more step when doing the expansions, to get polynomials $\tilde{Q}_{i,z}^{j,l}\in\textbf{P}_{\lfloor\beta-1+\varepsilon_0-j\varepsilon_0-l\rfloor}$, for $j\geq 1,l\geq0,$ such that $j\varepsilon_0+l< \beta-1+\varepsilon_0$, with the $\gamma$-th coefficient $(Q_{i,z}^{j,l})^{(\gamma)}\in C^{\beta-1+\varepsilon_0 - j\varepsilon_0-l-|\gamma|}_z(\partial\Omega\cap B_{1/2}),$ so that for
		$$\tilde{u}_{i,z} := u_i - \sum_{j\geq1,l\geq0}\tilde{Q}^{j,l}_{i,z}d^{j\varepsilon_0+l+s}$$
		there is a polynomial $P_z$ so that
		$$|L\tilde{u}_{i,z}-P_z|\leq \left\lbrace\begin{array}{rl}
		Cd^{\beta-2+s}& \text{if }\beta<1+s-\varepsilon_0\\
		C|x-z|^{\beta-1}d^{s-1}& \text{if }1+s-\varepsilon_0<\beta<2\\
		C|x-z|^{\beta-1+s}d^{-1}& \text{if }2<\beta,
		\end{array}\right.$$
		and when $\beta>1+s-\varepsilon_0$ also
		$$\left[ L\tilde{u}_{i,z}\right]_{C^{\beta-1-s+\varepsilon_0}(B_{\frac{3}{2}r_1}(x_1))}\leq C\left(\frac{|x_1-z|}{r_1}\right)^{\beta-s+\varepsilon_0},$$
		whenever $d(x_1)=r_1.$
		Applying the relevant expansion result, Lemma \ref{2functions1.2}, Lemma \ref{2functions2.2}, or Lemma \ref{2functions3.2}, gives the existence of $Q_z$ of degree $\lfloor \beta-1+\varepsilon_0\rfloor$, so that
		$$|\tilde{u}_{1,z} - Q_z\tilde{u}_{2,z}|\leq C|x-z|^{\beta-1+\varepsilon_0+s}\quad\text{and}\quad \left[ \tilde{u}_{1,z} - Q_z\tilde{u}_{2,z}\right]_{C^{\beta-1+\varepsilon_0+s}(B_r(x_0))}\leq C.$$
		Finally, Lemma \ref{divisionLemma} and Lemma \ref{theUltimateRegularityLemma} assure that
		$$\left|\left|\frac{u_1}{u_2}\right|\right|_{C^{\beta-1+\varepsilon_0}(\partial\Omega\cap B_{1/2})}\leq C,$$
		as wanted.
	\end{proof}

	We conclude this section with noticing that this also provides the proof of Theorem \ref{Harnack}.
	
	\begin{proof}[Proof of Theorem \ref{Harnack}]
		It is a special case of Proposition \ref{harnack2.2}.
	\end{proof}

	\section{Smoothness of the free boundary}
	
	In this section we use the developed tools on the height function 
	$$w:=u-\varphi,$$ 
	for solution $u$ of problem \eqref{obstacleProblem}. Note that in particular, $w$ solves
	\begin{equation}\label{heightFunctionEquation}
		 \left\lbrace \begin{array}{rcll}
		Lw &=& f - b\cdot\nabla w&\text{in } \Omega\cap B_1\\
		w&=&0&\text{in }B_1\cap\Omega^c,\\
		\end{array}\right.
	\end{equation}
	where $f := -(L+b\cdot \nabla)\varphi $ and $\Omega := \{w>0\}.$  The main goal of this section is to prove Theorem \ref{1.1}. Its proof, as well as this section is divided in two parts. In the first one we establish a general result, stating that if the free boundary $\partial\Omega$ and the height function $w$  satisfy
	\begin{align}\label{naturalConditions}
		\begin{split}
		0\in\partial\Omega,\\
		\partial\Omega\cap B_1 \in C^1,\\
		w\in C^{1}(B_1),\\
		|D w|\leq C d^s, \quad x\in B_1,\\
		\partial_\nu w \geq cd^s,\quad c>0, \quad x\in B_1,
		\end{split}
	\end{align}
	where $\nu$ is the normal vector to $\partial\Omega$ at $0$,
	then the free boundary $\partial\Omega\cap B_{1/2}$ is roughly as smooth as the obstacle. If the above conditions hold true for some point $x_0\in\partial\Omega$, we say that $x_0$ is \emph{a regular free boundary point.} Then in the second part we show how to obtain \eqref{naturalConditions} near the regular free boundary points in the case when the operator is the fractional Laplacian. Let us stress, that in the first part the operator can be very general, its kernel only has to satisfy \eqref{kernelConditions}, and some regularity condition. Therefore as soon as the conditions \eqref{naturalConditions} are established for solutions of \eqref{obstacleProblem} in this setting, we get the same regularity of the free boundary. Notice also, that if $\partial\Omega\cap B_1$ is $ C^{1,\alpha},$ then the fourth assumption in \eqref{naturalConditions} follows from the fact that $w\in C^1(B_1)$ and Proposition \ref{optimalBoundaryRegularity}.  
	
	For the height function $w$, we denote $$w_i := \partial_i w,\quad\quad i=1,\ldots,n.$$

	\subsection{General stable operators}
	Suppose that $0$ is a free boundary point, at which the height function $w$ satisfies \eqref{naturalConditions}. Without loss of generality we also assume that the normal vector to the free boundary $\nu(0) = e_n$.
	In order to obtain the higher regularity of the free boundary, we apply the results established in Section \ref{harnackSection} to the quotients $w_i/w_n$, for $i = 1,\ldots,n-1$. Since we can express the normal vector of the free boundary with these quotients (see \cite[Section 5]{AR20}), we get the regularity of the free boundary.  

	\begin{proposition}\label{regularityOfTheFreeBoundary}
		Let $\varphi$ be an $C^{\theta+s+\varepsilon_0}(\R^n)$ obstacle for some $\theta>1+s-\varepsilon_0$ with $\theta\not\in\N$, $\theta+\varepsilon_0\pm s\not \in\N$.
		Let $L$ be an operator whose kernel $K$ is $C^{2\theta+1}(\S^{n-1})$ and satisfies \eqref{kernelConditions}.
		Assume that $0$ is a free boundary point at which the conditions \eqref{naturalConditions} hold true. Then the free boundary is $C^{\theta}$ around $0$.
	\end{proposition}
	
	\begin{proof}
		Note that the partial derivatives $w_i$ solve
		$$ \left\lbrace \begin{array}{rcll}
		Lw_i+ b\cdot\nabla w_i& = &f_i &\text{in } \Omega\cap B_1,\\
		w_i&=&0&\text{in }B_1\cap\Omega^c,\\
		\end{array}\right.$$
		where $f_i = \partial^if\in C^{\theta-1-s+\varepsilon_0}.$ Thanks to Lemma \ref{followingRegularityLemma} we can apply Proposition \ref{harnack2.1} on $w_i$ and $w_n$, to get that the quotient $\frac{w_i}{w_n}$ is $C^{\varepsilon_0}(\overline{\Omega}\cap B_{1/2})$. Therefore the normal vector $\nu\in C^{\varepsilon_0}(\overline{\Omega}\cap B_{1/2})$ and hence the boundary is $C^{1+\varepsilon_0}.$ 
		We proceed with induction. 
		Assume that the free boundary is $C^\beta$ around $0$. Applying Proposition \ref{harnack2.2} to $w_i$ and $w_n$, gives that the quotient $w_i/w_n$ is $C^{\beta-1+\varepsilon_0}(\partial\Omega\cap B_{1/2})$. Therefore the normal vector $\nu\in C^{\beta-1+\varepsilon_0}(\overline{\Omega}\cap B_{1/2})$ and hence the free boundary is in fact $C^{\beta+\varepsilon_0}.$ We can proceed with the induction as long as the regularity of $f_i$ is better than required, $ \theta-1-s+\varepsilon_0\geq \beta-1-s+\varepsilon_0 $, which indeed renders the wanted regularity.
	\end{proof}

	\begin{corollary}\label{infiniteRegularityCorollary}
		Let $L$ be an operator whose kernel $K$ satisfies \eqref{kernelConditions} and is $C^\infty(\mathbb{S}^{n-1}).$ Let $0\in \partial\Omega$ be a free boundary point, and assume that \eqref{naturalConditions} holds. If $\varphi\in C^\infty,$ then the free boundary is $C^\infty$ around $0$.
	\end{corollary}

	\subsection{Fractional Laplacian}
	
	We conclude with establishing conditions \eqref{naturalConditions} for the height function in the case when the operator is the fractional Laplacian. All the work has already been  done in \cite{GPPS17,PP15}, we just show how to translate their results to our setting. 
	
	\begin{lemma}\label{allWeNeedFromPP}
		Let $L$ be the fractional Laplacian and let $\varphi$ be an obstacle in the space $C^{3s}(\R^n)\cap C_0(\R^n),$ which satisfies 
		$$\left((-\Delta)^s\varphi+b\cdot\nabla\varphi\right)^+ \in L^\infty(\R^n).$$ 
		Let $u$ be the solution to problem \eqref{obstacleProblem}. Let 
		Let $0$ be a regular free boundary point with $e_n$ as the normal vector. Then the height function $w:=u-\varphi\in C^{1+s}(\R^n)$, and there exists $r_0>0$ and $\alpha>0$, so that the free boundary is $C^{1,\alpha}$ in $B_{r_0}(0)$. Furthermore, for all $1\leq i\leq n$ the functions $w_i$ satisfy 
		$$|w_i(x)|\leq C d^s(x),\quad x\in B_{r_0}(0) .$$
		and, 
		$$w_n(x)\geq c d^s(x),\quad x\in B_{r_0}(0).$$
	\end{lemma}

\begin{proof}
	In \cite{PP15}, they establish that under these assumptions, the height function  $w\in C^{1+s}(\R^n)$. In \cite{GPPS17} the furthermore show that the homogeneous rescaling of the height function $v_{0}(x,y)$, for a regular free boundary point $0$ with the normal vector $e_n$, converges in $C^{1,\gamma}(\overline{B}_{1/8}^+)$ to
	$$a(x_n+\sqrt{x_n^2+y^2})^s(x_n-s\sqrt{x_n^2+y^2}).$$
	where $a>0$ (See \cite[results 3.7, 3.8, 5.2, 5.3, 5.5 ]{GPPS17}.) Reading this convergence at $y=0$, we get
	$$\frac{1}{r^{1+s}}(u-\varphi)(0+rx)\longrightarrow 2^s(1-s)a (x_n)_+^{1+s}\quad\quad \text{in }C^{1,\gamma}(\overline{B}_{1/8}),$$
	for some positive $\gamma$ and $a$. This implies that the partial derivatives satisfy
	$$\frac{1}{r^s}(u-\varphi)_i(0+rx)\longrightarrow 2^s(1-s^2)a(x_n)_+^{s}\delta_{i,n}, \quad \quad \text{in }L^\infty(\overline{B}_{1/8}).$$
	From this it is not hard to get bounds
	$$||(u-\varphi)_i||_{L^\infty(B_r)}\leq Cr^s,\quad i=1,\ldots,n\quad\text{and}$$
	$$||(u-\varphi)_n||_{L^\infty(B_r)}\geq cr^s,\quad c>0,$$
	for all $r\leq r_0$.
	
	Let us now show, that we also have $|(u-\varphi)_i|\leq Cd^s$ and $|(u-\varphi)_n| \geq c d^s$, for some $c>0$, where $d$ is the generalised distance function to the free boundary. The first one follows from optimal regularity of solutions for the obstacle problem, namely $u$ is $C^{1+s}(\R^n)$ (see \cite[Theorem 1.1]{PP15}),
	and so $w := u-\varphi$ is also $C^{1+s}$ and vanishes outside $\Omega := \{w > 0\}.$ So in $\Omega^c$ we have that also $w_i=0$ for every $i = 1,\ldots,n.$ Therefore, since $w_i\in C^s(\R^n),$ we have that $|w_i(x)| = |w_i(x)-w_i(z)|\leq ||w_i||_{C^s}|x-z|^s\leq Cd^s(x)$, for the closest point $z$ in the free boundary. 
	
	Let now $0$ be a regular free boundary point. In \cite[Theorem 1.3]{GPPS17} they show that around $0$, the free boundary is $C^{1,\alpha}$ for some positive $\alpha$. Therefore in a ball $B_{r_2}(0)$), we can apply \cite[Proposition 3.3]{RS16}, to get that for every $z\in \Gamma\cap B_{r_2/2}(0)$ there exists $Q(z)$, such that 
	\begin{equation}\label{expansionEstimate}
		|w_i(x)-Q(z)d^s(x)|\leq C|x-z|^{s+\alpha},
	\end{equation} 
	provided that $|\nabla w_i|\leq Cd^{s-1}$, which we show later. We already established that around $0$ we have $||w_n||_{L^\infty(b_r(0))}\geq cr^s$ for some positive $c$. This implies, that $Q(0)\neq 0$. Since the function $w$ is non-negative, $Q(0)$ must be positive.
	Estimate \eqref{expansionEstimate} together with Lemma \ref{newCoefRegularity} give that $z\mapsto Q(z)$ is a $C^\alpha$ map, and so continuous, which gives that in a perhaps smaller neighbourhood $B_{\tilde{r}_0}0$ we have $Q(z)\geq c'>0$. Choose now a point $x\in B_{r_0'(0)}$, where $r_0'$ will be specified later. Let $z$ be the closest boundary point, which falls into $B_{\tilde{r}_0}(0),$ so that $Q(z)\geq c'.$ Denote $x=z+t\nu_z$, so $d(x) = t$. Then \eqref{expansionEstimate}, with a triangle inequality gives
	\begin{equation}\label{nondegeneracy}
		w_n(x)\geq Q(z)d^s(x) - C|x-z|^{s+\alpha}\geq c' t^s-Ct^{s+\alpha}\geq \frac{c'}{2}t^s = \frac{c'}{2} d^s(x),
	\end{equation}
	if $t\leq t_0$ suitable chosen ($t_0\leq \left(\frac{c'}{2C}\right)^{1/\alpha})$. Hence we got such $r_0$ (the minimum of above constraints), that $w_n\geq c_1 d^s$ in $B_{r_0}(0),$ and $\Gamma\cap B_{r_0}(0)$ is $C^{1,\alpha}$.
\end{proof}

This provides the last ingredients for proving Theorem \ref{1.1}.

\begin{proof}[Proof of Theorem \ref{1.1}]
	The claim follows straight-forward from Lemma \ref{allWeNeedFromPP} and Corollary \ref{infiniteRegularityCorollary}.
\end{proof}

	\appendix
	\section{}
	
	\begin{lemma}\label{generalisedA9}
		Let $\alpha\in(-1,0)$, $\beta\in\R$, $r>0$ and $x_0\in B_1$ be such that $(x_0)_n>2r$. Then there exists a constant $C$, such that for every $x\in B_{r/2}(x_0)$ it holds
		$$\int_{B_1\backslash B_r(x)} (z_n)_+^{\alpha} |z-x|^{-n+\beta}dz\leq Cr^{\alpha+\beta},\quad\text{if }\alpha+\beta<0 \quad\text{and}$$
		$$\int_{B_r(x)} (z_n)_+^{\alpha} |z-x|^{-n+\beta}dz\leq Cr^{\alpha+\beta},\quad\text{if }\beta>0.$$
	\end{lemma}
	\begin{proof}
		The proof is exactly the same as the one of \cite[Lemma A.9]{AR20}.
	\end{proof}
	
	\begin{lemma}\label{upToBoundaryRegularity}
		Let $\Omega$ be a domain in $\R^n$ with $0\in\partial\Omega$ and $d$ the distance function to the boundary. Assume function $f$ satisfies 
		$$|f(x)|\leq C|x|^\alpha,$$
		and 
		$$\left[ f\right] _{C^\alpha(B_r(x_0))}\leq C,$$
		whenever $d(x_0) = 2r = |x_0|$.
		
		Denote $\mathcal{C} = \cup_{d(x_0)=|x_0|} B_r(x_0)$. Then $f\in C^\alpha(\overline{\mathcal{C}})$ with 
		$$|D^jf|\leq C |x|^{\alpha-j},\quad j\leq\alpha.$$
	\end{lemma}

	\begin{proof}
		Let us first show, that $f\in C^\alpha(\overline{\mathcal{C}}).$ Choose therefore a multi-index $\gamma$ of order $\lfloor\alpha\rfloor$ and points $x,y\in\mathcal{C}.$ Denote $x',y'$ the points such that $d(x')=|x'|$ and  $d(y')=|y'|$, so that $x\in B_{|x'|/2}(x')$ and $y\in B_{|y'|/2}(y')$, so that $|x'-y'|\leq |x-y|$, $|x'-x|\leq |x-y|$ and $|y-y'|\leq |x-y|$. For this we need to assume that $x$ and $y$ do not lie in any of the balls $B_r(x_0)$. Otherwise the desired property holds true by assumption. For simplicity let us also assume that $x'$ and $y'$ lie on the same line as $0$. Assume with out loss of the generality that $|x'|<|y'|.$ Denote $x_i = 2^ix$. Then 
		$$|\partial^\gamma f(x)-\partial^\gamma f(y)|\leq |\partial^\gamma f(x)-\partial^\gamma f(x')|+|\partial^\gamma f(x')-\partial^\gamma f(y')|+|\partial^\gamma f(y')-\partial^\gamma f(y)|.$$
		The first term is bounded with $C|x-x'|^{\langle \alpha\rangle}\leq|x-y|^{\langle \alpha\rangle},$ and similarly the last term.
		For the second one, we choose $K\in\N$ so that $2^K|x'|\leq |y'|\leq 2^{K+1}|x'|$, and $2^K|x'|\leq 2|x'-y'|$ and compute
		$$|\partial^\gamma f(x')-\partial^\gamma f(y')|\leq \sum_{i=1}^K |\partial^\gamma f(x_{i-1})-\partial^\gamma f(x_i)| + |\partial^\gamma f(x_K)-\partial^\gamma f(y')|$$
		$$\leq \sum_{i=1}^K C|x_i-x_{i-1}|^{\langle \alpha\rangle} + C|x_K-y'|^{\langle \alpha\rangle} \leq \sum_{i=1}^K C(2^i|x'|)^{\langle \alpha\rangle} + C|x_K-y'|^{\langle \alpha\rangle}$$
		$$\sum_{i=1}^K C(2^i2^{-K}|x'-y'|)^{\langle \alpha\rangle} + C|x'-y'|^{\langle \alpha\rangle}\leq  C|x'-y'|^{\langle \alpha\rangle},$$
		because the series we get is summable. Noticing that $|x'-y'|\leq |x-y|$ finishes the proof. 
		
		The only exception is when for example $x=0$. Then we do the same procedure, with $x_i = 2^{-i}y.$
		
		Once we know that $f\in C^\alpha(\overline{C})$, the growth control implies that $D^{j}f (0) = 0$ for all $j\leq \lfloor\alpha\rfloor$. Since $D^{\lfloor\alpha\rfloor}f$ is a $C^{\langle\alpha\rangle}$ function, we get
		$|D^{\lfloor\alpha\rfloor}f| \leq C |x|^{\langle\alpha\rangle}.$ Integrating this iteratively, we get the others.
	\end{proof}
	
	\begin{lemma}\label{zoomGrowthLemma}
		Let $u$ be a bounded function, satisfying $|u(x)|\leq C|x|^\alpha$ for some $\alpha>0$. Let $\beta>\alpha$, and $x_0\in B_{1/2}$ with $|x|=2r.$ Then the function $u_r(x):=u(x_0+rx)$ satisfies 
		$$\left|\left|\frac{u_r}{1+|\cdot|^\beta}\right|\right|\leq C r^\alpha.$$
	\end{lemma}
\begin{proof}
	We compute 
	$$|u_r(x)| = |u(x_0+rx)|\leq C|x_0+rx|^\alpha \leq Cr^\alpha (1+|x|)^\alpha \leq Cr^\alpha (1+|x|^\beta).$$
\end{proof}
	
	\begin{lemma}\label{lemmaA8}
		Let $f\in C^{\alpha}(\overline{B_1})$ with $\alpha >  1$. Then there exists a $C^{\alpha-1}$ map $g\colon \overline{B_1} \to \R^n$, such that $$f(x)-f(0) = x\cdot g(x).$$
		Moreover, $||g||_{C^{\alpha-1}}\leq C||f||_{C^\alpha}$.
	\end{lemma}
	\begin{proof}
		Write $f(x) - f(0) = \int_0^1 \partial_t f(tx) dt =  \int_0^1 x\cdot \nabla f(tx)  dt,$ and hence we can define $g(x) = \int_0^1 \nabla f(tx) dt.$ For any multi-index $k$ with $|\gamma|\leq \alpha -1$, we have $\partial^\gamma g(x) = \int_0^1 \nabla \partial^\gamma f (tx) t^{|\gamma|}dt.$ Hence we get 
		$||D^j g||_{L^\infty}\leq C||D^{j+1}f||_{L^\infty},$ for every $j\leq \alpha-1$ and for $|\gamma| = \left[ \alpha-1\right] $
		$$ |\partial^\alpha g(x_1) - \partial^\alpha g(x_2)| \leq \int_0^1 |\nabla \partial^\alpha f(tx_1) - \nabla \partial^\alpha f(tx_2)| t^{|\alpha|}dt $$
		$$\leq C\left[f \right]_{C^\alpha}\int_0^1 |tx_1 - tx_2|^{\left\langle\alpha\right\rangle} t^{|\alpha|}dt \leq C\left[f \right]_{C^\alpha}|x_1-x_2|^{\left\langle\alpha\right\rangle}.$$
	\end{proof}

	\begin{lemma}\label{growthLemma}
		Let $\Omega$ be $C^1$ domain and $f\colon \Omega\to\R$ a $C(\Omega)$ function, which for some $\alpha\in(0,1)$, $\beta<-\alpha$ satisfies
		$$\left[f \right]_{C^{\alpha}(B_r(x_0))} \leq C_0r^{\beta},$$ for any $x_0$, $r$ such that $B_{2r}(x_0)\subset \Omega$ and $C_0$ independent of $x_0,r$. 
		
		Then we have the following bound
		$$|f(x)|\leq CC_0d^{\alpha+\beta}(x).$$
	\end{lemma}
	
	\begin{proof}
		Choose $r_0>0$, so that for $K = \overline{\{y\in\Omega; \text{ }d(y,\partial\Omega)>r_0\}},$ and every $z\in\partial\Omega$ the intersection $K\cap\{z+t\nu_z;\text{ } t\in\R\}$ is non-empty. Since $K$ is compact and $f$ continuous on $K$, it is bounded there. Choose now a point $x$ in $\Omega$. Let $z$ be the closest boundary point, and $r= |x-z|$. Denote $x_i =  z + 2^ir\nu_z$, so that $x_0 = x$ and $x_k\in K$. Then, since $x_i$ and $x_{i+1}$ are both in the ball $B_{2^{i-1}r}(y_i)$ for $y_i = 1/2(x_i+x_{i+1})$, by assumption we have 
		$$|f(x_i)-f(x_{i-1})|\leq C_0(2^{i-2}r)^{\beta}|x_i-x_{i+1}|^\alpha \leq C_0(2^{i-2}r)^{\beta}(2^{i-1}r)^\alpha = C_02^{-2\beta}r^{\alpha+\beta}2^{(\alpha+\beta)(i-1)}.$$
		Hence, summing a geometric series and using that $x_k\in K$, we get
		$$|f(x)|\leq \sum_{i=1}^k|f(x_i)-f(x_{i-1})| + |f(x_k)|\leq C_{\alpha+\beta}C_0 r^{\alpha+\beta}+ C\leq CC_0r^{\alpha+\beta}.$$
	\end{proof}

	\begin{lemma}\label{divisionLemma}
		Let $0\in\partial\Omega$ and let $f,g,Q$ be functions on $\Omega$ which satisfy $|f-Qg|\leq C|x|^{a+b}$, and $\left[ f-Qg\right] _{C^{a+b}(B_r(x_0))}$, whenever $d(x_0)=|x_0-0|=2r$. Assume $g$ satisfies $|D^kg^{-1}|\leq C_k d^{-b-k}$, for all $0\leq k \leq \lfloor a\rfloor+1$. 
		
		Then we have $|\frac{f}{g}-Q|\leq C|x|^{a}$ on $\mathcal{C}_0$ and $\left[ \frac{f}{g}-Q\right] _{C^a(B_r(x_0))}\leq C$.
	\end{lemma}
	
	\begin{proof}
		The first estimate is just dividing by $g$ and taking into account its growth. We have to restrict ourselves to the cone, so that $|x|$ and $d$ become comparable.
		
		For the second one, pick $|\gamma|=\lfloor a\rfloor$ and compute
		$$|\partial^\gamma((f-Qg)g^{-1})(x)-\partial^\gamma((f-Qg)g^{-1})(y)|\leq $$
		$$\leq \sum_{\alpha\leq\gamma} \binom{\gamma}{\alpha}|\partial^\alpha(f-Qg)(x)-\partial^\alpha(f-Qg)(y)|\cdot|\partial^{\gamma-\alpha}g^{-1}(x)|+$$
		$$+\sum_{\alpha\leq\gamma}\binom{\gamma}{\alpha}|\partial^{\gamma-\alpha}g^{-1}(x)-\partial^{\gamma-\alpha}g^{-1}(y)|\cdot|\partial^\alpha(f-Qg)(y)|.$$
		First, notice that the assumptions on the growth and regularity $f-Qg$ imply that $f-Qg\in C^{a+b}(\mathcal{C}_0)$, with $D^j(f-Qg)(0) = 0$ for $j\leq \lfloor a+b\rfloor$, and so on $B_r(x_0)$ we have $|D^j(f-Qg)|\leq Cr^{a+b-k}$. Now we estimate the above expression with
		$$\sum_{\alpha\leq\gamma}C||D^{|\alpha|+1}(f-Qg)||_{L^\infty(B_r(x_0))}|x-y|||D^{|\gamma-\alpha|}g^{-1}||_{L^\infty(B_r(x_0))} +$$
		$$+\sum_{\alpha\leq\gamma} ||D^{|\gamma-\alpha|+1}g^{-1}||_{L^\infty(B_r(x_0))}|x-y|||D^{|\alpha|}(f-Qg)||_{L^\infty(B_r(x_0))}\leq$$
		$$\leq \left(Cr^{a+b-|\alpha|-1}r^{1-\langle a\rangle}r^{-b-|\gamma-\alpha|}\right)|x-y|^{\langle a\rangle}  +\left(Cr^{-b-|\gamma-\alpha|-1}r^{1-\langle a\rangle}r^{a+b-|\alpha|}\right)|x-y|^{\langle a\rangle}$$
		$$\leq C|x-y|^{\langle a\rangle}.$$
		Note that in the case when $\lfloor a+b\rfloor = \lfloor a\rfloor$, and when $\alpha=\gamma$ we have to use the regularity from the assumption for $(f-Qg)$, since we do not have the estimate on $D^{|\gamma|+1}(f-Qg)$, but we end up with the same powers of $r$.
	\end{proof}

	%TODO: in the lema below i had assumptions on u and v to have bounded k-th derivatives with D^s-k. maybe it has to be there.

	\begin{lemma}\label{theUltimateRegularityLemma}
		Let $\Omega$ be a domain of class $\beta'$ and let for every boundary point $z\in \partial\Omega\cap B_1$ the following hold true
		$$\left| \frac{u(x) - \sum_{k} P^k_z(x-z)d^{s+p_k}(x)}{v(x) - \sum_{k} R^{k}_z(x-z)d^{s+p_k}(x)}-Q_z(x-z) \right|\leq C_0|x-z|^{\beta'},\quad x \in \mathcal{C}_z$$
		and 
		$$\left[ \frac{u - \sum_{k} P^{k}_zd^{s+p_k}}{v - \sum_{k} R^{k}_zd^{s+p_k}}-Q_z\right]_{C^{\beta'}(B_r(x_0))}\leq C_0,$$
		for some $p_k>0$ and polynomials $P^k_z,R^{k}_z\in\textbf{P}_{\lfloor \beta'-p_k\rfloor}$ whose coefficients of order $\alpha$ are  $C^{\beta'-p_k-|\alpha|}_z(\partial\Omega\cap B_1)$ and a polynomial $Q_z\in\textbf{P}_{\lfloor\beta'\rfloor}.$ Assume that $|v|\leq C_1 d^s$. 
		
		Then the coefficients of $Q_z$ satisfy 
		$$\left|\left|Q_z^{(\alpha)}\right|\right|_{ C^{\beta'-|\alpha|}_z(\partial\Omega\cap B_{1/2})}\leq C.$$
		The constant $C$ depends only on $n,s,\beta',C_0,C_1,$ $\left|\left|(R^{k}_z)^{(\alpha)}\right|\right|_{ C^{\beta'-p_k-|\alpha|}_z(\partial\Omega\cap B_1)}$ and\\ $\left|\left|(P^k_z)^{(\alpha)}\right|\right|_{ C^{\beta'-p_k-|\alpha|}_z(\partial\Omega\cap B_1)}.$
	\end{lemma}
	\begin{proof}
		With similar argument as in the proof of Lemma \ref{newCoefRegularity} we argue that the coefficients of $Q_z$ are uniformly bounded.
		
		Let us denote $\tilde{u} = u(x) - \sum_{k} P^k_z(x-z)d^{s+p_k}(x)$ and $\tilde{v}$ analogously. The assumptions give that on the cone $\mathcal{C}_z$, the function $\eta_z:=\frac{\tilde{u}}{\tilde{v}}-Q_z$ is of class $C^{\beta'}$ with $D^{\lfloor \beta'\rfloor}\eta_z (z) = 0$. Hence for $|\gamma|=\lfloor\beta'\rfloor$, we have 
		$$Q_z^{(\gamma)}= \partial^\gamma Q_z = \partial^\gamma (\tilde{u}/\tilde{v})-\partial^\gamma\eta_z.$$
		Now choose $N\in\N$ big enough and "take" incremental quotient of the above equation in variable $z$ of increment $h$. We do it through the parametrisation of the boundary, which can be taken of class $C^{\beta'}$. When we get get the boundary points $z_i$, so that $\Delta^N_h Q_z^{(\gamma)} = \sum_{i=0}^N (-1)^i \binom{N}{i}Q_{z_i}^{(\gamma)}$. Then choose $x\in \cap_{i=0}^N \mathcal{C}_{z_i},$ so that $d(x)\leq C|h|$. Then we have 
		$$\Delta^N_h Q_z^{(\gamma)} = \Delta^N_h\partial^\gamma \frac{\tilde{u}}{\tilde{v}}-\Delta^N_h\partial^\gamma \eta_z.$$
		First, since for all $i$ we have $|\partial^\gamma\eta_{z_i}(x)|\leq C|x-z_i|^{\beta'-|\gamma|}$ we can bound \begin{equation}\label{etaEstimate}
			|\Delta^N_h\partial^\gamma \eta_z|\leq C_N d(x)^{\beta'-|\gamma|}\leq C_N |h|^{\beta'-|\gamma|}.
		\end{equation}
		Hence, let us focus on the first term only. We compute
		$$\partial^\gamma \frac{\tilde{u}}{\tilde{v}} = \sum_{\alpha\leq \gamma} \binom{\gamma}{\alpha} \partial^\alpha \tilde{u}\partial^{\gamma-\alpha}\frac{1}{\tilde{v}} = \sum_{\alpha\leq \gamma} \binom{\gamma}{\alpha} \partial^\alpha \tilde{u} \sum_{K=0}^{|\gamma-\alpha|}\frac{1}{\tilde{v}^{K+1}}\sum_{\delta_1+\ldots+\delta_K=\gamma-\alpha}c(\delta,\alpha,\gamma)\partial^{\delta_1}\tilde{v}\cdot\ldots\cdot \partial^{\delta_K}\tilde{v},$$
		where when $K=0$ the inner sum has to be understood as $1$ if also $|\gamma-\alpha|=0$, otherwise $0$. Also, the constant $c(\delta,\alpha,\gamma)$ is non-zero if and only if all $\delta_i$ are non-zero.
		So we need to estimate 
		$$\Delta^N_h\left(\partial^\alpha\tilde{u}\frac{\partial^{\delta_1}\tilde{v}\ldots\partial^{\delta_K}\tilde{v}}{\tilde{v}^{K+1}}\right)$$
		which we split even further with the Leibnitz rule into 
		$$ \sum_{N_1+\ldots N_{K+2}=N}\binom{N}{N_1\ldots N_{K+2}}\Delta^{N_1}_h\partial^\alpha \tilde{u}
		\Delta^{N_2}_h\partial^{\delta_1} \tilde{v}\ldots\Delta^{N_{K+1}}_h\partial^{\delta_K} \tilde{v}\Delta^{N_{K+2}}_h\frac{1}{\tilde{v}^{k+1}}$$
		and furthermore 
		$$\Delta^{N_{K+2}}_h\frac{1}{\tilde{v}^{k+1}} = \sum_{M=0}^{N_{K+2}} \frac{1}{\tilde{v}^{(K+1)(M+1)}}\sum_{L_1+\ldots+L_M=N_{K+2}}\Delta^{L_1}_h\tilde{v}^{K+1}\ldots \Delta^{L_M}_h\tilde{v}^{K+1}.$$
		With final Leibnitz rule on all of the above powers, we conclude that up to some constants
		$$\Delta^{N}_h\partial^\gamma \frac{\tilde{u}}{\tilde{v}} = 
		\sum_{\alpha}\sum_{K}\sum_{\delta}\sum_{N_l}\sum_M\sum_L  \Delta^{N_1}_h(\partial^\alpha \tilde{u})\Delta^{N_2}_h(\partial^{\delta_1} \tilde{v})\ldots \Delta^{N_{K+1}}_h(\partial^{\delta_K} \tilde{v})\frac{1}{\tilde{v}^{(M+1)(K+1)}}
		$$
		$$\times\prod_{i=1}^M \sum_{L^i_1+\ldots+L^i_{K+1}=L_i} \prod_{j=1}^{K+1}\Delta^{L^i_j}_h \tilde{v}  .
		  $$
		So we need to estimate 
		\begin{equation}\label{factorsHoriblles}
			\Delta^{N_1}_h(\partial^\alpha \tilde{u})\Delta^{N_2}_h(\partial^{\delta_1} \tilde{v})\ldots \Delta^{N_{K+1}}_h(\partial^{\delta_K} \tilde{v})\frac{1}{\tilde{v}^{(M+1)(K+1)}}\prod_{i,j}^{M,K+1} \Delta^{L^i_j}_h\tilde{v}.
		\end{equation}
		
		We treat every of the above factors separately. Starting with the first one, we begin with the following manipulation, where we take out the part, which is important for the finite difference:
		$$\partial^\alpha \tilde{u}(x) = \partial^\alpha u(x) - \partial^\alpha\sum_k  \sum_{\eta} c_{k,z}^\eta(x-z)^\eta d^{s+p_k}(x) $$
		$$= \partial^\alpha u(x) - \sum_k  \sum_{\eta}\sum_\varepsilon\binom{\eta}{\varepsilon} c_{k,z}^\eta(0-z)^\varepsilon \partial^\alpha(x-0)^{\eta-\varepsilon} d^{s+p_k}(x).$$
		Now we take the finite difference, and apply the Leibnitz rule to get (with omitting some constants)
		$$\Delta^{N_1}_h(\partial^\alpha \tilde{u}) = - \sum_k\sum_\eta\sum_\varepsilon\sum_{N'}\Delta^{N'}_hc_{k,z}^\eta \Delta^{N_1-N}_h(z-0)^\varepsilon \cdot \partial^\alpha(x-0)^{\eta-\varepsilon} d^{s+p_k}(x).$$
		When estimating, this gives 
		$$\big|\Delta^{N_1}_h(\partial^\alpha \tilde{u})\big|\leq C \sum_k\sum_\eta\sum_\varepsilon\sum_{N'} |h|^{\rho\land N'}|h|^{(|\varepsilon|\lor (N_1-N'))\land \beta'}|h|^{s+p_k+|\eta-\varepsilon|-|\alpha|},  $$
		where $\rho$ denotes the regularity of the coefficient, so $\rho = \beta' - p_k -|\eta|$. With treating the cases carefully, we can bound it by
		$$\leq C\sum_k\sum_\eta|h|^{(N_1+p_k+s-|\alpha|)\land (\rho + p_k+s+|\eta|-|\alpha|)}\leq C |h|^{s-|\alpha|+N_1\land \beta'}.$$
		Note also, that this is true also when $N_1=0$, due to the assumption on the growth of derivatives of $u$.
		For the other factors, we perform the same estimations. We plug it in \eqref{factorsHoriblles} to get
		$$\left|\Delta^{N}_h\partial^\gamma \frac{\tilde{u}}{\tilde{v}}\right|\leq C|h|^{(K+1)(M+1)s - |\gamma|}|h|^{N_1\land\beta'}|h|^{N_2\land\beta'}\ldots |h|^{N_{K+1}\land\beta'}|h|^{-(K+1)(M+1)s}\prod_{i,j}|h|^{L_{i,j}\land\beta'}.$$
		If we choose $N$ big enough, so that at least one of the minimums give $\beta'$ we get the claim, due to \eqref{etaEstimate}.
		
		To establish the regularity of other coefficients, we proceed with the same procedure, only that now we treat higher order coefficients of polynomial $Q_z$ as reminders. Concretely, for $|\gamma|<\lfloor \beta'\rfloor$ we have
		$$Q_z^{(\gamma)} = \partial^\gamma Q_z - \sum_{\gamma<\gamma'}c_{\gamma,\gamma'} Q_z^{(\gamma')}(x-z)^{\gamma'-\gamma}.$$
		After finite differences, we treat the first term as before to get $|h|^{\beta'-|\gamma|}$, while the second one gives at least the same by already proven regularity (we can treat it as we treated polynomials before). 
		
		The claim follows from \cite[Theorem 2.1]{increments}.
	\end{proof}

\begin{lemma}\label{newCoefRegularity}
	Let $u\in L^\infty (\R^n).$ Suppose that for every $z\in\partial\Omega\cap B_1$ we have polynomials $P^{k,l}_{z}\in \mathbb{P}_{\lfloor \beta' - k\varepsilon_0 - l \rfloor}$ such that
	$$\left|u(x)-\sum_{k,l\geq 0}P^{k,l}_z(x-z) d^{s+k\varepsilon_0 + l}(x)\right|\leq C_0|x-z|^{\beta'+s},\quad x\in B_1(z),$$ 
	and 
	$$\left[ u-\sum_{k,l\geq 0}P^{k,l}_z d^{s+k\varepsilon_0 + l}\right]_{C^{\beta'+s}(B_r(x_0))}\leq C_0,\quad \text{if }d(x_0)=2r=|x_0-z|.$$
	Suppose that when $(k,l)\neq (0,0)$, the coefficient $(P^{k,l}_z)^{(\alpha)}$ are $C_z^{\beta' - k\varepsilon_0 - l-|\alpha|}(\partial\Omega\cap B_1)$, with 
	$$\left|\left|(P^{k,l}_z)^{(\alpha)}\right|\right|_{C_z^{\beta' - k\varepsilon_0 - l-|\alpha|}(\partial\Omega\cap B_1)}\leq C_1.$$
	
	Then the same holds true also for $P^{0,0}_z$ i.e.
	$$\left|\left|(P^{0,0}_z)^{(\alpha)}\right|\right|_{C^{\beta'-|\alpha|}_z(\partial\Omega\cap B_{1/2})}.\leq C$$
	The constant $C$ depends only on $n,s,\beta',C_0,C_1$.
\end{lemma}

\begin{proof}
	We start with proving that all the coefficients of $P^{0,0}_z$ are uniformly bounded for every $z\in \partial\Omega\cap B_{1/2}(0)$. In this direction, we stress that $$\bigcap_{z\in\partial\Omega\cap B_{1/2}(0)} B_1(z)\cap \Omega $$
	has non-empty interior, hence we can get a ball $B$ inside, with $d(\partial\Omega, B)\geq c>0$. Now we can bound 
	$$||P^{0,0}_z||_{L^\infty(B)}\leq \frac{1}{c^s}||P^{0,0}_zd^s||_{L^\infty(B)}\leq \frac{1}{c^s}||\tilde{u}-P^{0,0}_zd^s||_{L^\infty(B)}+\frac{1}{c^s}||\tilde{u}||_{L^\infty(B)} ,$$
	where we denoted $\tilde{u} =  u-\sum_{k\geq 1,l\geq 0}P^{k,l}_z d^{s+k\varepsilon_0 + l}$. But since both terms on the right-hand side above are bounded independently of $z$, we have 
	$$||P^{0,0}_z||_{L^\infty(B)}\leq C.$$
	Now \cite[Lemma A.10]{AR20} applies.
	
	We proceed towards finite differences. First, due to Lemma \ref{divisionLemma} we rewrite the assumptions into
	$$\left|\frac{u}{d^s} - \sum_{k,l\geq 0} P^{k,l}_z d^{k\varepsilon_0+l}\right|\leq C|x-z|^{\beta'},\quad x\in \mathcal{C}_z, \quad\text{and}$$
	$$\left[ \frac{u}{d^s} - \sum_{k,l\geq 0} P^{k,l}_z d^{k\varepsilon_0+l} \right]_{C^{\beta'}(B_r(x_0))} \leq C,$$
	which implies that on the cone $\mathcal{C}_z$ we have 
	$$\left|D^j \left(\frac{u}{d^s} - \sum_{k,l\geq 0} P^{k,l}_z d^{k\varepsilon_0+l}\right)\right|\leq C|x-z|^{\beta'-j}$$
	for $0\leq j<\beta'.$
	We proceed as in Lemma \ref{theUltimateRegularityLemma}. 
	
	Choose $|\gamma|=\lfloor\beta'\rfloor$. On the cone $\mathcal{C}_z$ we have
	$$\partial^\gamma\frac{u}{d^s} -\sum_{(k,l)\neq (0,0)}\sum_{\alpha\leq\gamma}\binom{\gamma}{\alpha}\partial^\alpha P^{k,l}_z \partial^{\gamma-\alpha}d^{k\varepsilon_0+l} - \partial^\gamma P^{0,0}_z = \eta_z,$$
	with $|\eta_z|\leq C|x-z|^{\beta'-|\gamma|}.$ 
	As in Lemma \ref{theUltimateRegularityLemma} we take the finite difference ot he above equation of some order big enough, and choose $x$ in a suitable intersection of cones. Then we can estimate
	$$|\Delta^N_h \eta_z| \leq C|h|^{\beta'-|\gamma|},$$
	$$\left|\Delta^N_h \frac{u}{d^s}\right| = 0,$$
	as well as
	$$ \left|\Delta^N_h \partial^\alpha P^{k,l}_z \partial^{\gamma-\alpha}d^{k\varepsilon_0+l}\right| = \left|\Delta^N_h \sum_{\alpha'\geq \alpha} c_{\alpha',z}(x-z)^{\alpha'-\alpha} \partial^{\gamma-\alpha}d^{k\varepsilon_0+l}\right| $$
	$$= \left| \sum_{\alpha'\geq \alpha}\sum_{\delta\leq\alpha'-\alpha}\binom{\alpha'-\alpha}{\delta} \Delta^N_h\left(c_{\alpha',z}(z-0)^\delta\right) \cdot(x-0)^{\alpha'-\alpha-\delta} \partial^{\gamma-\alpha}d^{k\varepsilon_0+l}\right|.$$
	The finite difference we estimate as in the proof of Lemma \ref{theUltimateRegularityLemma}, ($N'$ comes from the Leibnitz rule) to get
	$$\leq  C \sum_{\alpha',\delta}\sum_{N'}|h|^{(\beta'-k\varepsilon_0-l -|\alpha'|)\land N'}|h|^{|\delta|\lor ((N-N')\land \beta')} \cdot |h|^{|\alpha'-\alpha-\delta|+k\varepsilon_0+l-|\gamma-\alpha|}
	\leq C|h|^{N\land\beta' - |\gamma|}.$$
	Once this is established, we conclude that for $N>\beta'$,
	$$\left|\Delta^N_h \left(P^{0,0}_z\right)^{(\gamma)}\right|\leq C |h|^{\beta'-|\gamma|}.$$
	
	We proceed in the same way as in Lemma \ref{theUltimateRegularityLemma}: when dealing with lower order coefficients of $P^{0,0}_z$ we treat higher order ones as "the reminder". 
	
	The claim follows from \cite[Theorem 2.1]{increments}.
\end{proof}

	\begin{lemma}\label{reminderRegularity}
		Let $b>0$ and $p>-b$. Suppose $\phi\colon \Omega\to\R$ is a $C^{b+p\lor b}(\overline{\Omega})$ function, with $|\phi(x)|\leq C |x-z|^{b}$ for some boundary point $z$. 
		
		Then for any $a<b+p$, and $d(x_1) = 2r$ we have 
		$$\left[ \phi d^p\right]_{C^{a}(B_r(x_1))}\leq C \left(\frac{|x_1-z|}{r}\right)^b.$$
	\end{lemma}
	\begin{proof}
		First we deduce, that $|D^k\phi(x)|\leq C|x-z|^{b-k}$, for $k\leq b$. Now choose a multi-index $\gamma$ od order $\lfloor a\rfloor$, and compute
		$$\big|\partial^\gamma(\phi d^p)(x)-\partial^\gamma(\phi d^p)(y)\big|\leq \sum_{\alpha\leq\gamma}\binom{\gamma}{\alpha} |\partial^\alpha \phi(x)-\partial^\alpha \phi(y)| \cdot|\partial^{\gamma-\alpha}d^p(x)|+ $$
		$$+\sum_{\alpha\leq\gamma}\binom{\gamma}{\alpha} |\partial^\alpha \phi(y)| \cdot|\partial^{\gamma-\alpha}d^p(x)-\partial^{\gamma-\alpha}d^p(y)|$$
		$$\leq C\left(\sum_{|\alpha|\leq b-1} + \sum_{b-1<|\alpha|\leq b+p-1}+\sum_{b+p-1<|\alpha|\leq \gamma}\right) ||D^{|\alpha|+1}\phi||_{L^\infty}|x-y|||D^{|\gamma-\alpha|}d^p||_{L^\infty}+ $$
		$$ + C \left(\sum_{|\alpha|\leq b}+\sum_{b<|\alpha|\leq\gamma}\right)||D^{|\gamma-\alpha|+1}d^p||_{L^\infty}|x-y|||D^{|\alpha|}\phi||_{L^\infty}$$
		which after taking the worst term (second sum, $|\alpha|=0$) gives 
		$$\leq C|x_1-z|^br^{p-a}|x-y|^{\langle a\rangle}.$$
		But since $r^{p-a}\leq r^{-b},$ the claim is proven.
	\end{proof}

	\begin{lemma}\label{A10}
		Let $p>1$, and assume $f$ satisfies $|f(x)|\leq C |x-z|^{p}$ on the cone $\mathcal{C}_z$, as well as the interior regularity estimate
		$$\left[ f\right] _{C^p{(B_{r}(x_1))}}\leq C\left(\frac{|x_1-z|}{r}\right)^{p},$$ where  $d(x_1) = |x_1-z'|=2r$ ($x_1$ does not need to be in the cone). Then we have
		$$|\partial^\gamma f(x)|\leq C|x-z|^{p}d^{-|\gamma|}(x),\quad\quad x\in\Omega\cap B_1(z), 1 \leq |\gamma|\leq \lfloor p\rfloor. $$
	\end{lemma}
	\begin{proof}
		Let us stress, that inside the cone the claim is true due to the growth and the regularity estimates.  First let us prove the case $|\gamma|= \lfloor p\rfloor$. From the assumptions of $f$, we have $|\partial^\gamma f(x)|\leq C|x-z|^{p-|\gamma|}$ on the cone. Choose now $x\in \Omega$ outside the cone. Let $z'$ be the closest boundary point. Denote $x_i := z' + 2^{i}(x-z')$ and $r_i = |x_i-x_{i-1}|=2^{i-1}d(x)$. Let $N$ be such that $x_N$ is the first point inside the cone $\mathcal{C}_z$. Then we compute
		\begin{align*}
			|\partial^\gamma f(x)|&\leq \sum_{i=1}^N |\partial^\gamma f(x_i)-\partial^\gamma f(x_{i-1})| + |\partial^\gamma f(x_N)|\\
			&\leq \sum_1^N C\left(\frac{|x_i-z|}{r_i}\right)^p r_i^{\langle p\rangle} + C|x_N-z|^{p-|\gamma|}\\
			&\leq C|z-x|^pd^{-\lfloor p\rfloor}(x) \sum 2^{-\lfloor p\rfloor i}\leq C|x-z|^pd^{-\lfloor p\rfloor}(x).
		\end{align*}
		Since $x_i\not \in \mathcal{C}_z$, we have $|x_i-z|\leq C |x-z|$.
		
		For the lower order derivatives, we integrate the obtained bound along the line from $x$ to $x_N$. Choose $|\alpha|=\lfloor p\rfloor-1$ and compute
		\begin{align*}
			|\partial^{\alpha}f(x)|&\leq |\partial^{\alpha}f(x)-\partial^{\alpha}f(x_N)|+|\partial^{\alpha}f(x_N)|\leq |\int_{x}^{x_N}|D^{|\alpha|+1}f(t)|dt\\
			&\leq \int_{d(x)}^{d(x_N)}C|x_N-z|^p t^{-|\alpha|-1}dt\leq C|x-z|^p d(x)^{-|\alpha|}.
		\end{align*}
		Iterating this, we prove the claim.
	\end{proof}
	\begin{corollary}\label{generalisedGrowthLemma}
		In the same setting as above, we get the estimate 
		$$|\nabla f(x)|\leq C|x-z|^pd^{-1}(x).$$
	\end{corollary}
	\begin{remark}
		We can integrate once more time, to get the estimate for $f$ in the full neighbourhoods of $z$: 
		$$|f(x)|\leq C|x-z|^p\log d(x).$$
	\end{remark}

\end{document}